\documentclass[11pt,a4paper]{article}
\usepackage{silence}
\usepackage[utf8]{inputenc}
\usepackage[margin=1in]{geometry}
\usepackage{xcolor}
\usepackage{graphicx}
\usepackage{bm}
\usepackage{bbm}
\usepackage[labelfont=sc]{caption}
\usepackage{subcaption}
\usepackage{amsmath}
\usepackage{amsthm}
\usepackage{amssymb}
\usepackage{amsfonts}
\usepackage{mathtools}
\usepackage{xparse}
\usepackage{epstopdf}
\usepackage{hyperref}
\hypersetup{%
    linktocpage=true,
    colorlinks=true,
    citecolor=red,
    linkcolor=blue,
}
\usepackage[capitalise,nameinlink]{cleveref}
\usepackage{sectsty}
\allsectionsfont{\sffamily}
\usepackage{setspace}
\DeclareDocumentCommand\sobolev{m m o} {H^{#1}_{#3}(#2 \IfNoValueF{#3})}
\DeclareDocumentCommand\lp{m m o} {L^{#1}\IfNoValueF{#3}{_{#3}}\left(#2\right)}
\DeclareDocumentCommand\cont{o m o} {C\IfNoValueF{#1}{^{#1}}(#2\IfNoValueF{#3}{;#3})}
\DeclareDocumentCommand\contc{o m o} {C_c\IfNoValueF{#1}{^{#1}}(#2\IfNoValueF{#3}{;#3})}
\DeclareDocumentCommand\norm{s m o} {\IfBooleanTF{#1}{\|#2\|}{\left\|#2\right\|}\IfNoValueF{#3}{_{#3}}}
\DeclareDocumentCommand\seminorm{s m o} {\IfBooleanTF{#1}{\|#2\|}{\left\|#2\right\|}\IfNoValueF{#3}{_{#3}}}
\DeclareDocumentCommand\ip{s m m o} {\IfBooleanTF{#1}{( #2,#3 )}{\left( #2,#3 \right)}\IfNoValueF{#4}{_{#4}}}
\DeclareDocumentCommand\eip{s m m o} {\IfBooleanTF{#1}{\langle #2,#3 \rangle}{\left\langle #2,#3 \right\rangle}\IfNoValueF{#4}{_{#4}}}
\DeclareDocumentCommand\abs{s m o} {\IfBooleanTF{#1}{|#2|}{\left|#2\right|}\IfNoValueF{#3}{_{#3}}}
\DeclareDocumentCommand\vecnorm{s m o} {\IfBooleanTF{#1}{|#2|}{\left|#2\right|}\IfNoValueF{#3}{_{#3}}}
\DeclareDocumentCommand\wass{s o m m} {W_{\IfNoValueF{2}{#2}}\IfBooleanTF{#1}{(#3, #4)}{\left(#3, #4\right)}}
\DeclareDocumentCommand\bigo{s o m} {\mathcal O\IfNoValueF{#2}{_{#2}}\IfBooleanTF{#1}{(#3)}{\left(#3\right)}}

\DeclareMathOperator{\cond}{cond}

\DeclareMathOperator{\e}{e}

\DeclareMathOperator{\sign}{sign}

\DeclareMathOperator*{\trace}{tr}
\DeclareMathOperator*{\argmin}{arg\,min}

\renewcommand{\t}{\mathsf T}
\newcommand{\hessian}{\operatorname{D}^2}
\newcommand{\proba}{\mathbf P}

\newcommand{\laplacian}{\Delta}

\newcommand{\dummy}{\mathord{\color{black!33}\bullet}}%
\newcommand{\expect}{\mathbf{E}}

\newcommand{\mat}[1]{#1}
\newcommand{\nat}{\mathbf N}

\newcommand{\real}{\mathbf R}

\newcommand{\vect}[1]{\boldsymbol{\mathbf #1}}
\newcommand{\grad}{\nabla}
\renewcommand{\d}{\mathrm d}

\newcommand{\pos}{\mathcal S_{++}}

\makeatletter
\DeclareDocumentCommand \derivative{s m o m}{%
    \def\@der{\IfBooleanTF{#1}{\mathrm{d}}{\partial}}
    \def\@default{%
        \mathchoice{%
                \frac{%
                    \@der\ifnum\pdfstrcmp{#2}{1}=0\else^{#2}\fi {\IfNoValueTF{#3}{}{#3}}
                }{%
                    \@for\@token:={#4}\do{\@der \@token}
                }
            } {%
                \@for\@token:={#4}\do{\@der_{\@token}\ifnum\pdfstrcmp{#2}{1}=0\else^{#2}\fi} \IfNoValueTF{#3}{}{#3}
            } {} {}
    }
    \IfBooleanTF{#1}{\IfNoValueTF{#3}{\@default}{%
                #3%
                \ifnum\pdfstrcmp{#2}{1}=0'\else%
                \ifnum\pdfstrcmp{#2}{2}=0''\else%
                \ifnum\pdfstrcmp{#2}{3}=0'''\else%
                \ifnum\pdfstrcmp{#2}{4}=0^{(iv)}\else^{(#2)}\fi\fi\fi\fi
            }
        }{\@default}
}
\makeatother

\definecolor{darkred}{rgb}{0.5,0,0}
\definecolor{darkgreen}{rgb}{0,0.5,0}
\definecolor{darkblue}{rgb}{0,0,.5}

\theoremstyle{plain}

\newtheorem{assumption}{Assumption}
\newtheorem{theorem}{Theorem}[section]
\newtheorem{lemma}[theorem]{Lemma}
\newtheorem{corollary}[theorem]{Corollary}
\newtheorem{proposition}[theorem]{Proposition}

\newtheorem{remark}{Remark}[section]
\numberwithin{equation}{section}
\crefname{lemma}{Lemma}{Lemmas}
\crefname{remark}{Remark}{Remarks}
\crefname{assumption}{Assumption}{Assumptions}
\crefname{proposition}{Proposition}{Propositions}
\crefname{section}{Section}{Sections}
\crefname{subsection}{Subsection}{Subsections}
\crefname{equation}{}{}
\Crefname{equation}{Equation}{Equations}

\usepackage{tikz}
\usepackage{tikz-cd}
\usepackage{pgfplotstable}
\pgfplotsset{compat=1.14}
\usetikzlibrary{patterns}
\usetikzlibrary{calc}
\usetikzlibrary{angles}
\usetikzlibrary{quotes}
\usetikzlibrary{external}

\newcommand{\email}[1]{\href{#1}{#1}}
\title{Consensus Based Sampling}
\author{%
    J. A. Carrillo\thanks{%
        Mathematical Institute, University of Oxford, Oxford OX2 6GG, UK (\email{carrillo@maths.ox.ac.uk}). Corresponding author.
    },
    F. Hoffmann\thanks{%
        Hausdorff Center for Mathematics, Rheinische Friedrich-Wilhelms-Universität, Bonn 53115, Germany (\email{franca.hoffmann@hcm.uni-bonn.de}).
    },
    A. M. Stuart\thanks{%
        Department of Computing and Mathematical Sciences, Caltech, Pasadena, CA 91125, USA (\email{astuart@caltech.edu}).
    },
    U. Vaes\thanks{%
        MATHERIALS team, Inria Paris, Paris 75012, France (\email{urbain.vaes@inria.fr}).
    }
}

\newcommand{\Lhess}{L}
\newcommand{\lhess}{\ell}
\newcommand{\Uhess}{U}
\newcommand{\uhess}{u}

\newcommand{\normal}{\mathsf{N}}
\newcommand{\cN}{\normal}

\renewcommand{\leq}{\leqslant}
\renewcommand{\geq}{\geqslant}
\renewcommand{\le}{\leqslant}
\renewcommand{\ge}{\geqslant}

\begin{document}
\maketitle

\begin{abstract}
    We propose a novel method for sampling and optimization tasks based on a stochastic interacting particle system. We explain how this method can be used for the following two goals:
    (i) generating approximate samples from a given target distribution;
    (ii) optimizing a given objective function. The approach is derivative-free and affine invariant, and is therefore well-suited for solving inverse problems defined by complex
    forward models: (i) allows generation of samples from the Bayesian posterior and (ii) allows determination of the maximum a posteriori estimator.
    We investigate the properties of the proposed family of methods in terms of various parameter choices, both analytically and by means of numerical simulations. The analysis and numerical
    simulation establish that the method has potential for
    general purpose optimization tasks over Euclidean space;
    contraction properties of the algorithm are established under suitable conditions, and computational experiments demonstrate
    wide basins of attraction for various specific problems.
    The analysis and experiments
    also demonstrate the potential for the sampling methodology
    in regimes in which the target distribution is unimodal
    and close to Gaussian; indeed we prove that the method recovers
    a Laplace approximation to the measure in certain
    parametric regimes and provide numerical evidence that this
    Laplace approximation attracts a large set of initial conditions in a number of examples.
\end{abstract}

{\bf Keywords.- }stochastic interacting particle systems, sampling, optimization

\section{Introduction}%

\subsection{Background}%
We consider the inverse problem of finding $\theta$ from $y$ where
\begin{equation}
    \label{eq:inverse_problem}
    y = G(\theta) + \eta.
\end{equation}
Here $y \in \real^K$ is the \emph{observation},
$\theta \in \real^d$ is the \emph{unknown parameter},
$G:\real^d\to\real^K$ is the \emph{forward model} and $\eta$ is the \emph{observational noise}.
We adopt the Bayesian approach to inversion~\cite{kaipio2006statistical} and assume that the parameter and the noise are independent and normally distributed:
$\theta \sim \normal(0, \mat \Sigma)$ and $\eta \sim \normal(0, \mat \Gamma)$.
By~\eqref{eq:inverse_problem} and Bayes' formula,
the posterior density (i.e., the conditional probability density function of $\theta$ given $y$) equals
\begin{equation}
    \label{eq:post}
    \rho(\theta) = \frac{\exp\bigl(-f(\theta) \bigr)}{\int_{\real^{d}}\exp\bigl(-f(\theta) \bigr) \, \d \theta},
\end{equation}
where
\begin{equation}
    \label{eq:added2}
    f(\theta) := \Phi(\theta; y) + \frac{1}{2} \vecnorm{\theta}[\mat \Sigma]^2,
    \qquad
    % \text{with }
    \Phi(\theta; y) = \frac{1}{2} \vecnorm{y - G(\theta)}[\mat \Gamma]^2.
\end{equation}
In the foregoing and in what follows, we adopt
the following notation: for a positive definite matrix $A$,
\[
    \eip*{\dummy}{\dummy}[A] = \eip*{\dummy}{A^{-1}\dummy}, \qquad \vecnorm{\dummy}[A]^2 = \eip{\dummy}{\dummy}[A].
\]
We also define the matrix norm $\norm{B}_{A}=\norm{A^{-1/2}BA^{-1/2}}$
(noting that this is not the induced matrix norm from
vector norm $|\dummy|_{A}$).

Solving inverse problems in the Bayesian framework can be
prohibitively expensive because of the need to characterize
an entire probability distribution. One approach to this
is simply to seek the point of maximum posterior probability,
the MAP point \cite{kaipio2006statistical,dashti2013map}, defined by
\begin{equation}
    \label{eq:MAP}
    \theta^*={\rm argmin}_{\theta}\,\,f(\theta).
\end{equation}
However, this essentially reduces the solution of the inverse problem
to a classical optimization approach \cite{engl1996regularization} and
fails to capture uncertainty. A compromise between a fully Bayesian
approach and the classical optimization approach is to seek
a Gaussian approximation of the measure
\cite{lu2017gaussian}. By the Bernstein--von Mises theorem (and its extensions) \cite{van2000asymptotic},
the posterior is expected to be well approximated by a
Gaussian density in the large data limit, if the parameter is
identifiable in the infinite data setting; a Gaussian approximation
is also expected to be good if the forward map is close
to linear. For these reasons, use of the
Laplace method \cite{shun1995laplace}
to obtain a Gaussian approximation of the
posterior density is often viewed
as a useful approach in many application domains.

Many inverse problems arising in applications are defined by complex forward models $G$, often available only as a black box, and in particular adjoints and derivatives may not be readily available.
Consensus-based approaches are proving to be interesting and viable derivative-free techniques for optimization \cite{pinnau2017consensus,carrillo2018analytical,carrillo2019consensus}. The focus of this paper is on developing
consensus-based sampling of the posterior distribution for Bayesian
inverse problems and,
in particular, on the study of such methods in the context of
Gaussian approximation of the posterior.

The computational methodology we introduce applies to arbitary measures with negative log density $f$, and is not restricted
to the choice in \eqref{eq:added2} resulting from the
inverse problem \eqref{eq:inverse_problem}.
Some of our analysis, however, is specific to the
inverse problem in the case where $G$ is linear.
The proposed methodology is potentially useful for
the solution of complex problems for which the evaluation
of $f$ or $G$ is expensive, and derivatives of $f$
and $G$ are not available, or noisy and not useable.
In this sense the proposed methodology is competitive
with state-of-the-art ensemble Kalman methods for
inverse problems, which are also of particular value
for derivative-free sampling when $G$ is expensive to
evaluate. The fact that the analysis of the accuracy of the
proposed sampling method is confined to unimodal distributions
which are close to Gaussian is also a limitation of
ensemble Kalman methods. Our work thus provides impetus
for further innovation in the analysis and design of particle-based, derivative-free sampling methods.

\subsection{Literature Review}%
\label{sub:literature_review}

Systematic procedures to sample probability measures have their
roots in statistical physics and the 1953 paper of
Metropolis et al \cite{metropolis1953equation}. In 1970 Hastings
recognized this work as a special case of what is now known as
the Metropolis-Hastings methodology \cite{hastings1970monte}. These
methods in turn may be seen as part of the broader
Markov Chain Monte Carlo (MCMC) approach to sampling \cite{brooks2011handbook}. In 2006, sequential Monte Carlo (SMC)
methods, based on creating a homotopy deforming the initial (simple
to sample) measure into the desired target measure, were introduced \cite{del2006sequential}; in practice
these methods work best when entwined with MCMC kernels. These
SMC methods introduce the idea of using the evolution of a
system of interacting particles
to approximate the desired target measure; the large
particle limit of this evolution captures the homotopy from the
initial measure to the target measure.
In a parallel development, the mathematical physics community has
developed a large body of understanding of interacting particle systems, and their mean field limits, initially
primarily for models on a countable state space \cite{liggett2012interacting,swart2017course} and more recently
for models in uncountable state space \cite{Snitzman,carrillo2010particle,BCC,BCC2,Jabin-Wang}.
Studying interactions between sampling, collective dynamics of
particles and mean field limits holds considerable promise as
a direction for finding improved sampling algorithms for specific
classes of problems and is an active area of research
\cite{reich2011dynamical,yang2013feedback,bunch2016approximations,van2019particle}.

The focus of this work is on sampling measure \eqref{eq:post},
or optimizing objective function \eqref{eq:MAP}, by means of algorithms which only involve black box
evaluation of $G$.
While some MCMC and SMC methods are of this type, the Metropolis
algorithm being a primary example, the use of collective dynamics of
particles opens the door to a wider range of methods to solve
inverse problems in this setting.
There are two primary classes
of methods emerging in this context: those arising from consensus
forming dynamics \cite{pinnau2017consensus} and those arising from ensemble Kalman methods \cite{reich2015probabilistic}.

Iterative ensemble Kalman methods for inverse problems were introduced
in \cite{Chen2012,emerick2013investigation}. Similar ideas are
also implicit in the work of
Reich \cite{reich2011dynamical} who studies
state estimation sequential data assimilation, rather than the inverse problem; however, what is termed
the ``analysis'' step in sequential data assimilation corresponds to
solving a Bayesian inverse problem. These iterative ensemble Kalman methods are similar to SMC
in that they seek to map the prior to the posterior in
finite continuous time or in a finite number of steps.
Reich also introduced continuous time analysis of ensemble Kalman
methods for state estimation in \cite{bergemann2010localization,bergemann2012ensemble}, naming
the resulting algorithm the ensemble Kalman Bucy filter (EnKBF);
the ensemble Kalman approach to inverse problems introduced
in \cite{Chen2012,emerick2013investigation} may be studied
using the EnKBF leading to a clear link with SMC methods
in continuous time. An alternative
Kalman methodology (ensemble Kalman inversion -- the EKI)
for the optimization approach to the
inverse problem, which involves iteration to
infinity, was introduced and studied in \cite{iglesias2013ensemble,iglesias2016regularizing}
in discrete time and in \cite{schillings2017analysis,schillings2018convergence}
in continuous time; the idea of using ensemble methods
for optimization rather than sampling was anticipated
in \cite{reich2011dynamical}.
The ensemble based optimization approach was generalized to approximate sampling of the
Bayesian posterior solution to the inverse problem in
\cite{garbuno2020interacting} (the ensemble Kalman
sampler -- the EKS), and studied further in \cite{carrillo2019wasserstein,garbuno2020affine,nusken2019note}.

The idea of consensus based optimization may be seen as a variation of particle
swarm optimization methods \cite{dorigo2005ant,kennedy2010particle} which are
themselves related to Cucker-Smale dynamics for collective behavior
and opinion formation \cite{To06,CS07,HL,BCC,carrillo2010particle,MT14}. These
dynamical systems model the
tendency of the constituent particles to align (consensus in velocity) or to concentrate in
certain variables modelling averaged quantities (consensus in position or
opinion), and they have been extensively studied in terms of long time
asymptotics leading to consensus \cite{CFRT10,MT14}. Consensus Based
Optimization (CBO) was introduced in \cite{pinnau2017consensus} based on the
following simple idea: particles are explorers in the landscape of the graph of
the function $f(\theta)$ to be minimized, they are able to exchange information
instantaneously, and they redirect their movement towards the location of a
consensus position in parameter space that is a weighted average of the
explorer's parameter values relative to the Gibbs measure associated to the
function $f$, $\frac1Z e^{-f(\theta)}$. Noise is introduced for suitable
exploration in parameter space but the strength of the noise is reduced
according to the distance to the consensus parameter values. These effects lead
to concentration in parameter space at the global minimum of the function, as
proven in \cite{carrillo2018analytical} for the mean-field limit PDE and in
\cite{ha2020convergence} for the particle system under certain conditions on
$f$ and the parameters of the model. The original CBO method has been recently
improved so as to be efficient for high-dimensional optimization problems \cite{carrillo2019consensus},
such as those arising in machine learning,
by adding coordinate-wise noise terms and introducing ideas from random batch methods \cite{jin2018random} for computing
stochastic particle systems efficiently. Furthermore, these ideas have been recently used
to solve constraint problems on the sphere \cite{fornasier2020consensus,fornasier2020consensusbased,fornasier2021consensusbased}.
There are other approaches to the use of interacting particles system in
optimization, including the use of individual gradient dynamics coupled through
a graph Laplacian
\cite{borovykh2021stochastic,borovykh2020stochastic,borovykh2020interact,kantas2019sharp}.

The development of the EKI into the EKS suggests a parallel
development of CBO into a sampling methodology.  In this paper we
pursue this idea and develop Consensus Based Sampling (CBS).
A key property of the EKS is that it is affine invariant \cite{goodman2010ensemble}
as shown in the paper \cite{garbuno2020affine} where the
Affine Invariant Interacting Langevin Dynamics (ALDI) algorithm is introduced; relatedly, in the mean
field limit, the rate of
convergence to the posterior is the same for
all Gaussian posterior distributions
\cite{garbuno2020interacting}. We will show identical properties for
the CBS algorithm. Our focus is on unimodal distributions and
obtaining Gaussian approximations to the target distribution. We note, however, that there are recent forays into the use of ensemble Kalman methods for the sampling of multimodal distributions \cite{reich2021fokker,birthdeath}.
Furthermore there is also recent work extending ensemble Kalman methods to inverse
problems beyond the setting of additive Gaussian noise; more
complex loss functions, such as cross-entropy and
those arising in logistic regression \cite{Kovachki_2019,pidstrigach2021affine} are considered. And finally, recent work
shows that ensemble methods automatically smooth noisy
likelihood functions, essentially denoising rough energy
landscapes~\cite{duncan2021ensemble}.
Similar developments for the CBS methodology proposed here would also be of interest.
Like the ensemble Kalman sampler, the CBS approach is only exact for Gaussian problems
and in the mean field limit. However recently developed methods
based on multiscale stochastic dynamics provide
a refineable methodology for sampling
from non-Gaussian distributions ~\cite{pavliotis21derivative}; methods such as CBS or EKS
may be used to precondition these
multiscale stochastic dynamics algorithms, making them more
efficient.
Alternatively, the CBS method may be used in the calibration step
employed within the calibrate-emulate-sample methodology
introduced in \cite{cleary2020calibrate}.
Thus, the methods developed in this paper potentially form  an
important component in an efficient and rigorously justifiable
approach to solving Bayesian inverse problems.

\subsection{Our Contributions}%
\label{sub:our_contributions}

We introduce CBS as a method to approximate probability distributions
of the form \eqref{eq:post}, or to find the MAP estimator \eqref{eq:MAP}.
The method requires $G$ only as a black-box (it is derivative-free) and hence is
of potential use for large-scale inverse problems.
We study the proposed algorithm in settings where the posterior is Gaussian or close to Gaussian. We reemphasize that
the computational methodology does not require the specific choice of $f$ in \eqref{eq:added2}, it applies to arbitary measures with negative log density $f$,
up to an additive constant; however some of our analysis exploits the specific form
in \eqref{eq:added2} in the case where $G$ is linear.
We show the following:

\begin{itemize}
\item in the case of linear $G$, and in the mean field limit, parameters
can be chosen in the algorithm so that, if initiated at a Gaussian,
successive iterates remain Gaussian and converge to the Gaussian posterior
\eqref{eq:post};
\item in the case of linear $G$, and in the mean field limit, parameters
can be chosen in the algorithm so that, if initiated at a Gaussian,
successive iterates remain Gaussian and converge to a Dirac located at the MAP point $\theta_*$ given by \eqref{eq:MAP};

\item the CBS method is affine invariant and, in the case of linear $G$ and in the mean field limit,
     converges at the same rate across all linear inverse problems defined by \eqref{eq:post};
     for linear~$G$, we obtain sharp convergence rates that are explicit in terms of all parameters of the method;

    \item in the case of nonlinear $G$, and in the mean field limit, parameters can be chosen in the algorithm so that it has a steady
    state solution which is Gaussian, close to the Laplace approximation of the posterior \eqref{eq:post} and the algorithm is a local contraction mapping in the neighbourhood
    of the steady state; we make explicit the dependence of this approximation, and its rate of attraction, on the parameters of the method;
\item we present numerical results illustrating the foregoing theory and, more generally, demonstrating the viability of the CBS scheme for sampling posterior distributions and for finding MAP estimators.
\end{itemize}

The results are in arbitrary dimension $d$, with the
exception of the results concerning the Laplace approximation
which are restricted to $d=1$. There are no intrinsic
barriers to extending the Laplace approximation
results to arbitrary dimension, but doing so will be technically
involved and would lose the focus of the paper.

In \cref{sec:presentation_of_the_method} we introduce the method,
including its continuous time limit, and mean field limits in both discrete
and continuous time; we establish its properties in the Gaussian setting.
\Cref{sec:results} contains analysis of the method beyond the
Gaussian setting, deriving conditions for convergence to an approximation
of the MAP estimator when in optimization mode, and for convergence to the
Laplace approximation of the target measure when in sampling mode.
In \cref{sec:numerical_experiments} we provide
the numerical experiments.
Proofs of most of the theoretical results in \cref{sec:presentation_of_the_method,sec:results} are presented in \cref{sec:proof_of_the_main_results}.

\section{Presentation of the Method}%
\label{sec:presentation_of_the_method}

We propose a novel method for sampling and optimization tasks based on a system of interacting particles.
Our goals are the following:
\begin{enumerate}
    \item[(1)] Sampling: to generate approximate samples from the posterior distribution~\eqref{eq:post};
        this allows to understand the distribution of parameters taking into account both model \eqref{eq:inverse_problem} and the available data $y$.
    \item[(2)] Optimization: to find the minimizer of $f(\dummy)$, which corresponds to the MAP point~\eqref{eq:MAP},
        the most likely parameter $\theta$ given the data $y$ and the model relating them.
\end{enumerate}
In order to introduce the approach,
we start by defining the mean-field limits of the algorithms,
in discrete and continuous time; later we explain how particle approximations of the mean-field limit lead to implementable algorithms.
We will be interested in the following McKean difference equation:
given parameters $\lambda > 0$, $\beta>0$ and $\alpha \in [0, 1)$,
\begin{align}
    \label{eq:mean_field_sdes}
    \left\{
    \begin{aligned}
    &\theta_{n+1} = \mathcal M_{\beta}(\rho_n) + \alpha \bigl(\theta_n - \mathcal M_{\beta}(\rho_n)\bigr)
    +  \sqrt{(1 - \alpha^2) \, \lambda^{-1} \mathcal C_{\beta} (\rho_n)} \, \vect \xi_n, \\
    &\rho_n = {\rm Law}(\theta_n).
    \end{aligned}
    \right.
\end{align}
where $\vect \xi_n$, for  $n \in \{0, 1, \dotsc\}$ are independent $\normal(\vect 0, I_d)$ random variables, and
$\mathcal M_{\beta}, \mathcal C_{\beta}$ denote respectively the mean and variance for a suitable reweighting of measures:
\begin{subequations}
\label{eq:mom}
    \begin{align}
 &{\mathcal M}_{\beta}: \rho \mapsto \mathcal M( L_{\beta} \rho)\,,\quad
 {\mathcal C}_{\beta}: \rho \mapsto \mathcal C( L_{\beta} \rho)\,,\quad
 L_{\beta}: \rho \mapsto \frac{\rho \e^{-\beta f}}{\int \rho \e^{-\beta f}}\,,\\
 &\mathcal M(\mu)=\int \theta\mu(d\theta)\,,\quad
 \mathcal{C}(\mu)=\int \bigl(\theta-\mathcal M(\mu)\bigr) \otimes\bigl(\theta-\mathcal M(\mu)\bigr) \mu(d\theta)\,.
\end{align}
\end{subequations}
Letting $\alpha=\exp(-\Delta t)$ and viewing $\theta_n$ as a discrete time approximation of a continuous time process $\theta(t)$ at time $t=n\Delta t$, we find
that the $\Delta t \to 0$ continuous-time limit associated with these dynamics is the following McKean SDE:
\begin{align}
    \label{eq:mean_field_sdes2}
    \left\{
        \begin{aligned}
        &\d \theta_t = - \bigl(\theta_t - \mathcal M_{\beta}(\rho_{t})\bigr) \, \d t + \sqrt{2 \lambda^{-1} \mathcal C_{\beta}(\rho_{t})} \, \d \vect{W}_t,\\
        &\rho_t = {\rm Law}(\theta_t).
        \end{aligned}
    \right.
\end{align}
where $\vect{W}_t$ denotes a standard Brownian motions in $\real^d$.
We refer to the two familes of methods as
 \emph{Consensus Based Sampling} (CBS) methods, parameterized by $\alpha,\beta$ with the ranges $\alpha\in[0,1)$ corresponding to \eqref{eq:mean_field_sdes} and $\alpha=1$ corresponding to \eqref{eq:mean_field_sdes2}.
Recall that $\beta>0$. We will
 focus on two choices of $\lambda$: (i) the choice $\lambda=1$, when the method is used to minimize $f(\dummy)$, which will be referred to as CBS-O($\alpha$,$\beta$); and (ii)  $\lambda=(1+\beta)^{-1}$ when the method is
 used for sampling the target distribution $e^{-f(\dummy)}$, which will be referred to as CBS($\alpha$,$\beta$).

In \cref{sub:notation}, we introduce the notation used throughout the paper.
In \cref{sec:mot} we give motivation for the mean field stochastic dynamical systems \eqref{eq:mean_field_sdes} and \eqref{eq:mean_field_sdes2}.
In \cref{sec:properties} we describe key properties of the mean field models,
and in \cref{sub:convergence_to_equilibrium},
we establish convergence to equilibrium for~\eqref{eq:mean_field_sdes} and~\eqref{eq:mean_field_sdes2} in the setting
where the forward model $G$ is linear and the law of the initial condition is Gaussian.
\Cref{sec:notation} introduces particle approximations to the mean field limit.

\subsection{Notation}
\label{sub:notation}
In what follows, we denote by $g(\dummy;\vect m, C)$ the density of the Gaussian random variable~$\normal(\vect m, C)$:
\begin{equation}
\label{eq:gnote}
    g(\theta;\vect m, C)
    =\frac{1}{\sqrt{(2\pi)^d \det(C)}}\exp\left(-\frac{1}{2} \vecnorm{\theta-\vect m}[C]^2\right)\,.
\end{equation}
We also use the short-hand notation
\begin{align}
    \label{eq:mbeta_and_cbete}%
    \vect m_{\beta}(\vect m, \mat C) &:=\mathcal M_{\beta}\bigl(g(\dummy; \vect m, \mat C)\bigr), \qquad
    \mat C_{\beta}(\vect m, \mat C)  :=\mathcal  C_{\beta}\bigl(g(\dummy; \vect m, \mat C)\bigr).
\end{align}
More generally, we frequently denote $\vect m_n =\mathcal{M}(\rho_n)$ and $\mat C_n = \mathcal{C}(\rho_n)$ for the standard mean and covariance calculated with respect to a probability measure $\rho_n$.
For a matrix $A\in\real^{d\times d}$,
we denote by $\norm{A}$ the operator norm induced by the Euclidean vector norm,
and by $\norm{A}_{\rm F}$ the Frobenius norm\footnote{%
    The Frobenius norm on matrices should not to be confused with the norm $|\vect u|_A:= \langle \vect u, A^{-1}\vect u \rangle^{\frac12}$ on vectors defined previously.
}. Sometimes, we will make use of the shorthand notation $\norm{A}_B:=\norm{B^{-1/2}AB^{-1/2}}$ for a given invertible matrix $B\in\real^{d\times d}$.
We let $\nat :=\{0,1,2,3,\dots\}$
and $\nat_{> 0} :=\{1,2,3,\dots\}$,
and we denote by $\pos^d$ the set of symmetric strictly positive definite matrices in $\real^{d \times d}$.
For symmetric matrices $X$ and $Y$,
the notation $X \succcurlyeq Y$ (resp. $X \preccurlyeq Y$) means that
$X - Y$ is positive semidefinite (resp. negative semidefinite).

\subsection{Motivation}
\label{sec:mot}

The mean-field model \eqref{eq:mean_field_sdes} contains
a number of tuneable parameters. In this section we give intuition about the role of these parameters in effecting approximate sampling or optimization for the inverse problem defined by \eqref{eq:inverse_problem}. We motivate sampling primarily through the discrete time mean field model and optimization primarily through the continuous time mean field model. However both discrete and continuous time models apply to optimization and to sampling.
In practice, the mean field SDEs in this subsection can be made into algorithms by
invoking finite particle approximations,
as described in Subsection \ref{sec:notation}.

\subsubsection{Sampling.}\label{sec:sampling}

Let $G(\dummy)=G\dummy$ be a linear map so that the posterior distribution given by \eqref{eq:post} is
Gaussian, and denote this Gaussian by $\cN(\vect a, \mat A)$.
The mean $\vect a$ and covariance $\mat A$ may be identified
by completing the square in \eqref{eq:post}: $f$ is of the form~$\frac{1}{2} \vecnorm{\theta - \vect a}_{\mat A}^2$,

To motivate the algorithms that are the object of study in this paper we describe parameter choices for which the iteration
\eqref{eq:mean_field_sdes} has equilibrium distribution given by the Gaussian $\cN(\vect a,\mat A)$.
For any choice of forward model $G$,
it can be shown that the evolution of the first and second moments is given by
\begin{subequations}
\label{eq:equations_moments_discrete}%
\begin{align}
    \label{eq:first_moment_discrete}%
    \mathcal M(\rho_{n+1}) &= \alpha \mathcal M(\rho_n) + (1 - \alpha) \mathcal M_{\beta}(\rho_n), \\
    \label{eq:second_moment_discrete}%
    \mathcal C(\rho_{n+1})  &= \alpha^2 \mathcal C(\rho_n) + \lambda^{-1} (1 - \alpha^2) \mathcal C_{\beta}(\rho_n).
\end{align}
\end{subequations}
From these identities it is clear that any fixed point of the
mean and covariance is independent of
$\alpha.$
Further, when the initial distribution $\rho_0$ is Gaussian the systems of equations
\eqref{eq:mean_field_sdes} for $\alpha\in [0,1)$ map Gaussians
into Gaussians.
Computing the relationship between the mean and covariance of the
Gaussian $\rho$ and
the mean and covariance of the
Gaussian $L_{\beta}\rho$
gives
\begin{subequations}
\label{eq:explicit_moments}
\begin{align}
    \label{eq:explicit_first_moment}
   \vect m_{\beta}(\vect m, \mat C)
    &= \left(\mat C^{-1} + \beta \mat A^{-1}\right)^{-1} \left(\beta \mat A^{-1} \vect a + \mat C^{-1} \vect m \right), \\
    \label{eq:explicit_second_moment}
    \mat C_{\beta}(\vect m, \mat C)
    &= \left(\mat C^{-1} + \beta \mat A^{-1}\right)^{-1}\,.
\end{align}
\end{subequations}
Therefore, the mean and covariance of a non-degenerate Gaussian steady state $g(\dummy; \vect m_{\infty}, \mat C_{\infty})$ for \eqref{eq:mean_field_sdes} satisfes
\begin{align*}
    \vect m_{\infty} &= \left(\mat C_\infty^{-1} + \beta \mat A^{-1}\right)^{-1} \left(\beta \mat A^{-1} \vect a + \mat C_\infty^{-1} \vect m_{\infty} \right), \\
    \mat C_\infty &= \lambda^{-1} \left(\mat C_{\infty}^{-1} + \beta \mat A^{-1}\right)^{-1}.
\end{align*}
This has solution
\begin{align*}
    \vect m_{\infty} = \vect a, \qquad
    \mat C_{\infty} = \frac{1 - \lambda}{\lambda \beta} \, \mat A.
\end{align*}
Choosing $\lambda^{-1}=1+\beta$ delivers
a steady state equal to the posterior
distribution. This motivates our choice of $\lambda$ in the sampling case. Furthermore, choosing
$\lambda=1$ is seen to be natural in the
optimization setting: the fixed point of the iteration is then a Dirac at the MAP estimator~$\vect a.$
We will demonstrate that these two distinguished choices of
$\lambda$ work well for sampling and optimization, beyond the setting of a Gaussian posterior $\cN(\vect a, \mat A)$.

\begin{remark}
    [Enlarging the Choice of Parameters.]
The mean-field dynamics \eqref{eq:mean_field_sdes} can be generalized to the form
\begin{align}
    \label{eq:general_evolution}
    \theta_{n+1} = p_1 \theta_n + p_2 \mathcal M(\rho_n) + p_3 \mathcal M_{\beta}(\rho_n) + \sqrt{p_4 \mathcal C (\rho_n) + p_5 \mathcal C_{\beta}(\rho_n)} \, \vect \xi_n,
    \qquad \rho_n = {\rm Law}(\theta_n)\,,
\end{align}
where $(\vect \xi_n)_{n = 0, 1,\dotsc}$ are independent $\normal(\vect 0, I_d)$ random variables.
Given $\beta$, one can ask the following question:
for what values of the parameters $(p_1, p_2, p_3, p_4, p_5)$ does
the dynamics~\eqref{eq:general_evolution} admit the Gaussian $\cN(\vect a,\mat A)$ as an equilibrium distribution?
A calculation analogous to that above shows that
$\normal(\vect a, A)$ is a steady state of~\eqref{eq:general_evolution} if and only if
\begin{subequations}
    \label{eq:genpar}
    \begin{align}
    &p_1 + p_2 + p_3 = 1, \\
    &p_1^2 + p_4 + p_5 (1 + \beta)^{-1} = 1.
    \end{align}
\end{subequations}
Note that these constraints do not guarantee that $\normal(\vect a, A)$ is the only steady state,
and in fact, if $p_1=1$ and $p_2=p_3=p_4=p_5=0$, then any distribution is a steady state.
In this paper, we study only the dynamics~\eqref{eq:mean_field_sdes},
which corresponds to the special case where $p_2=p_4=0$ and
$p_1=\alpha$, $p_3=1-\alpha$ and $p_5=\lambda^{-1}(1-\alpha^2)$,
but it is potentially useful to exploit this wider class of
mean-field models.
\end{remark}

\subsubsection{Optimization.}
We now discuss
the algorithm in optimization mode, through the lens of the continuous time limit.
Another starting point triggering the research in this paper
is the use of systems of interacting particles  for minimizing a target function $f(\theta)$. The papers  \cite{pinnau2017consensus,carrillo2018analytical}
introduce the CBO technique for achieving
this aim by means of particle appoximations of the stochastic dynamical system
\begin{align}
 \dot\theta &=  -(\theta - \bar \theta) + \sigma|\theta - \bar\theta| \,\dot{\vect W}^{(i)},  \qquad
 \bar \theta = \mathcal M_{\beta}(\rho_{t}),
\label{eq:particle_a}
\end{align}
where $\vect W$ is a standard Brownian motion in $\real^d$, $\sigma>0$ is the noise strength and $\rho_t$ is the law of $\theta.$ The  idea behind the CBO method is to think about realizations of $\theta$ as explorers, in the landscape of the  function $f(\theta)$, which can continuously exchange the evaluation of the function~$f$ at their position $\theta$, through $ \mathcal M_{\beta}(\rho_{t}).$  Then, the explorers compute a weighted average of their position in parameter space and direct their relaxation movement towards this average $\bar\theta$; this explains the first term on the right hand side of \eqref{eq:particle_a}. The role of the second term is to impose the property of noise strength decreasing proportionally to the distance of the explorer to the weighted average $\bar \theta$. The choice of the weighted average promotes the concentration towards parameter points $\theta$ leading to smaller values of $f$.
The resulting law of the system converges as $t\to\infty$ towards a Dirac mass concentrated at the MAP point $\theta^*$,
the global minimizer of~$f$,
under certain conditions on $f$; see \cite{carrillo2018analytical,ha2020convergence}.
    The weighted covariance $\bar {\mathcal C}=\mathcal C_{\beta}(\rho_t)$
    provides an alternative to the cooling schedule in \eqref{eq:particle_a}
    by way of using $\bar {\mathcal C}=\mathcal C_{\beta}(\rho_t)$ as the modulation of the noise.
In other words, one could propose as alternative to the CBO method \eqref{eq:particle_a}, the following mean field system
\begin{align}
    \label{eq:cbo_with_new_cooling}
    \dot\theta &= -(\theta - \bar \theta) + \sqrt{2\bar {\mathcal C}} \,\dot{\vect W}.
\end{align}
This gives \eqref{eq:mean_field_sdes2} in the optimization mode $\lambda=1$.
We show in \cref{prop:cv-cont} for the quadratic case, and~\cref{prop:convergence_rate_1d_opti} for the one-dimensional convex case,
that~\eqref{eq:cbo_with_new_cooling} converges precisely to the minimizer of~$f$,
whereas the CBO method usually concentrates to a point in the vicinity of the minimizer,
with an error depending on $\beta$.
On the other hand,
while the CBO dynamics concentrates exponentially fast under rather general assumptions on~$f$,
including the multidimensional non-convex setting~\cite{carrillo2018analytical,carrillo2019consensus},
the dynamics~\eqref{eq:cbo_with_new_cooling} converges algebraically in time
and our proofs concern only simple settings, considering
quadratic or one-dimensional convex functions $f$. Adapting the parameter $\beta$ during the evolution is shown empirically to improve the rate of convergence for \eqref{eq:cbo_with_new_cooling}, see the discussions in \cref{sec:numerical_experiments}; but analysis is needed to understand this property.
Other differences between the methods are that,
unlike CBO, the dynamics~\eqref{eq:cbo_with_new_cooling} is affine invariant (see \cref{ssub:affine_invariance}) and satisfies the invariant subspace property (see \cref{lemma:invariant_subspace}),
although further investigation is necessary to determine whether these two properties are useful in the context of optimization.

In terms of time complexity,
one iteration of (the particle approximations of) either method requires the evaluation of $f$ at all the particles;
thus, in the context of Bayesian inverse problems
where evaluating the forward model is the dominating computational expense,
the methods have a similar computational cost per iteration.
For problems where the dimension of the state space is very large and evaluation of $f$ is cheap,
however, the particle method corresponding to~\eqref{eq:cbo_with_new_cooling} is slightly more expensive than that of~\eqref{eq:particle_a},
as it requires calculating the square root of large matrices $\bar {\mathcal C}$.
We note, however,
that employing a generalized square root as proposed in~\cite{garbuno2020affine} for the ALDI method would help to mitigate this difficulty.

\subsection{Key Properties of the Mean Field Limits}\label{sec:properties}
In this subsection, we summarize key properties of the stochastic dynamics \eqref{eq:mean_field_sdes} and \eqref{eq:mean_field_sdes2}.
We consider, in turn:
(i) the time evolution of the laws;
(ii) the affine invariance;
(iii) the steady states; (iv) the evolution of the first and second moments; and (v) propagation properties for Gaussian initial conditions.

\subsubsection{Evolution Equations for the Law of the Mean Field Dynamical Systems.}
The time evolution of the law of the solution~\eqref{eq:mean_field_sdes} is governed by the following discrete-time dynamics on probability densities:
\begin{equation}
    \label{eq:iterative_scheme_measures}
    \rho_{n+1}(\theta) = \int_{\real^d} g\Bigl(\theta; \mathcal M_{\beta}(\rho_n) + \alpha \bigl(u - \mathcal M_{\beta}(\rho_n)\bigr), (1 - \alpha^2) \, \lambda^{-1} \, \mathcal  C_{\beta}(\rho_n)\Bigr) \, \rho_n(u)\, \d u.
\end{equation}
When $\alpha = 0$, the map~\eqref{eq:iterative_scheme_measures} takes a particularly simple form (recalling notation \ref{eq:gnote}
for a Gaussian):
\begin{equation}
    \label{eq:iterative_scheme_measures_infinite}
    \rho_{n+1} = g\bigl(\theta; \mathcal M_{\beta}(\rho_n),  \lambda^{-1} \, \mathcal  C_{\beta}(\rho_n)\bigr).
\end{equation}

Likewise, the time evolution of the law of the solution to~\eqref{eq:mean_field_sdes2} is governed by the following nonlinear and nonlocal Fokker--Planck equation:
\begin{equation}
    \label{eq:mean_field}
    \derivative{1}[\rho]{t} = \grad \cdot \Bigl( \bigl(\theta - \mathcal M_{\beta}(\rho)\bigr) \rho + \lambda^{-1} \, \mat C_{\beta}(\rho) \, \grad \rho \Bigr).
\end{equation}

\begin{remark}\label{existence and uniqueness}
    We will not discuss here the question of existence and uniqueness of solutions to~\eqref{eq:mean_field},
    and we assume from now on that there exists a unique strong solution to~\eqref{eq:mean_field} for smooth initial data $\rho_0\in\mathcal{P}_2(\real^d)$,
    implying in turn the existence and uniqueness of a solution to~\eqref{eq:mean_field_sdes2}.
    The equation \eqref{eq:mean_field} will be analyzed in
    subsequent work.
\end{remark}

\subsubsection{Affine Invariance.}%
\label{ssub:affine_invariance}
A fundamental property of both~\eqref{eq:mean_field_sdes} and~\eqref{eq:mean_field_sdes2} is that they are affine invariant,
in the sense of~\cite{goodman2010ensemble}; the utility of this concept has been established for MCMC methods in \cite{MR3362507}
and for Langevin based dynamics through the ALDI algorithm in \cite{garbuno2020affine}.
For linear inverse problems with posterior $\cN(\vect a, \mat A)$ this has the consequence that
the rate of convergence is independent of the conditioning of $\mat A$.
We study affine invariance of
\eqref{eq:mean_field_sdes}; a similar reasoning can be employed to show that the continuous-time mean-field dynamics \eqref{eq:mean_field_sdes2}
are also affine invariant.

In order to demonstrate affine invariance for~\eqref{eq:mean_field_sdes},
let $\{\theta_n\}_{n\in \nat}$ denote the solution to~\eqref{eq:mean_field_sdes} with initial condition $\theta_0 \sim \rho_0$,
and let $\rho_n = {\rm Law}(\theta_n)$.
Consider a vector $\vect b \in \real^d$ and an invertible matrix $B \in \real^{d \times d}$ which,
together, define the affine transformation $\theta \mapsto B \theta + \vect b$.
We introduce the following notation:
\[
    \widetilde \theta_n = B \theta_n + \vect b,
    \qquad \widetilde f(\widetilde \theta) = f\bigl(B^{-1} (\widetilde \theta - \vect b)\bigr), \qquad
    \widetilde L_{\beta}: \mu \mapsto \frac{\mu \e^{- \beta \widetilde f}}{\int_{\real^d} \e^{-\beta \widetilde f}}.
\]
We also introduce $\widetilde {\mathcal M}_{\beta}: \mu \mapsto \mathcal M(\widetilde L_{\beta} \mu)$
and $\widetilde {\mathcal C}_{\beta}: \mu \mapsto \mathcal C(\widetilde L_{\beta} \mu)$.
To prove the affine invariance of the scheme~\eqref{eq:mean_field_sdes},
we must show that $\{\widetilde \theta_n\}_{n \in \nat}$ is equal in law to the solution $\{\widehat \theta_n\}_{n \in \nat}$ of
\begin{equation}
    \label{eq:equation_affine_map}
    \widehat \theta_{n+1} = \alpha \widehat \theta_n + (1 - \alpha)\widetilde {\mathcal M}_{\beta}(\widehat \rho_n)
    +  \sqrt{(1 - \alpha^2) \, \lambda^{-1} \widetilde {\mathcal C}_{\beta} (\widehat \rho_n)} \, \widehat {\vect\xi}_n,
    \qquad \widehat \rho_n = {\rm Law}(\widehat \theta_n),
\end{equation}
with initial condition $\widehat \theta_0 = \widetilde \theta_0$ and where $\{\widehat {\vect \xi}_n\}_{n \in \nat}$ are independent $\normal (\vect 0, I_d)$ random variables.
In order to show this,
we apply the affine transformation $\theta \mapsto B \theta + \vect b$ to both sides of~\eqref{eq:mean_field_sdes},
which leads to
\[
    \widetilde \theta_{n+1} = \alpha \widetilde \theta_n + (1 - \alpha) \bigl(B \mathcal M_{\beta}(\rho_n) + \vect b\bigr)
    + B \sqrt{(1 - \alpha^2) \, \lambda^{-1} \mathcal C_{\beta} (\rho_n)} \, \xi_n, \\
    \qquad \rho_n = {\rm Law}(\theta_n).
\]
Now notice that $B \mathcal M_{\beta}(\rho_n) + \vect b = \widetilde {\mathcal M}(\widetilde \rho_n)$, where $\widetilde \rho_n = {\rm Law}(\widetilde \theta_n)$,
that
\[
    B \sqrt{\mathcal C_{\beta} (\rho_n)} \, {\vect\xi}_n = \sqrt{B \mathcal C_{\beta}(\rho_n) B^\t} \, {\vect\xi}_n \quad \text{in law,}
\]
and that $B \mathcal C_{\beta}(\rho_n) B^\t = \widetilde {\mathcal C}_{\beta}(\widetilde \rho_n)$,
which implies that $\{\widetilde \theta_n\}_{n \in \nat}$ is indeed a solution to~\eqref{eq:equation_affine_map}.

\subsubsection{Steady States.}%
The steady states of \eqref{eq:mean_field_sdes} and
\eqref{eq:mean_field_sdes2} coincide,
if they exist, and they are necessarily Gaussian.
Recall the notation \eqref{eq:gnote}. We have:

\label{sub:steady_states}
\begin{lemma}
    \label{lemma:basic_results_steady_states}
    Let probability distribution $\rho_{\infty}$ have finite second moment and be a steady-state solution
    of~\eqref{eq:iterative_scheme_measures} or \eqref{eq:mean_field}. Then
    \begin{equation}
        \label{eq:steady_state}
        \rho_{\infty}(\dummy) = g\bigl(\dummy; \mathcal M_{\beta}(\rho_{\infty}), \lambda^{-1} \mathcal C_{\beta}(\rho_{\infty})\bigr).
    \end{equation}
    Conversely, all probability distributions solving~\eqref{eq:steady_state} are steady states of \eqref{eq:iterative_scheme_measures} and \eqref{eq:mean_field}.
    In particular, all steady states are Gaussian (with the limiting case of Diracs included in the definition) and all Dirac masses are steady states.
\end{lemma}
\begin{proof}

    If $\rho_{\infty}$ is an invariant measure for the law of~\eqref{eq:mean_field_sdes2},
    then $\rho_{\infty}$ must be an invariant measure of the following SDE:
    \begin{equation}
        \label{eq:steady_ou}
        \d \theta_t = - \bigl(\theta_t - \mathcal M_{\beta}(\rho_{\infty})\bigr) \, \d t + \sqrt{2 \lambda^{-1} \mathcal C_{\beta}(\rho_{\infty})} \, \d {\vect W}_t.
    \end{equation}
    Since this is just the Ornstein--Uhlenbeck process,
    we deduce~\eqref{eq:steady_state}.

    Similarly, if $\rho_{\infty}$ is an invariant measure for the law of the discrete-time dynamics~\eqref{eq:mean_field_sdes},
    then~$\rho_{\infty}$ is the invariant measure of the following equation:
    \[
        X_{n+1} = \mathcal M_{\beta}(\rho_{\infty}) + \alpha\bigl(X_n - \mathcal M_{\beta}(\rho_{\infty})\bigr) + \sqrt{(1 - \alpha^2)\lambda^{-1} \mathcal C_{\beta}(\rho_{\infty})} \vect \xi_n,
    \]
    where $(\vect \xi_n)_{n = 0, 1,\dotsc}$ are independent $\normal(\vect 0, I_d)$ random variables.
    Since this equation is an exact discretization of~\eqref{eq:steady_ou},
    we deduce that~\eqref{eq:steady_state} holds.
\end{proof}

\subsubsection{Equations for the Moments.}%
\label{sub:equations_for_the_moments}
The evolution equations for the moments given in~\eqref{eq:equations_moments_discrete} hold regardless of whether $\rho_n$ is Gaussian
but they define closed equations characterizing $\rho_n$ completely in settings where $\rho_0$ is Gaussian.
The evolution of the moments can also be written for the limiting continuous time stochastic dynamical system~\eqref{eq:mean_field_sdes2} obtained when $\alpha\to 1$:
\begin{subequations}
\label{eq:equations_moments}
\begin{align}
    \label{eq:first_moment}%
    \partial_t \bigl(\mathcal M(\rho)\bigr) &= - \mathcal M(\rho) + \mathcal M_{\beta}(\rho), \\
    \label{eq:second_moment}%
    \partial_t \bigl(\mathcal C(\rho)\bigr)  &= - 2 \mathcal C(\rho) + 2 \lambda^{-1} \mathcal C_{\beta}(\rho).
\end{align}
\end{subequations}

\subsubsection{Propagation of Gaussians.}
We show that Gaussianity is preserved along the flow, both in discrete and continuous time.
\begin{lemma}
    \label{lemma:propagation_of_gaussians}
    Let $\lambda \in (0, 1]$ and $\beta>0$.
    \begin{enumerate}
 \item[(i)] Discrete time $\alpha=0.$ The law of ~\eqref{eq:mean_field_sdes}
 is Gaussian for all $n \in \nat.$
        \item[(ii)] Discrete time $\alpha\in(0,1)$. If the initial law $\rho_0$ for~\eqref{eq:mean_field_sdes} is Gaussian,
        then so is the law for any $n \in \nat_{>0}$,
        and the time evolution of the moments $(\vect m_n, C_n)$ of $\rho_n$ is governed by the recurrence relation
        \begin{subequations}
        \label{eq:equations_moments_discrete_gaussian}%
        \begin{align}
            \label{eq:first_moment_discrete_gaussian}%
            \vect m_{n+1} &= \alpha \vect m_n + (1 - \alpha) \vect m_{\beta}(\vect m_n, C_n), \\
            \label{eq:second_moment_discrete_gaussian}%
            C_{n+1}  &= \alpha^2 C_n + \lambda^{-1} (1 - \alpha^2) C_{\beta}(\vect m_n, C_n).
        \end{align}
        \end{subequations}
        with $\vect m_\beta$, $\mat C_\beta$ given by \eqref{eq:mbeta_and_cbete}.
    \item[(iii)] Continuous time $\alpha\to 1$.
        If the initial law $\rho_0$ for~\eqref{eq:mean_field_sdes2} is Gaussian,
        then so is the corresponding law for any $t>0$.
        The time evolution of the moments $\bigl(\vect m(t), \mat C(t)\bigr)$ of the solution is governed by the equation
    \begin{subequations}
        \label{eq:momentprop}
        \begin{align}
            \label{eq:momentprop-a}%
            \dot{\vect m} &= - \vect m + \vect m_\beta(\vect m, \mat C), \\
            \label{eq:momentprop-b}%
            \dot{\mat C} &= - 2 \mat C + 2 \lambda^{-1} \mat C_\beta(\vect m, \mat C).
        \end{align}
    \end{subequations}
    \end{enumerate}
\end{lemma}

\begin{proof}
For the discrete-time dynamics in setting (i),
this follows directly from~\eqref{eq:iterative_scheme_measures_infinite}.
For~(ii) note that,
if $\theta_n \sim \normal(\vect m_n, C_n)$,
then $\theta_{n+1}$, being the sum of Gaussian random variables as given in~\eqref{eq:mean_field_sdes},
is also normally distributed.

In order to show (iii), we consider a solution $\bigl(\vect m(t), \mat C(t)\bigr)$ to the moment equations \eqref{eq:momentprop}.
Then $g\bigl(\theta; \vect m(t), \mat C(t)\bigr)$ solves \eqref{eq:mean_field}.
    To see this, one can verify that general Gaussians $g(\theta; \vect m, \mat C)$ satisfy the relations
    \begin{gather*}
        \nabla_\theta g = -\nabla_{\vect m} g\,, \qquad x^T(\hessian_\theta g) y = 2 D_{\mat C} g : x\otimes y\,,
    \end{gather*}
    for any $x,y\in\real^d$;
    see similar computations in \cite{garbuno2020interacting,carrillo2019wasserstein}.
    The first identity can be checked directly,
    and the second identity follows e.g. from equations~(57) and (61) in \cite{cookbook}. Then
\begin{align*}
    \derivative{1}{t} \Bigl( g\bigl(\theta,\vect m(t), \mat C(t)\bigr) \Bigr)
    &= \nabla_{\vect m} g \cdot \dot{\vect m} + D_{\mat C} g : \dot{\mat C}
    \\&
    = -\nabla_{\vect m} g \cdot \left( \vect m - \vect m_\beta\right) + 2 D_{\mat C} g : \left( \lambda^{-1} \mat C_\beta - \mat C\right)
    \\&
    = \nabla_{\theta} g \cdot \left( \vect m - \vect m_\beta\right)
    + \nabla_\theta \cdot \left(-\mat C\nabla_\theta g \right) + \lambda^{-1}D^2_{\theta} g :  \mat C_\beta
    \\&
    = \grad_\theta \cdot \left( (\theta - \vect m_\beta) g + \lambda^{-1} \, \mat C_{\beta} \, \grad_\theta g \right),
\end{align*}
where we used the explicit expression of $C \grad_{\theta} g$ in the last equation.
\end{proof}

\subsection{Convergence for Gaussian Targets}%
\label{sub:convergence_to_equilibrium}

In this subsection,
we consider the case of a linear forward map in \eqref{eq:inverse_problem},
leading to the posterior distribution being a Gaussian $\cN(\vect a, \mat A)$ where, throughout, we assume that $A$ is strictly positive definite, $A\in\pos^d$.
The corresponding potential $f(\dummy)$ is given by
the quadratic function $f(\theta)=\frac{1}{2} \vecnorm{\theta - \vect a}_{\mat A}^2$. Recall the shorthand notation $\norm{B}_{A}=\norm{A^{-1/2}BA^{-1/2}}$.
Throughout this section, we denote
\[
    k_0= \norm{C_0^{-1}}_{A^{-1}}= \norm*{A^{1/2}C_0^{-1}A^{1/2}}.
\]

The main convergence results of this subsection, \cref{prop:cv-alpha0-main,prop:cv-discrete,prop:cv-cont},
establish the convergence of the moments of the solutions to~\eqref{eq:mean_field_sdes} and~\eqref{eq:mean_field_sdes2}, respectively,
in the case of Gaussian initial conditions.
All results show algebraic convergence in optimization mode ($\lambda=1$) and exponential convergence in sampling mode ($\lambda=(1+\beta)^{-1}$); this is analogous to what is
known about the EKI \cite{schillings2017analysis} and the EKS
\cite{garbuno2020interacting} methods.
We provide in~\cref{tab:cv-rates} an overview of the results we obtain.
Most proofs of the results presented in the rest of this subsection are given in \cref{sub:gaussian_proofs}.

\begin{table}[ht]
    \centering
    \begin{tabular}{|c|c|c|c|c|}
         \hline
         \phantom{$\Big($}
         & \multicolumn{2}{c|}{Sampling} & \multicolumn{2}{c|}{Optimization}
         \\ \hline
         \phantom{$\Big($}
         & Mean & Covariance & Mean & Covariance
         \\ \hline
         $\alpha = 0$
         & \phantom{$\Bigg($}\hspace{-4mm} $\left(\frac{1}{1+\beta}\right)^n$
         & $\left(\frac{1}{1+\beta}\right)^n$
         & $\frac{k_0}{k_0+\beta n}$
         & $\frac{k_0}{k_0+\beta n}$
         \\ \hline
         $\alpha \in (0, 1)$
         & \phantom{$\Bigg($}\hspace{-4mm} $\left(\frac{1+\alpha\beta}{1+\beta}\right)^n$
         & $\left(\frac{1+\alpha^2\beta}{1+\beta}\right)^n$
         & $\left(\frac{k_0+\beta}{k_0+\beta+\beta(1-\alpha^2)n}\right)^{\frac{1}{1+\alpha}}$
         & $\frac{k_0+\beta}{k_0+\beta+\beta(1-\alpha^2)n}$
         \\ \hline
         $\alpha = 1$
         & \phantom{$\Bigg($}\hspace{-4mm} $\e^{-\left(\frac{\beta}{1+\beta}\right)t}$
         & $\e^{-\left(\frac{2\beta}{1+\beta}\right)t}$
         & $\left(\frac{k_0+\beta}{k_0+\beta+2 \beta t}\right)^{\frac{1}{2}}$
         & $\frac{k_0+\beta}{k_0+\beta+2 \beta t}$
         \\ \hline
    \end{tabular}
    \caption{%
        Convergence rates for CBS in sampling and optimization modes,
        in the case of a Gaussian target distribution and a Gaussian initial condition with $C_0 \in \pos^d$.
        This table summarizes the results in \cref{prop:cv-alpha0-main,prop:cv-discrete,prop:cv-cont}. All rates are sharp, see \cref{rmk:sharpness}.
    }
    \label{tab:cv-rates}%
\end{table}

We draw a number of conclusions from these results. Firstly,
in the discrete time setting, smaller choices of $\alpha$ provide a faster rate of convergence, and choosing $\alpha=0$ is therefore the most favorable choice in this regard. Secondly, larger choices of $\beta$ increase the speed of convergence, without limit as $\beta \to \infty$ for $\alpha=0$; in the case $\alpha>0$, increasing $\beta$ is favourable but does not give rates which increase without limit.

%%%%%%%%%%%

\subsubsection{Convergence Analysis for the Discrete-Time Dynamics.}
Using the explicit expression of the weighted moments in the Gaussian case~\eqref{eq:explicit_moments},
we can rewrite the right-hand sides of \cref{eq:equations_moments_discrete_gaussian} as
\begin{align*}
    (\vect m_{n+1}-\vect a) &= \left[\alpha I_d +(1-\alpha) A(A+\beta C_n)^{-1}\right] (\vect m_n -\vect a), \\
    C_{n+1} &=
    \left[\alpha^2 I_d + (1-\alpha^2) \lambda^{-1} A(A+\beta C_n)^{-1}\right] C_n\,.
\end{align*}
Letting $\widetilde {\vect m}_n := A^{-1/2} (\vect m_n - \vect a)$ and $\widetilde{\mat C}_n := \beta A^{-1/2} {\mat C}_n A^{-1/2}$,
we can verify that $(\widetilde {\vect m}_n, \widetilde{\mat C}_n)_{n\in\nat}$ solves
the following recurrence relation:
\begin{subequations}
\label{eq:moments_quadratic_tilde}
\begin{align}
    \label{eq:first_moment_quadratic_discrete_tilde}%
    \widetilde {\vect m}_{n+1} &= \left[\alpha I_d +(1-\alpha) (I_d+\widetilde{\mat C}_{n})^{-1}\right] \widetilde {\vect m}_{n}\,, \\
    \label{eq:second_moment_quadratic_discrete_tilde}%
    \widetilde{\mat C}_{n+1}  &=
    \left[\alpha^2 I_d + (1-\alpha^2) \lambda^{-1} (I_d+\widetilde{\mat C}_n )^{-1}\right]\widetilde{\mat C}_n \,.
\end{align}
\end{subequations}
This is a recurrence relation uniquely solvable given initial conditions $(\widetilde {\vect m}_0, \widetilde {\mat C}_0)$.
We begin by studying the easier case $\alpha=0$, where the convergence of the scheme can be computed explicitly by a direct argument.
\begin{lemma}\label{lem:cv-alpha0}
    Consider the iterative scheme~\eqref{eq:iterative_scheme_measures_infinite} with $\alpha = 0$ and initial conditions $(\vect m_0, \mat C_0) \in \real^d \times \pos^d$.
    Then, for any $\lambda\in(0,1]$ and $\beta>0$, we have
    \[
    \vect m_{n} = \vect a
        + \lambda^n \mat C_{n} \mat C_{0}^{-1} (\vect m_0 - \vect a),\qquad
        C_{n}^{-1} =
        \begin{cases}
            \lambda^n C_{0}^{-1} + (1-\lambda^n) C_{\infty}^{-1}  & \text{if $\lambda \neq 1$}, \\
            C_0^{-1} + n \beta A^{-1} & \text{if $\lambda = 1$}.
        \end{cases}
    \]
\end{lemma}
\begin{proof}
When $\alpha = 0$, the evolution equations~\eqref{eq:moments_quadratic_tilde} for the moments simplify to
\begin{align*}
    \widetilde {\vect m}_{n+1} =  (I_d+ \widetilde{\mat C}_{n})^{-1} \widetilde {\vect m}_{n}\,,\qquad
    \widetilde{\mat C}_{n+1}^{-1}  =\lambda \left(\widetilde{\mat C}_n^{-1}
     + I_d\right)\,.
\end{align*}
For $\lambda=1$, the result for the covariance matrix is easily obtained by solving the second equation explicitly for $\widetilde{\mat C}_n^{-1}$.
Next, consider the case $\lambda \neq 1$. We have
\begin{align*}
    & \widetilde{\mat C}_n^{-1}
    = \lambda^{n}\widetilde{\mat C}_0^{-1} + (\lambda + \dotsc + \lambda^n)  I_d
    = \lambda^{n}\widetilde{\mat C}_0^{-1} +
    \lambda \left( \frac{1 - \lambda^{n}}{1 - \lambda} \right) I_d\,.
\end{align*}
For the evolution of the mean, notice that
\[
 \widetilde{\mat C}_{n+1}^{-1} \widetilde {\vect m}_{n+1}
 =\lambda \left(\widetilde{\mat C}_n^{-1}+ I_d\right)
 (I_d+ \widetilde{\mat C}_{n})^{-1} \widetilde {\vect m}_{n}
 =\lambda \widetilde{\mat C}_{n}^{-1} \widetilde {\vect m}_{n} \,.
\]
Hence, $\widetilde {\vect m}_{n}=\lambda^n\widetilde{\mat C}_{n} \widetilde{\mat C}_{0}^{-1} \widetilde {\vect m}_{0}$ and the result follows.
\end{proof}

We deduce from this result a convergence estimate for the mean and the covariance of the iterates.
\begin{proposition}
    \label{prop:cv-alpha0-main}
    Consider the iterative scheme~\eqref{eq:iterative_scheme_measures_infinite} with $\alpha = 0$ and initial conditions $(\vect m_0, \mat C_0) \in \real^d \times \pos^d$.
    Then the following statements hold:
    \begin{enumerate}
        \item[(i)] Sampling mode $\lambda = (1 + \beta)^{-1}$.
            For all $n\in\nat$, it holds that
            \begin{align*}
                |\vect m_{n} - \vect a|_A&\le \max \left(1, k_0 \right) \lambda^n |\vect m_0 - \vect a|_A\,,\\
                \norm{ C_n-A }_{A} &\le \max\left(1, k_0 \right) \lambda^n \norm{C_0-A}_{A}\,.
            \end{align*}
        \item[(ii)] Optimization mode $\lambda=1$. For all $n\in \nat$, it holds that
            \begin{align*}
                |\vect m_{n} - \vect a|_A&\le \left(\frac{k_0}{k_0+\beta n}\right) |\vect m_0 - \vect a|_A\,,\qquad
                C_n \preccurlyeq \left(\frac{k_0}{k_0+\beta n}\right) C_0.
            \end{align*}
    \end{enumerate}
\end{proposition}

In order to study the convergence in the general case $\alpha \in (0, 1)$,
we will reduce the evolution of the moments~\eqref{eq:moments_quadratic_tilde} to the scalar case,
\begin{subequations}
\label{eq:1Dmoments}
\begin{align}
    u_{n+1} &= \left[\alpha +(1-\alpha) (1 + v_{n})^{-1}\right] u_{n}, \\
    v_{n+1}  &=
    \left[\alpha^2 + (1-\alpha^2) \lambda^{-1} (1 + v_n )^{-1}\right]v_n
\end{align}
\end{subequations}
by diagonalization. Then, using \cref{lemma:convergence_recursion_equations}, the asymptotic behavior of the moments can be summarized as follows.

\begin{proposition}
    \label{prop:cv-discrete}
    Consider the iterative scheme~\eqref{eq:mean_field_sdes} with $\alpha \in (0,1)$ and initial conditions $(\vect m_0, \mat C_0) \in \real^d \times \pos^d$.
    Then the following statements hold:
    \begin{enumerate}
        \item[(i)] Sampling mode $\lambda = (1 + \beta)^{-1}$. For all $n\in\nat$,
            \begin{align*}
                |\vect m_{n} - \vect a|_A&\le \max \left(1, k_0 \right)^{\frac{1}{1+\alpha}}
                \bigl((1-\alpha)\lambda+\alpha\bigr)^n |\vect m_0 - \vect a|_A\,,\\
                \norm{ C_n-A }_{A} &\le \max\left(1, k_0 \right) \left((1-\alpha^2)\lambda+\alpha^2\right)^n \norm{C_0-A}_{A}\,.
            \end{align*}
        \item[(ii)] Optimization mode $\lambda=1$. For all $n\in \nat$, it holds that
            \begin{align*}
                |\vect m_{n} - \vect a|_A&\le \left(\frac{k_0+\beta}{k_0+\beta+\beta(1-\alpha^2)n}\right)^{\frac{1}{1+\alpha}}
                |\vect m_0 - \vect a|_A\,,\\
                C_n &\preccurlyeq \left(\frac{k_0+\beta}{k_0+\beta+\beta(1-\alpha^2)n}\right) C_0 \,.
            \end{align*}
    \end{enumerate}
\end{proposition}

\subsubsection{Convergence Analysis for the Continuous-time Dynamics.}%
Next, we consider the limiting case $\alpha\to 1$.
Rewriting the right-hand side of~\eqref{eq:momentprop-a} and~\eqref{eq:momentprop-b} using \eqref{eq:explicit_moments},
we obtain for any $\lambda\in (0,1]$ and $\beta>0$,
\begin{subequations}
\label{eq:system_moments_continuous_time}
\begin{align}
    \label{eq:system_moments_continuous_time_m}
    \dot {\vect m} &= - \beta \mat C \left(A + \beta\mat C \right)^{-1} (\vect m - \vect a), \\
    \label{eq:system_moments_continuous_time_c}
    \dot {\mat C}  &= - 2\beta \, \mat C \,  \left(A + \beta\mat C \right)^{-1} \, \left( \mat C -  \left(\frac{1-\lambda}{\beta\lambda}\right) A \right).
\end{align}
\end{subequations}

\begin{proposition}
    \label{prop:cv-cont}
    Let $\bigl(\vect m(t), \mat C(t)\bigr)$ denote the solution to~\cref{eq:system_moments_continuous_time} with initial conditions $(\vect m_0, \mat C_0) \in \real^d \times \pos^d$.
    Then the following statements hold:
    \begin{itemize}
        \item[(i)] Sampling mode $\lambda = (1 + \beta)^{-1}$. For all $t>0$,
            \begin{subequations}
            \begin{align*}
                \vecnorm{\vect m(t)-\vect a}_A &\leq \max\bigl(1,k_0^{\lambda/2}  \bigr) \e^{-(1 - \lambda) t} \vecnorm{\vect m_0 - \vect a}_A, \\
                \norm{\mat C(t) - \mat A}_A &\leq
                \max\bigl(1,k_0^{\lambda} \bigr)\e^{-2 (1 - \lambda) t} \norm{ \mat C_0-\mat A}_A\,.
            \end{align*}
            \end{subequations}
        \item[(ii)] Optimization mode $\lambda=1$. For all $t \geq 0$, it holds
            \begin{subequations}
            \begin{align*}
                \vecnorm{\vect m(t) - \vect a}_A &\leq\left(  \frac{k_0+\beta}{k_0 + \beta + 2 t \beta} \right)^{\frac{1}{2}} \vecnorm{\vect m_0 - \vect a}_A\,,\\
                C(t) &\preccurlyeq \left(  \frac{k_0+\beta}{k_0 + \beta + 2 t \beta} \right)C_0 \,.
            \end{align*}
            \end{subequations}
    \end{itemize}
\end{proposition}

\begin{remark}[Discrete to Continuum]
    Notice that, by letting $\alpha = \e^{-t/n}$ in the convergence results obtained for $\alpha\in(0,1)$ in~\cref{prop:cv-discrete}
    and taking the limit $n \to \infty$,
    we recover the convergence results of the continuous-time setting,
    up to the constant prefactor.
\end{remark}

\begin{remark}[Sharpness]\label{rmk:sharpness}
    It is possible to show, using the lower bounds on the trend to equilibrium provided by \cref{lemma:convergence_recursion_equations,lemma:cv-alpha1-1D},
    that the convergence rates we obtained in \cref{prop:cv-discrete,prop:cv-cont} are all sharp with respect to $n$ and $t$ respectively.
    Note that the argument leading to \cref{prop:cv-discrete} also applies to the case $\alpha=0$.
    However, the upper bounds we obtain in \cref{prop:cv-alpha0-main} are stronger than those we would be able to obtain by applying \cref{lemma:convergence_recursion_equations} for $\alpha=0$.
    Lower bounds for the sampling mode in the case $\alpha=0$ can be obtained the same way as for $\alpha\in(0,1)$.
    In optimization mode ($\lambda=1)$,
    we can derive lower bounds explicitly using the expression from \cref{lem:cv-alpha0} as follows:
    for $\widetilde{\mat C}_n := \beta A^{-1/2} {\mat C}_n A^{-1/2}$, we have $\widetilde{\mat C}_0 \preccurlyeq \norm*{\widetilde{\mat C}_0} I_d$, so
    \begin{align*}
        \widetilde{\mat C}_{n}^{-1}  = \widetilde{\mat C}_0^{-1}
     + n I_d
     \preccurlyeq \left(1+n \norm*{\widetilde{\mat C}_0}\right)  \widetilde{\mat C}_0^{-1}
     \quad \Rightarrow \quad
      \mat C_n \succcurlyeq \left(\frac{1}{1+\beta n \norm*{\mat C_0}_A}\right) \mat C_0.
    \end{align*}
    The conclusion from the above observations is that all rates provided in \cref{tab:cv-rates} are sharp.
\end{remark}

\begin{remark}
    [Attractor]
    As a consequence of the above convergence results for linear objective functions~$f$,
    the steady state $(\vect a, A)$ is the unique attractor of
    the moment equations \eqref{eq:equations_moments_discrete_gaussian}
    and \eqref{eq:momentprop}
    when taking an initial condition with $C_0 \in \pos^d$.
    Therefore, whilst the mean-field dynamics
    \eqref{eq:iterative_scheme_measures} and \eqref{eq:mean_field} admit infinitely many steady states given by all Dirac distributions in
    addition to the Gaussian steady state $\normal(\vect a, A)$,
    the solutions to the mean-field dynamics always converge to the desired target measure $\normal(\vect a, A)$ when initialized at Gaussian initial conditions with $C_0 \in \pos^d$,
    avoiding the manifold of Diracs along the evolution.
\end{remark}

%%%%%%%%%%%%%%%%%%%%%%

\subsection{Particle Approximations}\label{sec:notation}

In this subsection we describe particle approximations of
the mean field dynamics \eqref{eq:mean_field_sdes} and~\eqref{eq:mean_field_sdes2}. This leads to the implementable algorithms used in Section \ref{sec:numerical_experiments}.
The following is a discrete-time system of interacting particles in $\real^d$ with mean field limit given by~\eqref{eq:mean_field_sdes}:
\begin{equation}
    \label{eq:interacting_particle_system}
    \theta^{(j)}_{n+1} =
    \mathcal M_{\beta} (\rho_n^J) + \alpha \bigl(\theta^{(j)}_n - \mathcal M_{\beta}(\rho_n^J)\bigr)
    +  \sqrt{(1 - \alpha^2) \, \lambda^{-1} \mathcal C_{\beta} (\rho_n^J)} \, \vect \xi_n^{(j)},
    \qquad j = 1, \dots, J.
\end{equation}
Here $\vect \xi^{(j)}_n$, for $j \in \{1, \dotsc, J\}$ and $n \in \nat$, are independent $\normal(\vect 0, I_d)$ random variables,
and $\rho^J_n$ is the empirical measure associated with the particle system at iteration $n$,
\[
    \rho^J_n := \frac{1}{J} \sum_{j=1}^{J} \delta_{\theta^{(j)}_n}\,.
\]
We note that
\begin{subequations}
    \label{eq:emp}
\begin{align}
    \mathcal M_{\beta} (\rho^J_n) &= \frac%
    {\sum_{j=1}^{J}  \e^{-\beta f(\theta^{(j)}_n)} \, \theta^{(j)}_n }
    {\sum_{j=1}^{J}  \e^{-\beta f(\theta^{(j)}_n)}},
    \\
    \label{eq:covariance}
    \mathcal C_{\beta} (\rho^J_n) &= \frac%
    {\sum_{j=1}^{J}  \left((\theta^{(j)}_n - \mathcal M_{\beta}(\rho^J_n)) \otimes \left(\theta^{(j)}_n - \mathcal M_{\beta}(\rho^J_n)\right)\right) \, \e^{-\beta f(\theta^{(j)}_n)}}
    {\sum_{j=1}^{J}  \e^{-\beta f(\theta^{(j)}_n)} }.
\end{align}
\end{subequations}
The limit cases $\alpha=0$ and $\alpha\to 1$ for fixed $\lambda>0$ and $\beta>0$ reduce to simpler systems.
Indeed, in the case where $\alpha = 0$,
the method simplifies to
\[
    \theta^{(j)}_{n+1} =
    \mathcal M_{\beta} (\rho_n^J)
    +  \sqrt{\lambda^{-1} \mathcal C_{\beta} (\rho_n^J)} \, \vect \xi_n^{(j)},
    \qquad j = 1, \dots, J.
\]
On the other hand,
when $\alpha \approx 1$,
the particle evolution \cref{eq:interacting_particle_system} may be viewed as a time discretization with timestep $\Delta t = - \log \alpha$
of the following continuous-time interacting particle system,
in which we generalize the notation \eqref{eq:emp} to continuous time in the obvious way:
\begin{align}
    \label{eq:consensus_based_sampler}%
    \dot\theta^{(j)} &=
    - \bigl(\theta^{(j)} - \mathcal M_{\beta}(\rho^J_t)\bigr)
    +  \sqrt{2 \lambda^{-1} \, \mathcal C_{\beta} (\rho^J_t)} \, \dot{\vect W}^{(j)},
    \qquad j = 1, \dots, J,
\end{align}
where $\{ \vect W^{(j)} \}_{j=1}^{J}$ are independent standard Brownian motions in $\real^d$.
The formal mean field limit of this equation is given by~\eqref{eq:mean_field_sdes2}.

We note that
the finite-dimensional particle systems~\eqref{eq:interacting_particle_system} and \eqref{eq:consensus_based_sampler} are both affine invariant;
the proof is similar to that given for the mean-field limit.
    In addition, like ensemble Kalman based methods for inverse problems~\cite{MR3041539},
    the particle systems~\eqref{eq:interacting_particle_system} and \eqref{eq:consensus_based_sampler} both satisfy the following invariant subspace property.
    \begin{lemma}
        \label{lemma:invariant_subspace}
        Let $\mathcal S$ denote the linear span of~$\{\theta_0^{(j)}\}_{j=1}^J$.
        Then $\theta^{(j)}_n \in \mathcal S$ for all $(j, n) \in \{1, \dotsc, J\} \times \nat$ and
        $\theta^{(j)}_t \in \mathcal S$ for all $(j, t) \in \{1, \dotsc, J\} \times [0, \infty)$.
    \end{lemma}
    \begin{proof}
        We prove only the first claim, which follows from a simple recursion.
        Let us assume the claim is true for $(j, n) \in \{1, \dotsc, J\} \times \{0, \dotsc, N\}$
        and prove that it is then also true for $n = N+1$.
        Let $\vect a \in \mathcal S^{\perp}$,
        where $\mathcal S^{\perp}$ is the orthogonal complement of~$\mathcal S$ in $\real^d$.
        Taking the inner product of both sides of~\eqref{eq:interacting_particle_system} with $\vect a$,
        we obtain for all $j \in \{1, \dotsc, J\}$ that
        \begin{align*}
            \vect a^\t \theta^{(j)}_{N+1}
            &= \vect a^\t \mathcal M_{\beta} (\rho_n^J) + \alpha \vect a^\t \bigl(\theta^{(j)}_N - \mathcal M_{\beta}(\rho_N^J)\bigr)
            + \vect a^\t \sqrt{(1 - \alpha^2) \, \lambda^{-1} \mathcal C_{\beta} (\rho_N^J)} \, \vect \xi_N^{(j)} \\
            &= 0 + 0 + \vect a^\t \sqrt{(1 - \alpha^2) \, \lambda^{-1} \mathcal C_{\beta} (\rho_N^J)} \, \vect \xi_N^{(j)},
        \end{align*}
        and so by the Cauchy--Schwarz inequality,
        \begin{align*}
            \lvert \vect a^\t \theta^{(j)}_{N+1} \rvert^2
            \leq (1 - \alpha^2) \, \lambda^{-1}
            \left\lvert \vect \xi_N^{(j)}  \right\rvert^2 \, \bigl\lvert \sqrt{\mathcal C_{\beta} (\rho_N^J)} \vect a \bigr\rvert^2
            = 0,
        \end{align*}
        because $\mathcal C_{\beta} (\rho_N^J) \vect a = 0$ by the formula~\eqref{eq:covariance} for the weighted covariance $\mathcal C_{\beta}(\rho_N^J)$.
        Since $\vect a$ was arbitrary in $\mathcal S^\perp$,
        the proof is complete.
    \end{proof}

\begin{remark}[Cooling schedule]
    To improve algorithmic implementations it will be
    of value to develop a rigorous understanding of the relationship between the number of particles $J$
    and the parameter $\beta$ needed to establish good performance of the method.
    Relatedly, it will also be useful to investigate theoretically the rate of convergence to equilibrium
    in the setting where a cooling schedule is employed for $\beta$. See \cref{sec:numerical_experiments} for numerical investigations in this direction.
\end{remark}

\section{Analysis Beyond The Gaussian Setting}\label{sec:results}

In this section,
we study the proposed method~\eqref{eq:mean_field_sdes} in the case where the function $f$ is not necessarily quadratic, and so the target probability distribution may be non-Gaussian.
We begin, in \cref{sec:prelimbounds}, by presenting preliminary bounds on $\vect m _\beta(\vect m, \mat C)$ and $\mat C_\beta(\vect m, \mat C)$ defined in~\eqref{eq:mbeta_and_cbete},
and then we analyze the optimization ($\lambda = 1$) and sampling ($\lambda = (1 + \beta)^{-1}$) methods in \cref{sub:analysis_of_the_optimization_scheme,sub:analysis_of_the_sampling_scheme}, respectively.
The proofs of all results are presented in \cref{sec:proof_of_the_main_results}, with the exception
of Theorem \ref{thm:existence-sstates} which is presented
in-text.

The results in this section are based on the following two assumptions.
\begin{assumption}[Convexity of the potential]
    \label{assumption:convexity_potential}
    The function $f$ satisfies $f \in C^{2}(\real^d)$ and $\hessian f(\theta) \succcurlyeq \Lhess \succcurlyeq \lhess I_d$ for all $\theta \in \real^d$,
    for some $\Lhess \in \pos^d$
    and some $\lhess > 0$.
\end{assumption}
Assumption~\ref{assumption:convexity_potential} guarantees the existence of a unique global minimizer for $f$,
which we will denote throughout this section by
\[
    \theta_*:=\argmin_{\theta\in\real^d} f(\theta).
\]
\begin{assumption}
    [Bound from above on the Hessian]
    \label{assumption:convexity_potential_above}
    The function $f$ satisfies $f \in C^{2}(\real^d)$ and $\hessian f(\theta) \preccurlyeq \Uhess \preccurlyeq u I_d$ for all $\theta \in \real^d$,
    for some $\Uhess \in \pos^d$ and some $\uhess > 0$.
\end{assumption}

These assumptions are very similar to the ones made in \cite{carrillo2018analytical} in order to
show the convergence of the CBO method \cite{pinnau2017consensus} for global optimization.
The convergence results we present in this section are summarized in~\cref{table:convergence_rates_nongaussian}.

\begin{table}[ht]
    \centering
    \begin{tabular}{|c|c|c|c|c|}
         \hline
         \phantom{$\Big($}
         & \multicolumn{2}{c|}{Sampling} & \multicolumn{2}{c|}{Optimization}
         \\[2mm] \hline
         \phantom{$\Big($}
         &  Mean ($d=1$) & Covariance ($d=1$) & Mean ($d=1$) & Covariance (any $d$)
         \\[2mm] \hline
         $\alpha = 0$
         & \phantom{$\Bigg($}\hspace{-4mm} $\left(\frac{k}{\beta}\right)^n$
         & $\left(\frac{k}{\beta}\right)^n$
         & $\lesssim \frac{\log (n)}{n}$
         & $\frac{\widetilde k_0}{\widetilde k_0+\beta n}$
         \\[2mm] \hline
         $\alpha \in (0, 1)$
         & \phantom{$\Bigg($}\hspace{-4mm} $\left(\alpha + (1-\alpha^2)\frac{k}{\beta}\right)^n$
         & $\left(\alpha + (1-\alpha^2)\frac{k}{\beta}\right)^n$
         & $\begin{array}{c}
               \lesssim n^{-1/q}  \\
               \mbox{(not optimal)}
         \end{array}$
         & $\frac{\widetilde k_0+\beta}{\widetilde k_0+\beta+\beta(1-\alpha^2)n}$
         \\ \hline
         $\alpha = 1$
         & \phantom{$\Bigg($}\hspace{-4mm} $\e^{-\left(1 - \frac{2k}{\beta} \right) t}$
         & $\e^{-\left(1 - \frac{2k}{\beta} \right) t}$
         & $\begin{array}{c}
               \lesssim t^{-1/q}  \\
               \mbox{(not optimal)}
         \end{array}$
         & $\frac{\widetilde k_0+\beta}{\widetilde k_0+\beta+2 \beta t}$
         \\ \hline
    \end{tabular}
    \caption{%
        Sharp upper bounds on the convergence rates for CBS in sampling and optimization modes,
        in the case of a non-Gaussian target distribution and a Gaussian initial condition with strictly positive definite covariance matrix $C_0$.
        Here $k$ is a positive constant independent of $n$, $t$, $\alpha$ and $\beta$,
        and $\widetilde k_0 := \norm{L^{1/2}C_0^{-1}L^{1/2}}$,
        where $L$ is the symmetric positive definite matrix from \cref{assumption:convexity_potential},
        and $q$ is any constant strictly greater than $2 \max(2, \uhess/\lhess)$,
        where $\lhess$ and $\uhess$ are the constants from \cref{assumption:convexity_potential} and \cref{assumption:convexity_potential_above}, respectively.
        Obtaining sharp convergence rates for the mean in the non-Gaussian case for $\alpha \neq 0$ in optimization mode is an open problem.
    }
    \label{table:convergence_rates_nongaussian}%
\end{table}

\subsection{Preliminary Bounds}\label{sec:prelimbounds}
We first obtain sharp bounds on $C_{\beta}$ which,
in the special case when $f$ is quadratic,
enable to recover~\eqref{eq:explicit_second_moment}.
The first bound relies on a logarithmic Sobolev inequality for the probability measure $\frac{1}{Z_\beta} \e^{-\beta f}$,
where $Z_\beta$ is the normalization constant.
\begin{lemma}[Upper bound on weighted covariance]
    \label{lemma:bound_second_moment}
    If~\cref{assumption:convexity_potential} holds,
    then
    \[
        \forall (\vect m, C) \in \real^d \times \pos^d, \qquad
       \mat C_\beta(\vect m, \mat C) \preccurlyeq \left( C^{-1} + \beta \Lhess \right)^{-1}.
    \]
\end{lemma}

\begin{remark}
    We note that, by the standard Holley--Stroock result, see e.g.~\cite[Theorem 2.11]{MR3509213},
    a similar bound could be obtained when $f$ is of the type $f_c + f_b$,
    where $f_c$ satisfies the convexity property \cref{assumption:convexity_potential}
    and $f_b$ is a bounded function.
\end{remark}

The next lemma provides a bound from below on $C_\beta$.
\begin{lemma}
    [Lower bound on weighted covariance]
    \label{lemma:lower_bound_second_moment}
    If \cref{assumption:convexity_potential_above} holds,
    then
    \[
        \forall (\vect m, C) \in \real^d \times \pos^d, \qquad
        \mat C_\beta(\vect m, \mat C) \succcurlyeq \left( C^{-1} + \beta \Uhess \right)^{-1}.
    \]
\end{lemma}

We now obtain a crude bound on the weighted first moment $\vect m_\beta(\vect m, C)$,
which will be our starting point for establishing the existence of a steady state for the sampling scheme.
This bound is useful because it shows that $\vect m_\beta(\vect m, C) \xrightarrow[\beta \to \infty]{} \theta_*$ for any fixed $\vect m$ and $C > 0$.
\begin{lemma}
    [Bound on weighted mean]
    \label{lemma:first_moment_several_dimensions}
    If~\cref{assumption:convexity_potential,assumption:convexity_potential_above} hold,
    then there exists a positive constant $k = k(\lhess, \uhess, d)$ such that,
    \[
        \forall (\vect m, C, \beta) \in \real^d \times \pos^d \times \real_{>0}, \qquad
        \vecnorm{\vect m_\beta(\vect m, C) - \theta_*}
        \leq \sqrt{\frac{\norm{C^{-1}}}{\lhess \beta}} \vecnorm{\vect m - \theta_*} + k \left( \frac{1}{\norm{C}} + \beta \lhess \right)^{-1/2}.
    \]
\end{lemma}
Unfortunately, this bound degenerates in the limit $C \to 0$.
In spatial dimension one, we will obtain, in the proof of~\cref{proposition:convergence_optimization},
a finer bound on the weighted mean that can be used for proving convergence of the optimization scheme.

%%%%%%%%%%%%%%%%%%%%%%%%%%%%%%%%%%%%%%%%%%%%%%%%%%%%
%%%%%%%%%%%%%%%%%%%%%%%%%%%%%%%%%%%%%%%%%%%%%%%%%%%%
\subsection{Analysis of the Optimization Scheme}%
\label{sub:analysis_of_the_optimization_scheme}

In this subsection, we are concerned with the
large-time convergence of the law of the solutions to
the mean-field evolution equations~\eqref{eq:mean_field_sdes} and~\eqref{eq:mean_field_sdes2} when $\lambda = 1$
and under the following assumption on the initial condition:
\begin{assumption}
    [Non-degenerate Gaussian initial conditions]
    \label{assumption:gaussian_init}
    \label{lab:gaussian_initial_conditions}
    The initial condition for the mean field evolution~\eqref{eq:iterative_scheme_measures}
    (or~\eqref{eq:mean_field}, in the continuous time setting)
    is Gaussian with strictly positive definite covariance matrix.
\end{assumption}

Under this assumption,
following~\cref{lemma:propagation_of_gaussians},
the solutions are normally distributed for all (discrete or continuous) times
with the first and second moments evolving according to~\cref{eq:equations_moments_discrete_gaussian} and~\eqref{eq:momentprop}, respectively.
We will show that, under appropriate assumptions, the mean converges to $\theta_*$ and the covariance to zero.

Throughout this subsection,
we denote by $\{(\vect m_n, C_n)\}_{n \in \nat}$ a solution to~\eqref{eq:equations_moments_discrete_gaussian} with $C_0 \succcurlyeq 0$,
and by $\big\{\bigl(\vect m(t), C(t)\bigr)\big\}_{t \in [0, \infty)}$ a solution to~\eqref{eq:momentprop} with $C(0) \succcurlyeq 0$.
We also denote by $\rho_n$ and $\rho_t$ solutions to~\eqref{eq:iterative_scheme_measures} and~\eqref{eq:mean_field}, respectively.

We begin by showing that the covariance matrices decrease to zero with rates matching those obtained
in the case of quadratic $f$ in~\cref{sub:convergence_to_equilibrium},
up to constant prefactors.

\begin{proposition}
    [Collapse of the ensemble in optimization mode]
    \label{lemma:collapse_optimization}
    Let $\lambda=1$ and $\beta>0$ and assume that \cref{assumption:convexity_potential} holds.
    Then we have
    \begin{enumerate}
        \item[(i)] Discrete time $\alpha=0$. If $C_0 \in \pos^d$, then
            \begin{equation}
        \label{eq:upper_bound_collapse-alpha0}
        C_{n} \preccurlyeq
        \left(\frac{\norm{L^{-1/2}C_0^{-1}L^{-1/2}}}{\norm{L^{-1/2}C_0^{-1}L^{-1/2}} +\beta n}\right) C_0.
    \end{equation}
        \item[(ii)] Discrete time $\alpha\in (0,1)$. If $C_0 \in \pos^d$, then
    \begin{equation}
        \label{eq:upper_bound_collapse}
        C_{n} \preccurlyeq
        \left(\frac{\norm{L^{-1/2}C_0^{-1}L^{-1/2}} +\beta}{\norm{L^{-1/2}C_0^{-1}L^{-1/2}} +\beta+ \beta(1-\alpha^2)n}\right) C_0.
    \end{equation}
    \item[(iii)] Continuous time $\alpha=1$. If $C(0) \in \pos^d$, then
    \begin{equation}
        \label{eq:upper_bound_collapse_continuous}
        C(t) \preccurlyeq
        \left(\frac{\norm{L^{-1/2}C(0)^{-1}L^{-1/2}} +\beta}{\norm{L^{-1/2}C(0)^{-1}L^{-1/2}}+\beta+ 2\beta t}\right) C(0).
    \end{equation}
        \end{enumerate}
\end{proposition}

Ideally, we would like to show that $\vect m_n \xrightarrow[n\to \infty]{} \theta_*$ and $\vect m(t) \xrightarrow[t\to \infty]{} \theta_*$;
however,
we were able to show this result only in the one-dimensional setting.
In the multi-dimensional case, we establish the following weaker result.

\begin{theorem}
    \label{thm:no_bad_convergence}
    Let $\lambda=1$, $\beta>0$, $\mat C_0\in\pos^d$,
    and suppose that \cref{assumption:convexity_potential,assumption:convexity_potential_above} hold.
    If there exists $\hat\theta \in \real^d$ such that $\vect m_n \xrightarrow[n \to \infty]{} \hat\theta$ for some $\alpha \in [0, 1)$ or $\vect m(t) \xrightarrow[t \to \infty]{} \hat\theta$ for $\alpha=1$,
    then $\hat\theta=\theta_*$ is the minimizer of~$f$.
\end{theorem}
It follows from the identity
\begin{equation}
    \label{eq:wasserstein}
    \forall \mu \in \mathcal P_2(\real^d), \qquad W_2(\mu, \delta_{\theta_*})^2 = \vecnorm{\mathcal M(\mu) - \theta_*}^2 + \trace \bigl(\mathcal C(\mu)\bigr),
\end{equation}
where $\wass[2]{\dummy}{\dummy}$ denotes the quadratic Wasserstein distance,
that \cref{lemma:collapse_optimization} and \cref{thm:no_bad_convergence}
can be combined in order to obtain convergence results for the solutions to the mean field systems~\eqref{eq:iterative_scheme_measures} and~\eqref{eq:mean_field}.
For example, the following result holds in the discrete-time case.
\begin{corollary}
    Suppose that \cref{assumption:convexity_potential,assumption:convexity_potential_above,assumption:gaussian_init} hold.
    If there exists $\hat \theta$ such that $\mathcal M(\rho_n) \xrightarrow[n \to \infty]{}~\hat \theta$,
    then  $\wass[2]{\rho_n}{\delta_{\theta_*}} \xrightarrow[n\to \infty]{} 0$.
\end{corollary}

In the one-dimensional case,
it is possible to prove the convergence of $m_n$ and $m(t)$ to the minimizer $\theta_*$ without the \emph{a priori} assumption that $m_n$ and $m(t)$ have a limit.
\begin{proposition}
    [Convergence in the one-dimensional case]
    \label{proposition:convergence_optimization}
    Let $d=1$, $\lambda=1$, $\beta>0$, $\mat C_0\in \pos^d$, and suppose that \cref{assumption:convexity_potential,assumption:convexity_potential_above} are satisfied.
    Then it holds that $m_n \xrightarrow[n \to \infty]{} \theta_*$ for $\alpha\in[0,1)$ and,
    likewise, $m(t) \xrightarrow[t \to \infty]{} \theta_*$ for $\alpha=1$.
\end{proposition}
As above, this result can be combined with \cref{lemma:collapse_optimization} to obtain
a convergence result in Euclidean Wasserstein distance for the solution to~\eqref{eq:iterative_scheme_measures} and~\eqref{eq:mean_field},
under \cref{assumption:convexity_potential,assumption:convexity_potential_above,assumption:gaussian_init}.
When deriving this convergence result, we obtain non-optimal rates of order $n^{-1/r}$ for the case $\alpha=0$,
$n^{-1/2r}$ for $\alpha\in (0,1)$
and $t^{-1/2r}$ for $\alpha = 1$, with $r=r(u,l)>2$.

To conclude this section,
we present a convergence result for $m_n$ with an explicit sharp rate in the particular case $\alpha = 0$.
\begin{proposition}[Rate of convergence]
    \label{prop:convergence_rate_1d_opti}
    Let $d=1$, $\lambda=1$, $\beta>0$, $\alpha=0$, $\mat C_0\in \pos^d$  and
    suppose that \cref{assumption:convexity_potential,assumption:convexity_potential_above} are satisfied.
    Suppose additionally that $\e^{-\beta f}$ is, together with all its derivatives, bounded from above uniformly in $\real$.
    Then there exists a positive constant $k = k(m_0, C_0)$ such that,
    for sufficiently large $n$,
    \[
        \vecnorm{ m_n - \theta_*} \leq k \left( \frac{\log n}{n} \right).
    \]
\end{proposition}
The rate of convergence obtained in \cref{prop:convergence_rate_1d_opti} is almost optimal
in view of the fact shown in \cref{sub:convergence_to_equilibrium} that $\vecnorm{ m_n - \theta_*}$ scales with $n$ as $\mathcal O(1/n)$ in the case when $f$ is quadratic.
We expect the result to extend to other values of $\alpha$ and to the continuous-time solution to~\eqref{eq:momentprop},
but we focus on the case $\alpha = 0$ in order to avoid overly lengthy and technical proofs.
We point out that, already in the Gaussian case, the argument to obtain an optimal decay rate for $\alpha \in (0, 1]$ is quite technical.
Finding a simplified argument to prove optimal rates in the optimization setting is an interesting open problem,
which we leave for future work.

\subsection{Analysis of the Sampling Scheme}%
\label{sub:analysis_of_the_sampling_scheme}

In this subsection, we investigate the existence of steady states and convergence for the mean field dynamics associated with the consensus-based samplers, that is when used with $\lambda = (1+\beta)^{-1}$.
We consider both the iteration~\eqref{eq:iterative_scheme_measures} (in the case $\alpha\in [0,1)$) and the nonlocal, nonlinear Fokker--Planck equation~\eqref{eq:mean_field} (in the case $\alpha= 1$).

We begin by stating an existence result in the multi-dimensional setting.
Since the corresponding proof is very short,
we include it in this section.
\begin{theorem}[Existence of steady states]
\label{thm:existence-sstates}
Let $\lambda=(1+\beta)^{-1}$, $\beta>0$ and $\alpha\in[0,1]$.
    Suppose \cref{assumption:convexity_potential,assumption:convexity_potential_above} are satisfied.
    Then there exists $\underline \beta$ such that,
    for all $\beta \geq \underline \beta$,
    the dynamics~\eqref{eq:iterative_scheme_measures} and \eqref{eq:mean_field} admit a Gaussian steady state $ g\bigl(\dummy; \vect m_{\infty}(\beta), \mat C_{\infty}(\beta)\bigr)$
    satisfying
    \[
        \Uhess^{-1} \preccurlyeq \mat C_{\infty}(\beta) \preccurlyeq \Lhess^{-1} \quad \text{ and } \quad \vecnorm{\vect m_{\infty}(\beta) - \theta_*} = \mathcal O\left(\frac{1}{\sqrt{\beta}}\right).
    \]
\end{theorem}
\begin{proof}
    By~\cref{lemma:basic_results_steady_states},
    a Gaussian $g(\dummy; \vect m_{\infty}, \mat C_{\infty})$ is a steady state if and only if
    \[
        \vect m_{\infty} =\vect m_\beta(\vect m_{\infty}, \mat C_{\infty})
            \qquad \text{and} \qquad
        \mat C_{\infty} = \lambda^{-1}\mat C_\beta(\vect m_{\infty}, \mat C_{\infty})\,,
    \]
    i.e.\ if and only if $\bigl(\vect m_{\infty}(\beta), \mat C_{\infty}(\beta)\bigr)$ is a fixed point of the map
    \[
        \Phi_{\beta} : (\vect m, \mat C) \mapsto \bigl(\vect m_\beta(\vect m, \mat C), (1+\beta)\mat C_\beta(\vect m, \mat C)\bigr).
    \]
    In order to prove the result,
    we show that
    $\Phi_{\beta}(S_{\beta}) \subset S_{\beta}$ for all $\beta$ sufficiently large, where
    \[
        S_{\beta} = \Bigl\{(\vect m, \mat C): \vecnorm{\vect m-\theta_*} \leq R \beta^{-1/2} ~\text{and}~ \Uhess^{-1} \preccurlyeq \mat C \preccurlyeq \Lhess^{-1} \Bigr\}
    \]
    and $R = 2k/\sqrt{\lhess}$, with $k = k(\lhess, \uhess, d)$ the constant from \cref{lemma:first_moment_several_dimensions}.
    Since $\Phi_{\beta}$ is continuous, the result then follows from Brouwer's fixed point theorem.
    By \cref{lemma:bound_second_moment,lemma:lower_bound_second_moment},
    it holds that $\Uhess^{-1} \preccurlyeq (1+\beta)\mat C_\beta(\vect m, \mat C) \preccurlyeq \Lhess^{-1}$ for any $(\vect m, \mat C) \in S_{\beta}$,
    so we have to show only that there exist $\underline \beta$ such that
    \[
       \forall \beta \geq \underline{\beta},\quad \forall (\vect m, \mat C) \in S_{\beta},  \qquad
        \vecnorm{\vect m_\beta(\vect m, \mat C) - \theta_*} \leq R \beta^{-1/2}.
    \]
    If $(\vect m, \mat C) \in S_{\beta}$,
    then by \cref{lemma:first_moment_several_dimensions} there exists $k = k(\lhess, \uhess, d)$ such that
    \begin{align*}
        \forall \beta > 0, \qquad
        \vecnorm{\vect m_\beta(\vect m, C) - \theta_*}
        &\leq \frac{R}{\beta} \sqrt{\frac{\uhess}{\lhess}} + k \left( \lhess + \beta \lhess \right)^{-1/2} \\
        &\leq R \beta^{-1/2} \left(\sqrt{\frac{\uhess}{\beta\lhess}} + \frac{k}{R} \sqrt{\frac{1}{\lhess}} \right)
        = R \beta^{-1/2} \left(\sqrt{\frac{\uhess}{\beta\lhess}} + \frac{1}{2} \right)
    \end{align*}
    from where the statement follows easily with $\underline{\beta} = \frac{4\uhess}{\lhess}$.
\end{proof}
This result shows that the sampling scheme admits a steady state whose mean is close to the minimizer of $f$ for large $\beta$,
but it does not provide much information on the covariance of the Gaussian steady state.
In the one-dimensional setting,
we can show that the steady state is in fact unique and arbitrarily close to the Laplace approximation of the target distribution
provided that $\beta$ is sufficiently large.
By the Laplace approximation $\hat \rho$ of the target distribution,
we mean the Gaussian probability distribution $g\bigl(\dummy; \theta_*, \hessian f(\theta_*)^{-1}\bigr)$,
that is
\[
    \hat \rho(\theta) := \frac{\e^{-\hat f(\theta)}}{\int_{\real^d} \e^{-\hat f(\theta)}\, \d \theta},
    \qquad
    \hat f(\theta) := f(\theta_*) + \frac{1}{2} \bigl((\theta - \theta_*) \otimes (\theta - \theta_*) \bigr) : \hessian f(\theta_*).
\]
(Note that $\hat \rho$ coincides with the target distribution when $f$ is quadratic.)
In order to establish results in the one-dimensional setting,
we make the following additional assumption on $f$.
\begin{assumption}
    \label{assumption:assumption_f}
    Let $d=1$.
    The function $f$ is smooth and,
    together with all its derivatives,
    it is bounded from above by the reciprocal of a Gaussian,
    in the sense that for all $i \in \{0, 1, \dotsc\}$ there exists $\lambda_i \in \real$ such that
    \begin{equation*}
        \norm{\e^{- \lambda_i t^2} \derivative*{i}[f]{t}(t)}[\infty] < \infty.
    \end{equation*}
\end{assumption}

We let $C_* := 1/f''(\theta_*)$ and denote by $B_R(m_*, C_*)$ the closed ball of radius $R$ around $(m_*, C_*)$.
\begin{theorem}[Convergence to the steady state]
    \label{thm:convergence}
    Let $d=1$ and $\lambda=(1+\beta)^{-1}$,
    and suppose \cref{assumption:convexity_potential,assumption:assumption_f} hold.
    For any $R\in (0,C_*)$, there exists $\underline\beta = \underline \beta(f, R)$
    and $k=k(f, R)$ such that the following statements hold for all $\beta \geq \underline \beta$:
    \begin{itemize}
        \item \textbf{Steady state.}
            There exists a pair
            \(
                \bigl(m_{\infty}(\beta), C_{\infty}(\beta)\bigr)
            \),
            unique in $B_{R}(\theta_*, C_*)$,
            such that the Gaussian density $\rho_{\infty} = g(\dummy; m_{\infty}, C_{\infty})$ satisfies~\eqref{eq:steady_state},
            and this pair satisfies
            \[
                \vecnorm{ \begin{pmatrix} m_{\infty}(\beta) \\ C_{\infty}(\beta) \end{pmatrix} - \begin{pmatrix} m_* \\ C_0 \end{pmatrix} }
                \leq \frac{k}{\beta}.
            \]
            By \cref{lemma:basic_results_steady_states},
            the density $\rho_{\infty}$ is a steady state
            of both the iterative scheme~\eqref{eq:iterative_scheme_measures} with any $\alpha \in [0, 1)$ and the nonlinear Fokker--Planck equation~\cref{eq:mean_field},
            corresponding to $\alpha=1$.

        \item \textbf{Discrete time $\alpha\in [0,1)$.}
            If~\cref{assumption:gaussian_init} holds and the moments of the initial (Gaussian) law satisfy $(m_0, C_0) \in B_R(\theta_*, C_*)$,
            then the solution to the iterative scheme~\cref{eq:iterative_scheme_measures}
            converges geometrically to the steady state~$\rho_{\infty}$
            provided that $\alpha + (1-\alpha^2)\frac{k}{\beta} < 1$.
            More precisely,
            \[
                \forall n \in \nat, \qquad
                \vecnorm{ \begin{pmatrix} m_n \\ C_n \end{pmatrix} - \begin{pmatrix} m_{\infty}(\beta) \\ C_{\infty}(\beta) \end{pmatrix} }
                \leq \left(\alpha + (1-\alpha^2)\frac{k}{\beta}\right)^n \vecnorm{ \begin{pmatrix} m_0 \\ C_0 \end{pmatrix} - \begin{pmatrix} m_{\infty}(\beta) \\ C_{\infty}(\beta) \end{pmatrix} }.
            \]
        \item \textbf{Continuous time $\alpha=1$.}
            If~\cref{assumption:gaussian_init} holds and the moments of the initial (Gaussian) law satisfy $\bigl(m_0, C_0\bigr) \in B_R(\theta_*, C_*)$,
            then the solution to the mean field Fokker Planck equation~\eqref{eq:mean_field} converges exponentially to the steady state~$\rho_{\infty}$
            provided that $1 - \frac{2k}{\beta} > 0$.
            More precisely,
            \[
                \forall t \geq 0, \qquad
                \vecnorm{ \begin{pmatrix} m(t) \\ C(t) \end{pmatrix} - \begin{pmatrix} m_{\infty}(\beta) \\ C_{\infty}(\beta) \end{pmatrix} }
                \leq \exp \left( - \left(1 - \frac{2k}{\beta} \right) t \right) \vecnorm{ \begin{pmatrix} m_0 \\ C_0 \end{pmatrix} - \begin{pmatrix} m_{\infty}(\beta) \\ C_{\infty}(\beta) \end{pmatrix} }.
            \]
    \end{itemize}
\end{theorem}

There is no conceptual obstruction to generalizing this result to the multi-dimensional setting,
but the associated calculations involving the Laplace's method,
on which the proof of \cref{thm:convergence} relies,
are significantly more technical than in the one-dimensional setting,
so we focus here on the one-dimensional case only.

\section{Numerical Experiments}%
\label{sec:numerical_experiments}
In this section,
we present numerical experiments illustrating our method.
The performance of CBS in optimization mode is studied in~\cref{sub:use_in_general_purpose_optimization}.
We then illustrate the efficacy of the method for sampling in \cref{sub:low_dimensional_parameter_space},
where a simple inverse problem with low-dimensional parameter and data is considered,
and in \cref{sub:high_dimensional_parameter_space},
where a more realistic and challenging example is examined.
Video animations associated with the numerical experiments presented in this section are freely available online~\cite{figures}.

\subsection{General-Purpose Optimization}%
\label{sub:use_in_general_purpose_optimization}

In this subsection,
we study the efficacy of our method for solving optimization problems that do not necessarily originate from a Bayesian context.
We also show empirically how the convergence of the algorithm can be improved by adapting the parameter $\beta$ appropriately during the simulation.
Throughout the subsection,
we consider the same non-convex test functions as those taken in~\cite{pinnau2017consensus}:
the translated Ackley function, defined for $x \in \real^d$ by
\begin{equation}
    \label{eq:ackley}
    f_A(x) = -20 \exp \left( - \frac{1}{5} \sqrt{\frac{1}{d} \sum_{i=1}^{d} |x_i - b|^2} \right)
    -\exp \left( \frac{1}{d} \sum_{i=1}^{d} \cos\bigl(2 \pi (x_i - b)\bigr) \right) + \e \,+\, 20,
\end{equation}
and the Rastrigin function, defined by
\begin{equation}
    \label{eq:rastrigin}
    f_R(x) =  \sum_{i=1}^{d} \Bigl((x_i - b)^2 - 10 \cos \bigl(2 \pi (x_i - b)\bigr) + 10 \Bigr).
\end{equation}
Both functions are minimized at $x_* = (b, \dotsc, b)$, where $b \in \real$ is a translation parameter.
They are depicted in \cref{fig:ackley_rastrigin}.
\begin{figure}[ht]
    \centering
    \includegraphics[width=0.49\linewidth]{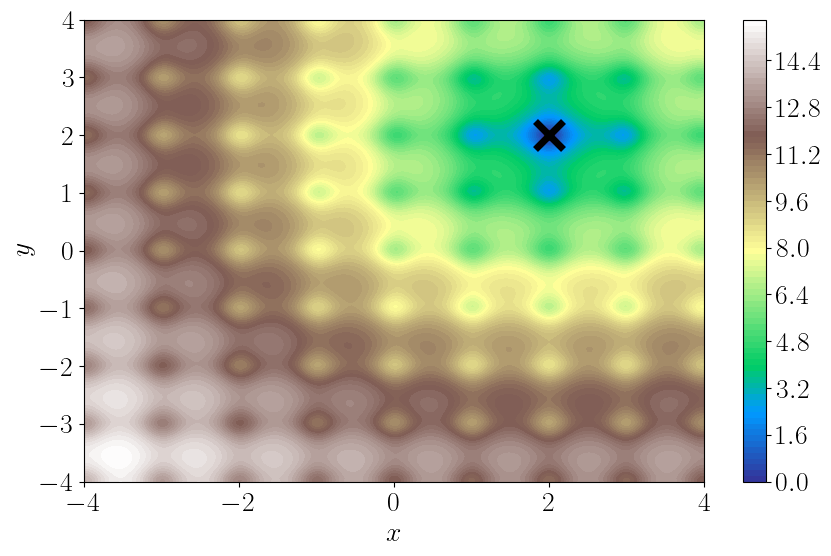}
    \includegraphics[width=0.49\linewidth]{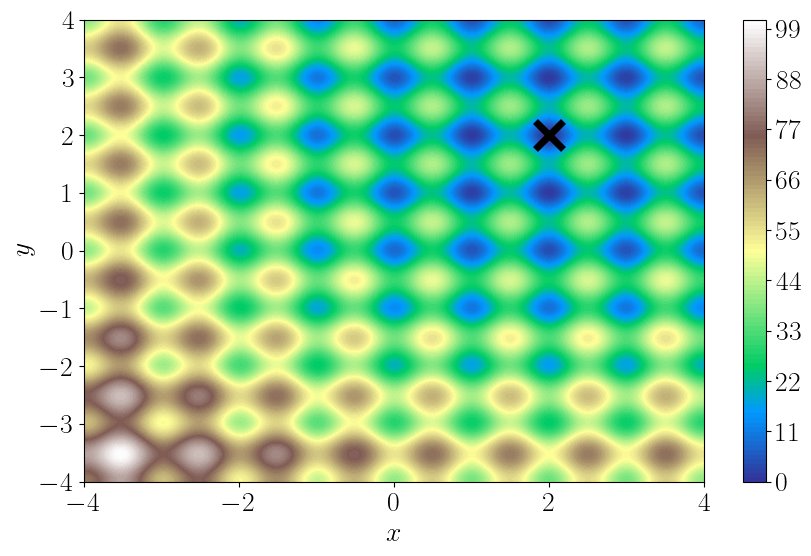}
    \caption{
        Ackley (left) and  Rastrigin (right) functions for $d = 2$ and $b = 2$;
        see~\eqref{eq:ackley} and~\eqref{eq:rastrigin}.
    }%
    \label{fig:ackley_rastrigin}
\end{figure}

In all simulations presented below, the  initial particle ensemble members are drawn independently from $\normal (0, 3 I_d)$,
and the simulation is stopped when $\vecnorm{\mathcal C(\rho^J_n)}_{\rm F} < 10^{-12}$ for the first time;
here $\rho^J_n$ denotes the empirical measure associated with the ensemble at iteration $n$.

\subsubsection{Dynamic Adaptation of \texorpdfstring{$\beta$}{beta}}%
In this paragraph, we show numerically that adapting $\beta$ dynamically during a simulation can be advantageous for convergence.
We consider the following simple adaptation scheme with parameter $\eta \in \left(\frac{1}{J}, 1\right)$:
denoting by $\{\theta^{(j)}_n\}_{j=1}^J$ the ensemble at step $n$,
the parameter $\beta$ employed for the next iteration is obtained as the positive solution to the following equation:
\begin{equation}
    \label{eq:ess}
    J_{\rm eff}(\beta) := \frac{\left( \sum_{j=1}^{J} \omega_j \right)^2}
    {\sum_{j=1}^{J} \abs{\omega_j}^2} = \eta J,
    \qquad \omega_j := \e^{- \beta f(\theta_n^{(j)})}.
\end{equation}
Employing the notation $f_j = f(\theta_n^{(j)})$,
we calculate
\[
    J_{\rm eff}'(\beta) = - 2 \beta \frac{ \left( \sum_{j=1}^{J} \omega_j \right) \left( \sum_{j=1}^{J} f_j \omega_j \right) - \left( \sum_{j=1}^{J} f_j \abs{\omega_j}^2 \right)}{\left( \sum_{j=1}^{J} \abs{\omega_j}^2 \right)^2 } \leq 0,
\]
so $J_{\rm eff}$ is a continuous, non-increasing function with $J_{\rm eff}(0) = J$ and $\lim_{\beta \to \infty} J_{\rm eff}(\beta) = 1$.
Consequently, equation~\eqref{eq:ess} admits a unique solution in $(0, \infty)$.
The left-hand side of~\eqref{eq:ess} is known in statistics as an \emph{effective sample size},
which motivates the notation $J_{\rm eff}$.
When this approach is employed, the parameter $\beta$ is generally small in the early stage of the simulation as long as the initial ensemble has large enough spread,
and it increases progressively as the simulation advances and the ensemble spread decreases. In other words, this cooling schedule for $\beta$ ensures that roughly always the same proportion $\eta$ of particles contribute to the weighted sums in the scheme.
This adaptation approach is useful for a two primary reasons:
\begin{itemize}
    \item
        On the one hand, provided that $\eta$ and $J$ are sufficiently large,
        adapting $\beta$ according to~\eqref{eq:ess} ensures that situations where the ensemble quickly collapses to a very narrow distribution do not arise.  An early collapse of the ensemble is not desirable as the scheme may then get stuck in local minima of the objective function $f$, or in the case when the collapse is not complete, the convergence is slowed down considerably.
        This issue is especially critical when the scheme~\eqref{eq:interacting_particle_system} is employed with $\alpha = 0$:
        in this case, if $\beta$ is not sufficiently small at the beginning of the simulation,
        it is often the case that the weighted covariance of the initial ensemble is very close to zero,
        in which case the ensemble collapses nearly to a point in a single step.
    \item
        On the other hand, increasing $\beta$ in the later stage of the simulation significantly accelerates convergence to the minimizer.
        Indeed, when a fixed value of $\beta$ is employed,
        the weights~$\{\omega_j\}_{j=1}^{J}$ all converge to the same value as the simulation progresses and the ensemble collapses,
        and so the influence of the objective function on the dynamics diminishes.
        By increasing $\beta$ dynamically, we strengthen the bias of the dynamics towards areas of small $f$,
        thereby accelerating convergence.
\end{itemize}
In the remainder of this section,
we consider for simplicity only the choice $\eta = \frac{1}{2}$.
A more detailed analysis of the efficiency of this approach,
through both theoretical and numerical means, is left for future work. More generally, an interesting open question is whether it is possible to determine an optimal cooling schedule for $\beta$ taking the above considerations into account.
We illustrate in~\cref{tab:adaptation_beta} the performance of CBS in optimization mode,
with both fixed and adaptive $\beta$, for finding the minimizer of the Ackley function with $b = 0$ in dimension~2.
The data presented in each cell are calculated from 100 independent runs of the method.
For all the values of $J$ and $\alpha$ considered,
using the adaptive strategy based on~\eqref{eq:ess} provides a significant advantage,
in terms of both the number of iterations required for convergence and the accuracy of the approximate minimizer.

\begin{table}[!ht]
    \centering
    \begin{tabular}{c|c|c|c|c}
        Adapt? & $\alpha$ & $J=50$ & $J=100$ & $J=200$ \\
        \hline
        no & $0$
           & $ 100 \% \,|\, 511 \,|\, 8.73 \times 10^{-3} $
           & $ 100 \% \,|\, 966 \,|\, 4.34 \times 10^{-3} $
           & $ 100 \% \,|\, 1767 \,|\, 2.5 \times 10^{-3} $
           \\
        no & $.5$
           & $ 100 \% \,|\, 611 \,|\, 1.22 \times 10^{-2} $
           & $ 100 \% \,|\, 1191 \,|\, 6.87 \times 10^{-3} $
           & $ 100 \% \,|\, 2141 \,|\, 3.38 \times 10^{-3} $
           \\
        no & $.9$
           & $ 100 \% \,|\, 2028 \,|\, 1.6 \times 10^{-2} $
           & $ 100 \% \,|\, 3693 \,|\, 8.31 \times 10^{-3} $
           & $ 100 \% \,|\, 7259 \,|\, 5.22 \times 10^{-3} $
           \\
           \hline
        yes & $0$
            & $ 100 \% \,|\, 31 \,|\, 1.86 \times 10^{-7} $
            & $ 100 \% \,|\, 31 \,|\, 1.09 \times 10^{-7} $
            & $ 100 \% \,|\, 31 \,|\, 8.44 \times 10^{-8} $
            \\
        yes & $.5$
            & $ 100 \% \,|\, 49 \,|\, 2.86 \times 10^{-7} $
            & $ 100 \% \,|\, 48 \,|\, 2.0 \times 10^{-7} $
            & $ 100 \% \,|\, 48 \,|\, 1.43 \times 10^{-7} $
            \\
        yes & $.9$
            & $ 100 \% \,|\, 251 \,|\, 2.27 \times 10^{-6} $
            & $ 100 \% \,|\, 242 \,|\, 4.36 \times 10^{-7} $
            & $ 100 \% \,|\, 238 \,|\, 2.87 \times 10^{-7} $
            \\
    \end{tabular}
    \caption{%
        Performance of the CBS in optimization mode for the Ackley function in spatial dimension $d = 2$,
        without and with adaptive $\beta$.
        The three data presented in each cell are respectively
        the success rate of the method,
        the average number of iteration until the stopping criterion is met,
        and the average (over the successful runs) error at the final iteration,
        computed as the infinity norm between the minimizer and the ensemble mean.
        Our definition of the success rate is very similar to that used in~\cite{pinnau2017consensus}:
        a run is considered successful if the ensemble mean is within .25,
        in infinity norm, of the minimizer at the final iteration.
    }
    \label{tab:adaptation_beta}
\end{table}

\subsubsection{Low-dimensional Optimization Problem: \texorpdfstring{$d = 2$}{}}

The performance of CBS in optimization mode is illustrated in~\cref{tab:cbs_ackley_2d,tab:cbs_rastrigin_2d},
for the Ackley and Rastrigin functions respectively, in spatial dimension $d = 2$.
We make a few observations:
\begin{itemize}
    \item
        \emph{Influence of $\alpha$}:
        The simulations corresponding to $\alpha = 0$ consistently require fewer iterations to converge than those corresponding to $\alpha = \frac{1}{2}$,
        and they have a better success rate for the Rastrigin function.
    \item
        \emph{Influence of $J$}:
        For the Rastrigin function,
        a high number of particles, i.e.\ a large value of $J$,
        correlates with a better success rate.
        With only 50 particles, the method often converges to the wrong local minimizer,
        but with 200 particles the ensemble almost always collapses at the global minimizer.
    \item
        \emph{Influence of $b$}:
        For the Rastrigin function,
        a low value of $b$ correlates with better performance.
        This behavior, which was observed also for CBO in~\cite{pinnau2017consensus},
        is not surprising because, when $b = 0$,
        the minimizer is centered with respect to the initial ensemble.
\end{itemize}
We also note that, like CBO~\cite{pinnau2017consensus},
our method performs markedly better for the Ackley function than for the Rastrigin function.
Snapshots of the particles are presented in \cref{fig:convergence_ackley_rastrigin} for the parameters $\alpha = 0$ and $J = 100$.

\begin{table}[htpb]
    \centering
    \begin{tabular}{c|c|c|c|c}
         $b$ & $\alpha$ & $J=50$ & $J=100$ & $J=200$ \\
        \hline
        $0$ & $0$
        & $ 100 \% \,|\, 31 \,|\, 1.86 \times 10^{-7} $
        & $ 100 \% \,|\, 31 \,|\, 1.09 \times 10^{-7} $
        & $ 100 \% \,|\, 31 \,|\, 8.44 \times 10^{-8} $
        \\
        $0$ & $.5$
        & $ 100 \% \,|\, 49 \,|\, 2.86 \times 10^{-7} $
        & $ 100 \% \,|\, 48 \,|\, 2.0 \times 10^{-7} $
        & $ 100 \% \,|\, 48 \,|\, 1.43 \times 10^{-7} $
        \\
        \hline
        $1$ & $0$
        & $ 100 \% \,|\, 31 \,|\, 1.83 \times 10^{-7} $
        & $ 100 \% \,|\, 31 \,|\, 1.16 \times 10^{-7} $
        & $ 100 \% \,|\, 31 \,|\, 7.91 \times 10^{-8} $
        \\
        $1$ & $.5$
        & $ 100 \% \,|\, 49 \,|\, 3.23 \times 10^{-7} $
        & $ 100 \% \,|\, 49 \,|\, 2.05 \times 10^{-7} $
        & $ 100 \% \,|\, 49 \,|\, 1.47 \times 10^{-7} $
        \\
        \hline
        $2$ & $0$
        & $ 100 \% \,|\, 31 \,|\, 1.86 \times 10^{-7} $
        & $ 100 \% \,|\, 32 \,|\, 1.1 \times 10^{-7} $
        & $ 100 \% \,|\, 32 \,|\, 8.61 \times 10^{-8} $
        \\
        $2$ & $.5$
        & $ 100 \% \,|\, 51 \,|\, 3.03 \times 10^{-7} $
        & $ 100 \% \,|\, 50 \,|\, 1.92 \times 10^{-7} $
        & $ 100 \% \,|\, 50 \,|\, 1.38 \times 10^{-7} $
        \\
    \end{tabular}
    \caption{%
        Performance of the CBS in optimization mode for the Ackley function in spatial dimension $d = 2$.
        See the caption of \cref{tab:adaptation_beta} for a description of the data presented.
    }
    \label{tab:cbs_ackley_2d}
\end{table}

\begin{table}[htpb]
    \centering
    \begin{tabular}{c|c|c|c|c}
         $b$ & $\alpha$ & $J=50$ & $J=100$ & $J=200$ \\
        \hline
        $0$ & $0$
        & $ 83 \% \,|\, 41 \,|\, 1.73 \times 10^{-7} $
        & $ 99 \% \,|\, 45 \,|\, 1.19 \times 10^{-7} $
        & $ 100 \% \,|\, 45 \,|\, 8.43 \times 10^{-8} $
        \\
        $0$ & $.5$
        & $ 77 \% \,|\, 74 \,|\, 3.39 \times 10^{-4} $
        & $ 98 \% \,|\, 69 \,|\, 2.21 \times 10^{-7} $
        & $ 100 \% \,|\, 66 \,|\, 1.56 \times 10^{-7} $
        \\
        \hline
        $1$ & $0$
        & $ 84 \% \,|\, 42 \,|\, 1.85 \times 10^{-7} $
        & $ 99 \% \,|\, 44 \,|\, 1.03 \times 10^{-7} $
        & $ 100 \% \,|\, 45 \,|\, 7.8 \times 10^{-8} $
        \\
        $1$ & $.5$
        & $ 72 \% \,|\, 68 \,|\, 6.03 \times 10^{-7} $
        & $ 91 \% \,|\, 68 \,|\, 2.23 \times 10^{-7} $
        & $ 100 \% \,|\, 68 \,|\, 1.56 \times 10^{-7} $
        \\
        \hline
        $2$ & $0$
        & $ 79 \% \,|\, 42 \,|\, 1.84 \times 10^{-7} $
        & $ 96 \% \,|\, 44 \,|\, 1.12 \times 10^{-7} $
        & $ 100 \% \,|\, 45 \,|\, 7.78 \times 10^{-8} $
        \\
        $2$ & $.5$
        & $ 58 \% \,|\, 80 \,|\, 4.14 \times 10^{-4} $
        & $ 74 \% \,|\, 75 \,|\, 3.52 \times 10^{-5} $
        & $ 96 \% \,|\, 74 \,|\, 1.54 \times 10^{-7} $
        \\
    \end{tabular}
    \caption{%
        Performance of the CBS in optimization mode for the Rastrigin function in spatial dimension $d = 2$.
        See the caption of \cref{tab:adaptation_beta} for a description of the data presented.
    }
    \label{tab:cbs_rastrigin_2d}
\end{table}

\begin{figure}[htpb]
    \centering
    \includegraphics[width=0.45\linewidth]{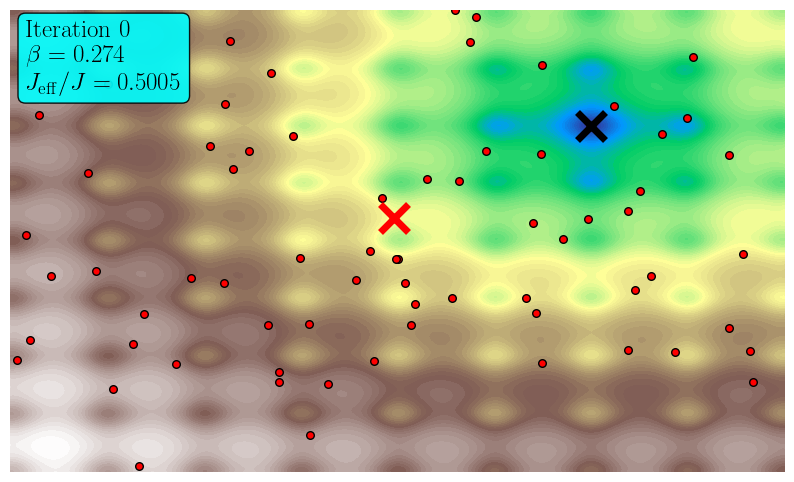}
    \includegraphics[width=0.45\linewidth]{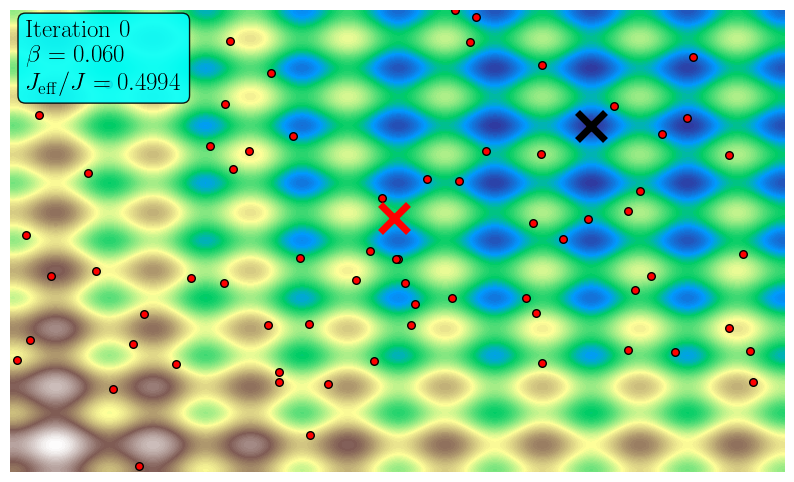}
    \includegraphics[width=0.45\linewidth]{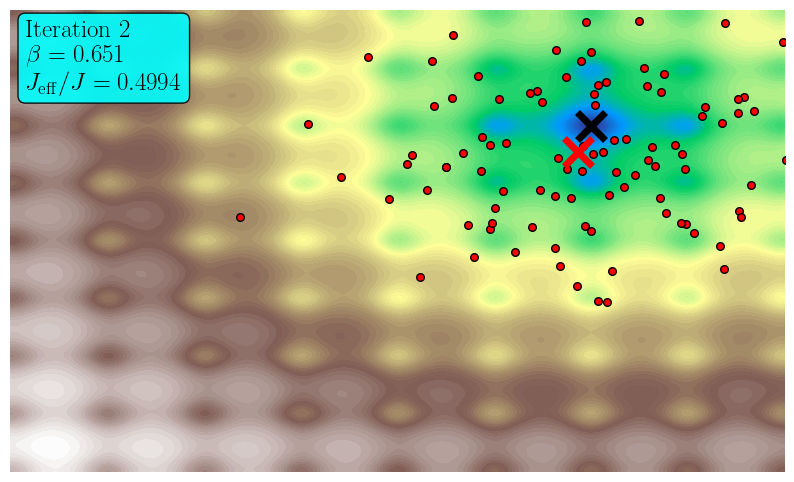}
    \includegraphics[width=0.45\linewidth]{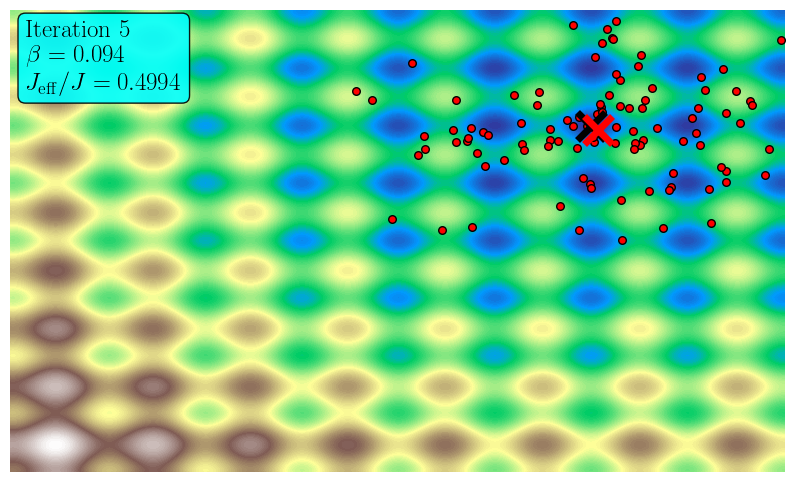}
    \includegraphics[width=0.45\linewidth]{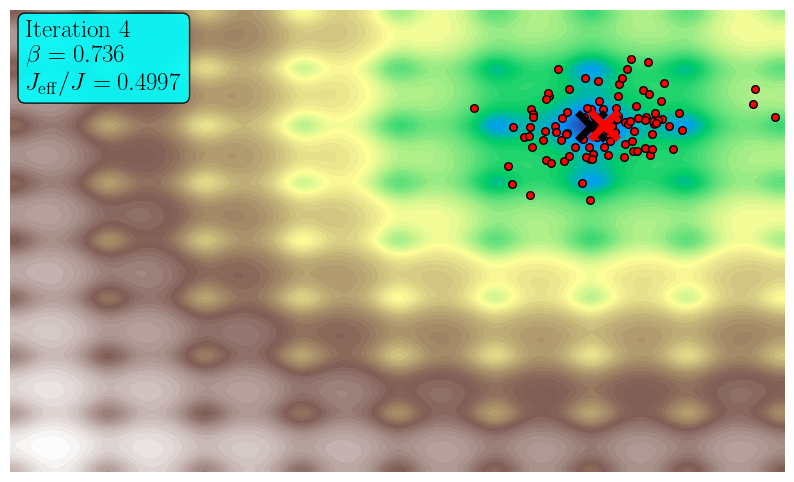}
    \includegraphics[width=0.45\linewidth]{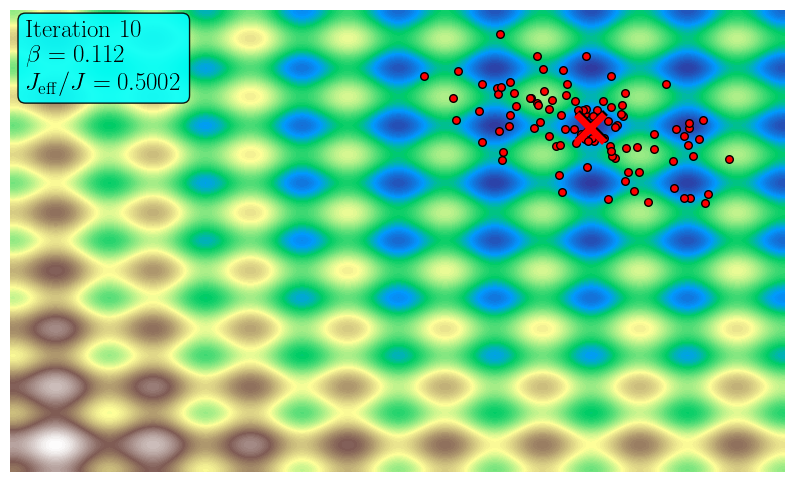}
    \includegraphics[width=0.45\linewidth]{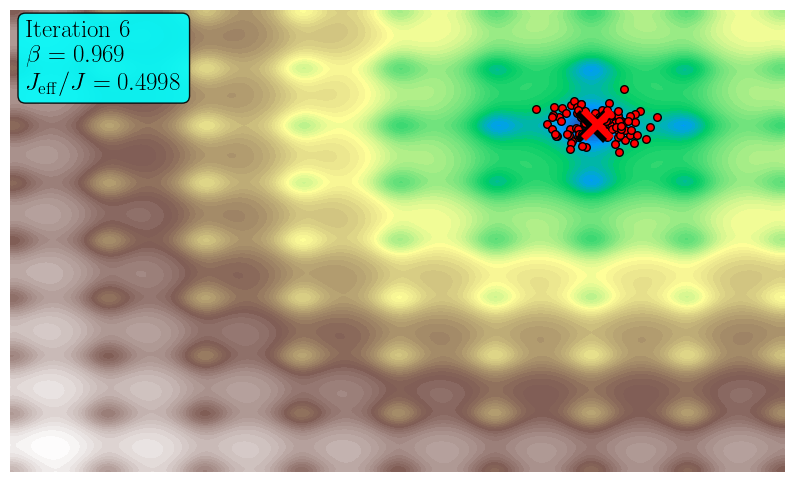}
    \includegraphics[width=0.45\linewidth]{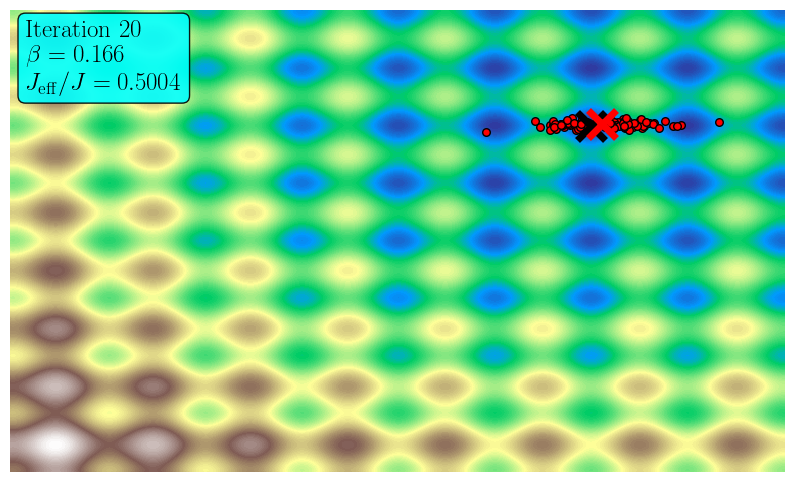}
    \includegraphics[width=0.45\linewidth]{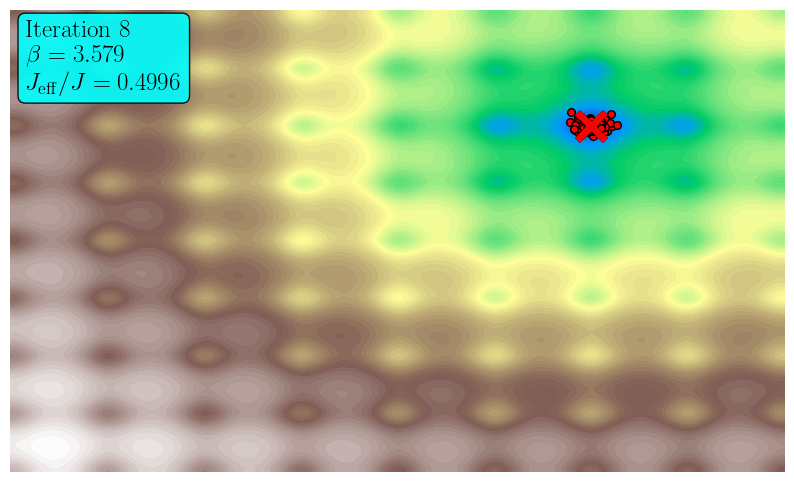}
    \includegraphics[width=0.45\linewidth]{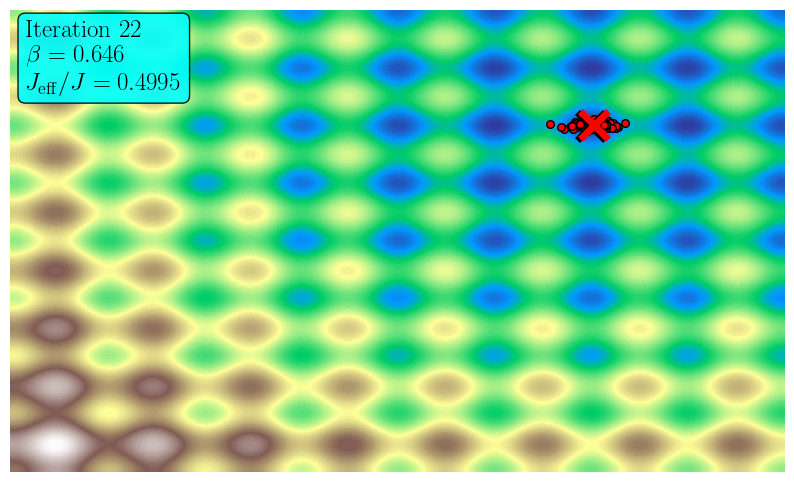}
    \caption{
        Illustration of the convergence of CBS in optimization mode for the Ackley (left) and Rastrigin (right) functions in dimension 2,
        for the parameters $J = 100$, $\alpha = 0$ and with adaptive~$\beta$.
        The black cross denotes the unique global minimizer,
        and the red cross shows the ensemble mean.
    }%
    \label{fig:convergence_ackley_rastrigin}
\end{figure}

\subsubsection{Higher-Dimensional Optimization Problem: \texorpdfstring{$d = 10$}{}}

In this paragraph,
we repeat the numerical experiments of the previous section in higher dimension $d = 10$.
We employ an adaptive $\beta$ in all the simulations,
as this approach was shown in the previous subsection to perform much better.
The associated results are presented in \cref{tab:ackley_high_dim,tab:rastrigin_high_dim},
which show that the method performs better for small $\alpha$ and large $J$ for this case as well.
Overall, the method seems to require a larger ensemble size than CBO in order to guarantee a similar success rate.
A fair comparison of the computational expenses required by both methods is difficult, however,
because the number of time steps employed in CBO is not documented in~\cite{pinnau2017consensus}.

\begin{table}[ht]
    \centering
    \begin{tabular}{c|c|c|c|c}
         $b$ & $\alpha$ & $J=100$ & $J=500$ & $J=1000$ \\
        \hline
        $0$ & $0$
        & $ 100 \% \,|\, 95 \,|\, 4.19 \times 10^{-4} $
        & $ 100 \% \,|\, 77 \,|\, 9.81 \times 10^{-8} $
        & $ 100 \% \,|\, 78 \,|\, 6.97 \times 10^{-8} $
        \\
        $0$ & $.5$
        & $ 100 \% \,|\, 248 \,|\, 1.27 \times 10^{-2} $
        & $ 100 \% \,|\, 109 \,|\, 1.71 \times 10^{-7} $
        & $ 100 \% \,|\, 110 \,|\, 1.13 \times 10^{-7} $
        \\
        \hline
        $1$ & $0$
        & $ 100 \% \,|\, 100 \,|\, 1.34 \times 10^{-3} $
        & $ 100 \% \,|\, 78 \,|\, 1.04 \times 10^{-7} $
        & $ 100 \% \,|\, 78 \,|\, 6.79 \times 10^{-8} $
        \\
        $1$ & $.5$
        & $ 98 \% \,|\, 278 \,|\, 3.27 \times 10^{-2} $
        & $ 100 \% \,|\, 111 \,|\, 1.72 \times 10^{-7} $
        & $ 100 \% \,|\, 111 \,|\, 1.13 \times 10^{-7} $
        \\
        \hline
        $2$ & $0$
        & $ 98 \% \,|\, 125 \,|\, 7.72 \times 10^{-3} $
        & $ 100 \% \,|\, 78 \,|\, 9.71 \times 10^{-8} $
        & $ 100 \% \,|\, 79 \,|\, 6.85 \times 10^{-8} $
        \\
        $2$ & $.5$
        & $ 65 \% \,|\, 306 \,|\, 6.53 \times 10^{-2} $
        & $ 100 \% \,|\, 113 \,|\, 1.7 \times 10^{-7} $
        & $ 100 \% \,|\, 113 \,|\, 1.13 \times 10^{-7} $
        \\
    \end{tabular}
    \caption{%
        Performance of the CBS in optimization mode for the Ackley function in dimension~10.
        See the caption of \cref{tab:adaptation_beta} for a description of the data presented.
    }
    \label{tab:ackley_high_dim}
\end{table}

\begin{table}[htpb]
    \centering
    \begin{tabular}{c|c|c|c|c}
         $b$ & $\alpha$ & $J=100$ & $J=500$ & $J=1000$ \\
        \hline
        $0$ & $0$
        & $ 6 \% \,|\, 222 \,|\, 2.1 \times 10^{-2} $
        & $ 95 \% \,|\, 107 \,|\, 9.69 \times 10^{-8} $
        & $ 100 \% \,|\, 111 \,|\, 6.62 \times 10^{-8} $
        \\
        $0$ & $.5$
        & $ 10 \% \,|\, 331 \,|\, 6.68 \times 10^{-2} $
        & $ 99 \% \,|\, 150 \,|\, 1.88 \times 10^{-7} $
        & $ 100 \% \,|\, 155 \,|\, 1.14 \times 10^{-7} $
        \\
        \hline
        $1$ & $0$
        & $ 4 \% \,|\, 224 \,|\, 4.61 \times 10^{-2} $
        & $ 94 \% \,|\, 108 \,|\, 9.66 \times 10^{-8} $
        & $ 100 \% \,|\, 111 \,|\, 6.97 \times 10^{-8} $
        \\
        $1$ & $.5$
        & $ 0 \% \,|\, 334 \,|\, - $
        & $ 74 \% \,|\, 165 \,|\, 5.75 \times 10^{-7} $
        & $ 99 \% \,|\, 162 \,|\, 1.18 \times 10^{-7} $
        \\
        \hline
        $2$ & $0$
        & $ 0 \% \,|\, 224 \,|\, - $
        & $ 74 \% \,|\, 113 \,|\, 9.82 \times 10^{-8} $
        & $ 99 \% \,|\, 114 \,|\, 7.07 \times 10^{-8} $
        \\
        $2$ & $.5$
        & $ 0 \% \,|\, 333 \,|\, - $
        & $ 19 \% \,|\, 190 \,|\, 1.17 \times 10^{-4} $
        & $ 69 \% \,|\, 189 \,|\, 1.24 \times 10^{-7} $
        \\
    \end{tabular}
    \caption{
        Performance of the CBS in optimization mode for the Rastrigin function in dimension~10.
        See the caption of \cref{tab:adaptation_beta} for a description of the data presented.
    }
    \label{tab:rastrigin_high_dim}
\end{table}

\subsection{Sampling: Low-Dimensional Parameter Space}%
\label{sub:low_dimensional_parameter_space}

We first consider an inverse problem with low-dimensional parameter space that
was first presented in~\cite{MR3400030} and later employed as a test problem in~\cite{HV18,garbuno2020interacting}.
In this problem,
the forward model maps the unknown $(u_1, u_2) \in \real^2$ to the observation $\bigl(p(x_1), p(x_2)\bigr) \in \real^2$,
where $x_1 = 0.25$ and $x_2 = 0.75$ and
where $p(x)$ denotes the solution to the boundary value problem
\begin{equation}
    \label{eq:boundary_value_problem}%
    -\e^{u_1} \derivative*{2}[p]{x} = 1, \qquad x \in [0, 1],
\end{equation}
with boundary conditions $p(0) = 0$ and $p(1) = u_2$.
This problem admits the following explicit solution~\cite{HV18}:
\[
    p(x) = u_2 x + \e^{-u_1} \left( - \frac{x^2}{2} + \frac{x}{2} \right).
\]
We employ the same parameters as in~\cite{garbuno2020interacting}:
the prior distribution is $\normal(0, \sigma^2 {I_2})$ with $\sigma = 10$,
and the noise distribution is $\normal(0, \gamma^2 {I_2})$ with $\gamma = 0.1$.
The observed data is $y = (27.5, 79.7)$.

We now investigate the efficiency of~\eqref{eq:interacting_particle_system} for sampling from the posterior distribution.
To this end, we use the parameters $\alpha = \beta = \frac{1}{2}$ and $J = 1000$ particles.
The ensemble after 100 iterations is depicted in~\cref{fig:1d_elliptic_optimization_noise}, together with the true posterior.
It appears from the figure that the Gaussian approximation of the posterior provided by scheme~\eqref{eq:interacting_particle_system} is close to the true posterior,
and indeed we can verify that the mean and covariance of the true and approximate posterior distributions,
which are given respectively by
\[
    m_p =
    \begin{pmatrix}
        -2.714...  \\
        104.346...
    \end{pmatrix}
    \quad
    C_p =
    \begin{pmatrix}
        0.0129... & 0.0288... \\
        0.0288... & 0.0808...
    \end{pmatrix}
\]
and
\[
    \widetilde{m}_p
    \begin{pmatrix}
        -2.712...  \\
        104.356...
    \end{pmatrix}
    \quad
    \widetilde{C}_p =
    \begin{pmatrix}
        0.0135... & 0.0302... \\
        0.0302... & 0.0829...
    \end{pmatrix},
\]
are fairly close.
\begin{figure}[ht]
    \centering
    \includegraphics[height=0.36\linewidth]{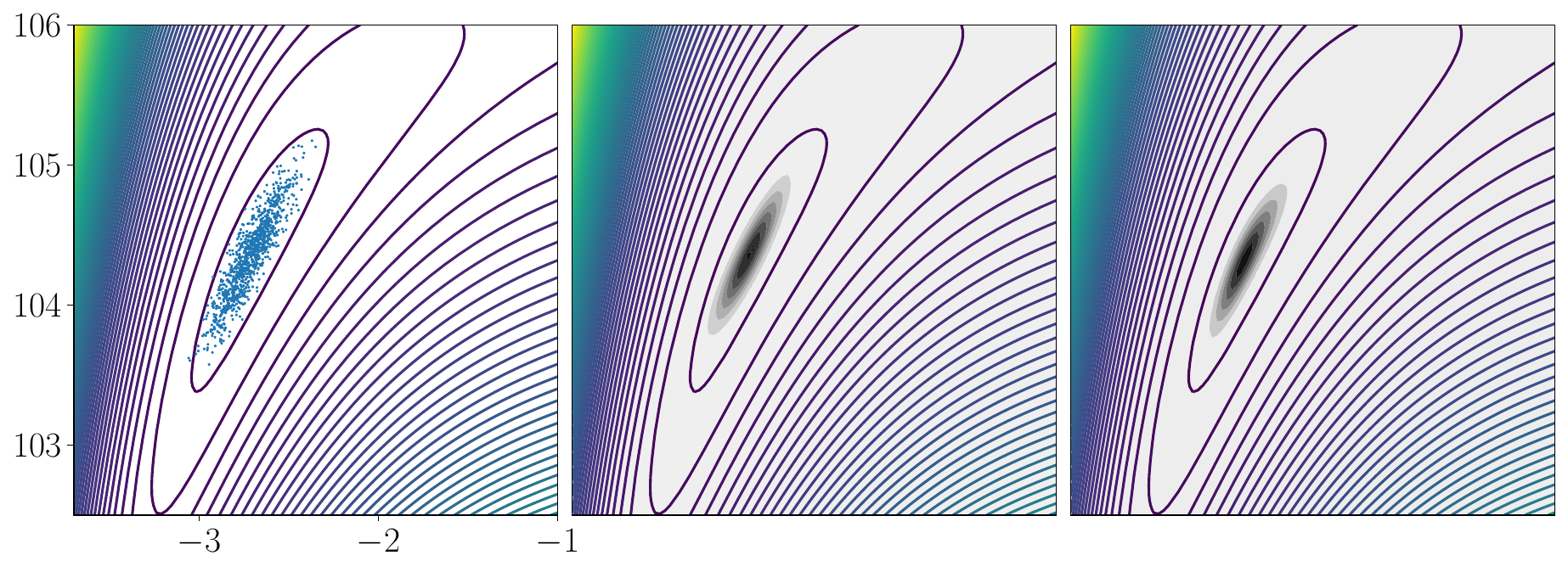}
    \caption{%
    {\bf Left:} Particles at iteration $n = 100$ for fixed $\alpha = \beta = \frac{1}{2}$.
    {\bf Middle:} Gaussian density with the same mean and covariance as the empirical distribution associated with these particles.
    {\bf Right:} True Bayesian posterior.
    }%
    \label{fig:1d_elliptic_optimization_noise}
\end{figure}

\subsection{Sampling: Higher-Dimensional Parameter Space}%
\label{sub:high_dimensional_parameter_space}

In this section,
we consider the more challenging inverse problem of finding the permeability field of a porous medium from noisy pressure measurements in a Darcy flow; for other methods applied to this problem, see~\cite{dashti2013map,garbuno2020interacting,pavliotis21derivative}.
Assuming Dirichlet boundary conditions and scalar permeability for simplicity, we consider the forward model mapping the logarithm of the permeability, denoted by $a(x)$,
to the solution of the PDE
\begin{subequations}
\label{eq:pde_darcy}
\begin{align}
    - \grad \cdot \big( \e^{a(x)} \grad p(x) \big) &= f(x), \qquad &x \in D, \\
    p(x) &= 0, &x \in \partial D.
\end{align}
\end{subequations}
Here $D = [0, 1]^2$ is the domain and $f(x) = 50$ represents a source of fluid.
We assume that noisy pointwise measurements of $p(x)$
are taken at a finite number equispaced points in $D$, given by
\begin{equation*}
    \label{eq:data_points}
    x_{ij} = \left( \frac{i}{M}, \frac{j}{M} \right), \qquad 1 \leq i, j \leq M - 1,
\end{equation*}
and that these measurements are perturbed by Gaussian noise with distribution $\mathcal N(0, \gamma^2 {I_K})$,
where $\gamma = 0.01$ and $K = (M-1)^2$.
For the prior distribution, we employ a Gaussian measure on $\lp{2}{D}$ with mean zero and precision (inverse covariance) operator given by
\[
    \mathcal C^{-1} = (- \laplacian + \tau^2 \mathcal I)^{r},
\]
equipped with Neumann boundary conditions on the space of mean-zero functions.
Here $r$ and $\tau$ are parameters controlling the smoothness and characteristic inverse length scale of samples drawn from the prior,
respectively.
The eigenfunctions and eigenvalues of the covariance operator are
\[
    \psi_{\ell}(x) = \cos \bigl( \pi (\ell_1 x_1 + \ell_2 x_2)\bigr),
    \qquad \lambda_{\ell} = \left(\pi^2 \vecnorm{\ell}^2 + \tau^2\right)^{-r}, \qquad \ell \in \nat^2.
\]
By the Karhunen--Loève (KL) expansion~\cite{pavliotis2011applied},
it holds for any $a \sim \mathcal N(0, C)$ that
\begin{equation}
    \label{eq:karhunen-loeve}
    a(x)  = \sum_{\ell \in \nat^2} \ip{a}{\psi_{\ell}} \psi_{\ell}(x)
    =: \sum_{\ell \in \nat^2} \sqrt{\lambda_{\ell}} \, \theta_{\ell} \, \psi_{\ell}(x)\,,
\end{equation}
for independent coefficients $\theta_{\ell} \sim \mathcal N(0, 1)$, and where $\ip{\dummy}{\dummy}$ denotes the $L^2$-inner product.

In order to approach the problem numerically,
we take as object of inference a finite number of terms $\{\theta_{\ell}\}_{|\ell|_{\infty} \leq N}$ in the KL expansion of the log-permeability,
which may be ordered as a linear vector given an ordering of $\{0, \dotsc, N\}^2$.
The associated prior distribution is given by the finite-dimensional Gaussian $\mathcal N(0, I_d)$, where $d = (N + 1)^2$.
At the numerical level, the forward model is evaluated as follows:
for a given vector of coefficients $\theta \in \real^{d}$,
a log-permeability field is calculated by summation as $a(\dummy; \theta) := \sum_{\vecnorm{\ell}_{\infty} \leq N} \sqrt{\lambda_{\ell}} \, \theta_{\ell} \, \psi_{\ell}(\dummy)$,
and the corresponding solution to~\eqref{eq:pde_darcy} is approximated with a finite element method (FEM).
Linear shape functions over a regular mesh with 100 subdivisions per direction are employed for the finite element solution.

For the numerical experiments presented below,
a true value $\theta^{\dagger} \in \real^{d}$ for the vector of coefficients is drawn from $\mathcal N(0, I_d)$ and
employed in order to construct the true permeability field
which, in turn, is used with the FEM described above in order to generate the data.
In particular, we employ only $(N+1)^2$ terms in the KL expansion of the true permeability.
We note that, with this approach, the resulting random field
should be viewed only as an approximate sample from $\mathcal N(0, \mathcal C)$.
Our aim is to study the performance of CBS,
not the effect of FEM discretization and truncation of the KL series on the solution of the inverse problem.

The ensemble obtained after 100 iterations of CBS with adaptive $\beta$, with $\alpha = 0$ and with $J = 512$
is depicted in~\cref{fig:log_of_permeablitiy_sampling},
along with the marginals of the Gaussian distribution with the same first and second moments as the empirical measure associated with the ensemble.
The particles forming the initial ensembles were drawn independently from $\normal(0, 9 I_d)$.
In order to validate our results,
we use as point of reference the solution provided by the ensemble Kalman sampling method~\cite{garbuno2020interacting},
combined with the adaptive time-stepping scheme from~\cite{Kovachki_2019}.
It appears from the simulations that the agreement between the posterior distribution obtained by CBS and that obtained by ensemble Kalman sampling is very good,
and both approximate posteriors are in good agreement with the true solution.

Using the final ensemble as initial condition for~\eqref{eq:interacting_particle_system} in optimization mode,
and running 50 more iterations of the algorithm, one obtains an approximation of the MAP estimator,
whose associated permeability field is illustrated in \cref{fig:log_of_permeablitiy}.
Here we use as point of comparison the solution provided by the ensemble Kalman inversion approach~\cite{iglesias2013ensemble}.
We present below the values of the first 9 Karhunen--Loève coefficients of (i) the true permeability,
(ii) the MAP estimator obtained by CBS,
and (iii) the MAP estimator obtained by ensemble Kalman inversion:
\begin{align*}
    (u^\dagger)^\t &=
    \begin{pmatrix}
        1.19 &  -2.52 &    2.07 &  -0.97 &  -0.10 &  -1.54 &    0.10 &  -0.00 &    1.01 & \dots
             % &    0.84  & 1.16 &  -0.48 &  -0.24 &    0.07 &  -0.51 & -1.58
    \end{pmatrix}, \\
    (u_{\rm MAP}^{\rm CBS})^\t &=
    \begin{pmatrix}
        1.17 &  -2.48 &    2.04 &  -0.73 &  -0.23 &  -1.65 &  -0.22 &  -0.02 &    0.23 & \dots
             % &    1.27 & 1.23 &    0.19 &  -0.84 &    0.72 &  -0.49 &  -0.37
    \end{pmatrix}, \\
    (u_{\rm MAP}^{\rm EKI})^\t &=
    \begin{pmatrix}
        1.17 &  -2.48 &    2.04 &  -0.73 &  -0.23 &  -1.65 &  -0.23 &  -0.02 &    0.24 & \dots
             % &    1.27 & 1.24 &    0.20 &  -0.83 &    0.73 &  -0.49 &  -0.36
    \end{pmatrix}.
\end{align*}
(All the numbers displayed here were rounded to two decimals.)
The agreement between the MAP estimators as approximated by ensemble Kalman inversion and by our method is very good,
and both vectors are close to the KL series of the logarithm of the true permeability.

\begin{figure}[ht]
    \centering
    \includegraphics[width=\linewidth]{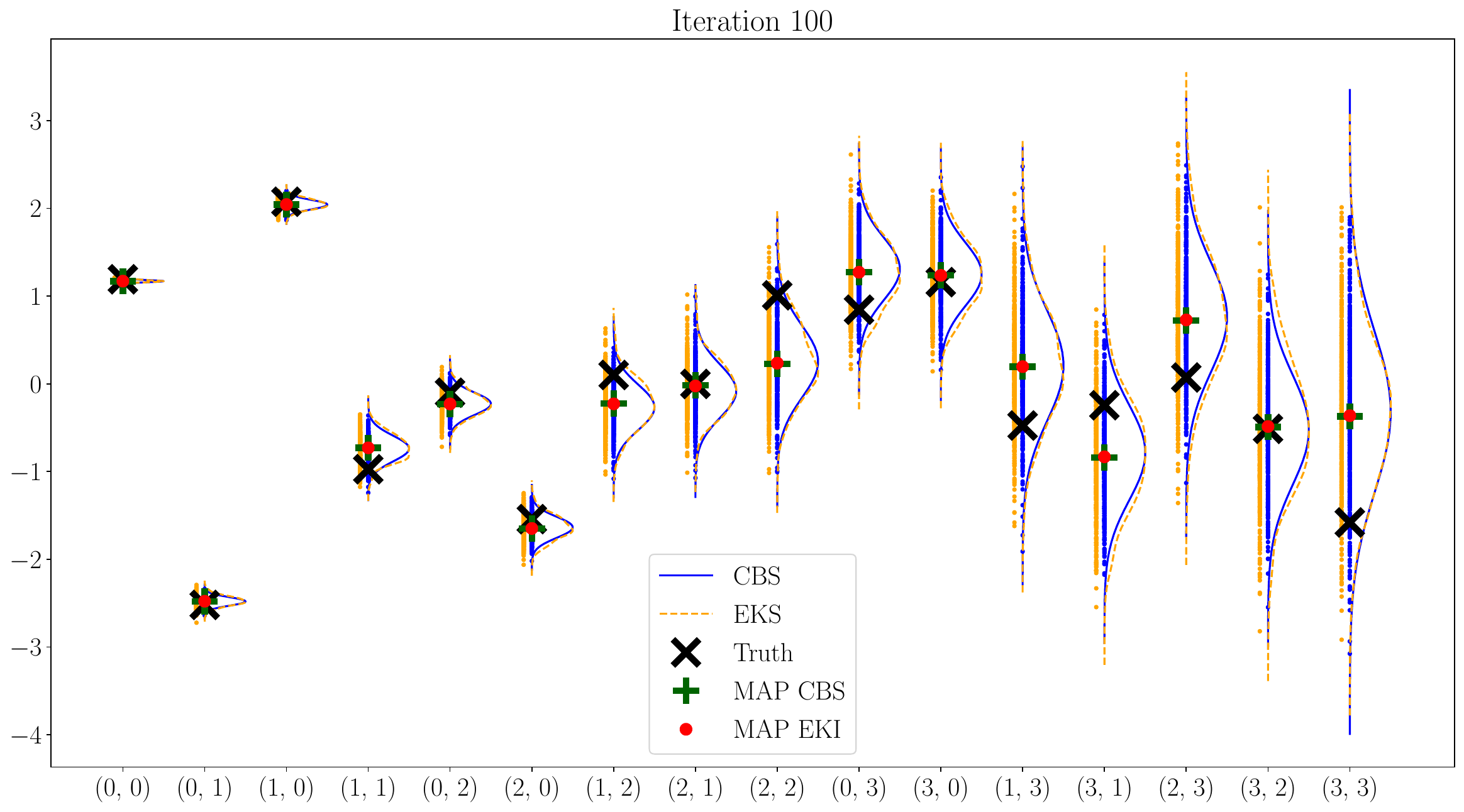}
    \caption{%
        Approximate posterior samples produced by~\eqref{eq:interacting_particle_system} with $\alpha = 0$ and adaptive $\beta$.
        Here, the labels on the $x$-axis denote the multi-indices associated with the KL coefficients of the permeability.
        The (non-normalized) solid curves represent the marginals of the Gaussian distribution whose mean and covariance are calculated from the samples produced by CBS.
        The (non-normalized) dashed curves are the marginal distributions obtained by kernel density estimation using Gaussian kernels from the samples produced by ensemble Kalman sampling~\cite{garbuno2020interacting}.
        The black crosses denote the true values of the KL coefficients, i.e.\ the values employed to generate the data.
    }%
    \label{fig:log_of_permeablitiy_sampling}
\end{figure}

\begin{figure}[ht!]
    \centering
    \includegraphics[width=\linewidth]{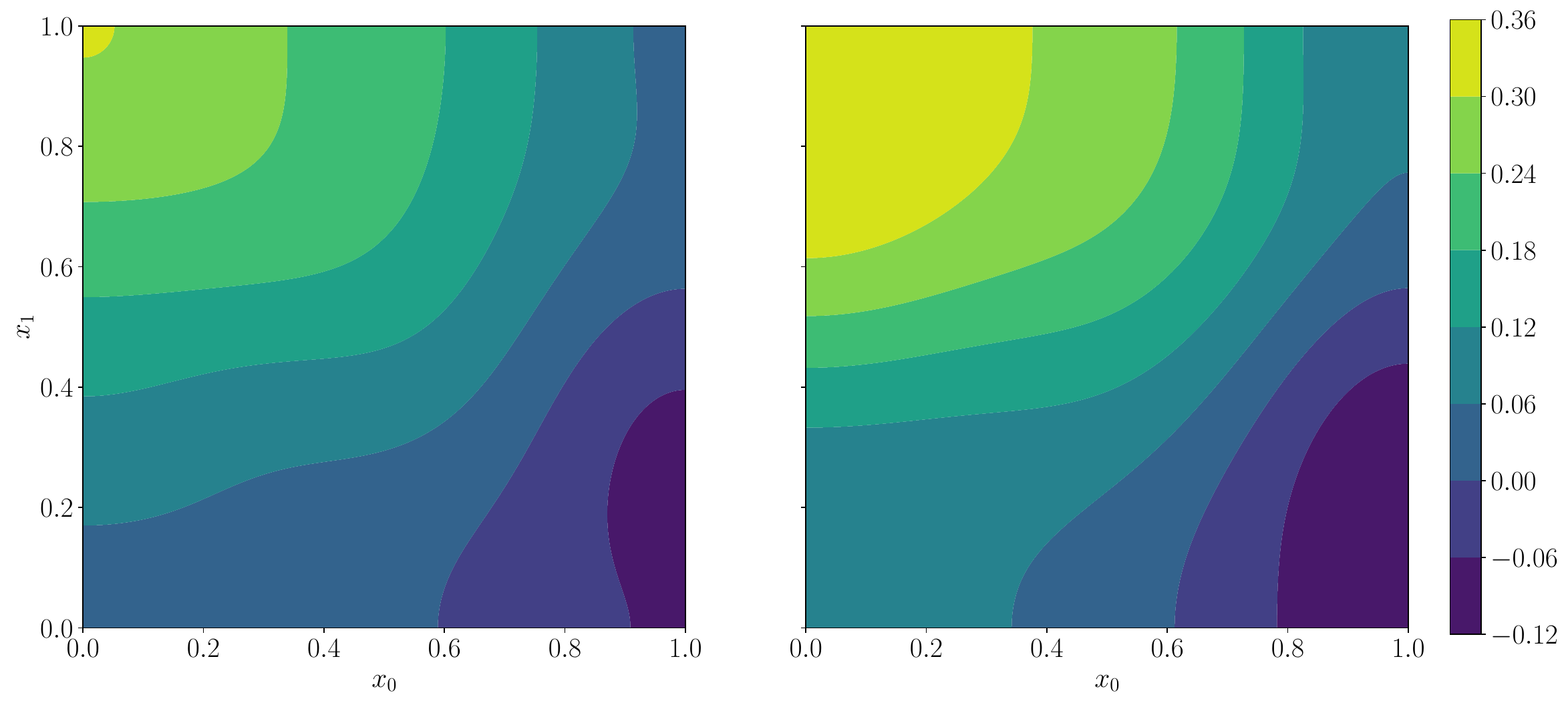}
    \caption{%
        Logarithms of true (left) and approximate permeability profiles (right).
        The approximate permeability profile was constructed from the approximation of the MAP estimator provided by~\eqref{eq:interacting_particle_system}
        with $\alpha = 0$, adaptive $\beta$ and $\lambda = 1$ (optimization mode),
        with $J = 512$ particles.
    }%
    \label{fig:log_of_permeablitiy}
\end{figure}

\subsection{Discussion}
\label{sec:discussion-numerics}
We draw the following conclusions from the numerical experiments presented in this section.
\begin{itemize}
    \item It is crucial to dynamically adapt the parameter $\beta$ during a simulation for our method to be competitive,
        both for optimization and sampling tasks.
        We obtained very good numerical results with the adaptation scheme based on the effective sample size in~\eqref{eq:ess}.

    \item For optimization tasks,
        our method generally requires more particles than CBO~\cite{pinnau2017consensus} in order to consistently find the global minimizer when the number of local minima is large.
        Relatedly, for a given number of particles,
        the probability of converging to (a small neighborhood) of the correct minimizer appears to be better for CBO.

    \item For sampling tasks,
        our numerical experiments suggest that the CBS method is competitive with the ensemble Kalman sampling scheme~\cite{garbuno2020interacting}.
        The number of iterations required by both methods in order to reach equilibrium is of the same order of magnitude,
        and the quality of the posterior approximation appears similar in the test cases we considered.
\end{itemize}

In future work, we will aim to give our proposed $\beta$-adaptation scheme a theoretical footing,
and to investigate other adaptation strategies.
It will also be worthwhile to more precisely compare our method with discretizations of CBO and EKS in terms of computational cost,
especially for PDE-based inverse problems,
where evaluations of the forward model are typically the predominant computational cost.
Finally, it would be interesting, both for optimization and sampling tasks,
to investigate whether ideas from~\cite{jin2018random,2019arXiv190810890N,garbuno2020affine} could be leveraged in order to improve the performance of our method
when the number of particles is of the same order of magnitude as the dimension of the parameter space.

\section{Proof of the Main Results}%
\label{sec:proof_of_the_main_results}

Throughout this section,
for a given $\vect m\in \real^d$ and $\mat C\in\real^{d\times d}$, we will use the notation
\begin{equation}\label{eq:rhobeta}
    \rho_\beta(\theta; \vect m, \mat C)=\frac{1}{Z_{\beta}}e^{-V_{\beta}(\theta)}\,,\qquad
    V_{\beta}(\theta; \vect m, \mat C) := \frac{1}{2}| \theta - \vect m|_C^2 + \beta f(\theta)\,,
\end{equation}
where $Z_{\beta} = Z_{\beta}(\vect m, \mat C)$ is the normalization constant.
When the parameters $\vect m$, $\mat C$ are clear from the context, we will often write just $\rho_{\beta}(\theta)$ and $V_{\beta}(\theta)$ for conciseness.

\subsection{Proof of the Convergence Estimates in the Gaussian Setting}%
\label{sub:gaussian_proofs}

\begin{proof}[Proof of \cref{prop:cv-alpha0-main}]
    Consider first the sampling case $\lambda=(1+\beta)^{-1}$.
    Using the same notation as in the proof of \cref{lem:cv-alpha0}, we have
    \begin{equation}
        \label{eq:corrolary_cn}
        \widetilde C_n^{-1} - \beta^{-1} I_d = \lambda^n \left(\widetilde C_0^{-1} - \beta^{-1} I_d \right).
    \end{equation}
    Rearranging the equation, we obtain
    \[
        \widetilde C_n - \beta I_d = (\widetilde C_n \widetilde C_0^{-1}) \lambda^n \left(\widetilde C_0 - \beta I_d \right).
    \]
    Since $\widetilde C_n$ commutes with $\widetilde C_0^{-1}$ from~\eqref{eq:corrolary_cn},
    the matrix $\widetilde C_n \widetilde C_0^{-1}$ is symmetric and positive definite.
    By~\eqref{eq:corrolary_cn}, the eigenvalues $\{\ell_i\}$ of $\widetilde C_n \widetilde C_0^{-1}$ are of the form
    \[
        \ell_i = \frac{1}{\beta^{-1} m_i + \lambda^n (1 - \beta^{-1} m_i)}
        \leq \max \left\{ \frac{\beta}{m_i}, 1 \right\} \leq \max \left\{1, k_0\right\} ,
    \]
    where $\{m_i\}$ denote the eigenvalues of $\widetilde C_0$. Hence,
    \begin{align*}
        \norm*{C_n-A}_A
        &=\beta^{-1}\norm*{\widetilde C_n - \beta I_d}
        = \lambda^n\beta^{-1} \norm{\left(\widetilde C_n \widetilde C_0^{-1} \right) \left(\widetilde C_0 - \beta I_d\right)}\\
        &\le \lambda^n \norm*{\widetilde C_n \widetilde C_0^{-1}} \beta^{-1}\norm*{\widetilde C_0 - \beta I_d}
        \le \lambda^n \max\left(1, k_0 \right) \norm*{ C_0 -  A}_A\,.
    \end{align*}
    This shows the convergence result of the covariance,
    and the convergence result for the mean follows similarly using \cref{lem:cv-alpha0}:
    \begin{align*}
        |\vect m_{n} - \vect a|_A &= |\widetilde{\vect m}_n|
        = \lambda^n |\widetilde C_n \widetilde C_0^{-1} \widetilde{\vect m}_0|
        \le \lambda^n \norm*{\widetilde C_n \widetilde C_0^{-1}} |\widetilde{\vect m}_0|\\
        &\le \lambda^n \max\left(1,k_0\right) |\widetilde{\vect m}_0|
        =\lambda^n \max\left(1,k_0\right) |\vect m_{0} - \vect a|_A\,.
    \end{align*}

    In the optimization case $\lambda = 1$,
    we have using the definition of $k_0$ that
    \[
        \widetilde C_n^{-1} = \widetilde C_0^{-1} + n I_d \succcurlyeq \left( 1  + \frac{\beta n}{k_0} \right) \widetilde C_0^{-1}
        \qquad \Rightarrow \qquad
        \widetilde C_n \preccurlyeq \left( \frac{k_0}{ k_0  + \beta n} \right)  \widetilde C_0.
    \]
    This shows the convergence result for the covariance,
    which directly implies the convergence estimate for the mean.
\end{proof}

\begin{proof}
    [Proof of \cref{prop:cv-discrete}]
    Notice that the right-hand side of~\eqref{eq:second_moment_quadratic_discrete_tilde} commutes with $\widetilde C_n$,
    so there exists an orthogonal matrix $Q$ such that $\widehat C_n := Q^\t \widetilde C_n Q$ is diagonal for all $n \in \nat$.
    Introducing $\widehat m_n = Q^\t \widetilde m_n$,
    we can check that $\widehat m_n$ and $\widehat C_n$ solve again ~\eqref{eq:moments_quadratic_tilde}.
    Therefore, for all $i \in \{1, \dotsc, d\}$,
    it holds that $(u_{i,n}, v_{i,n}) := \bigl((\widehat m_n)_i, (\widehat C_n)_{ii}\bigr)$ solves the discrete-time equation~\eqref{eq:1Dmoments}
    with initial conditions which depend on $i.$
    % Denote $\vect U_n, \vect V_n\in\real^d$ with $(\vect U_n)_i=u_{i,n}$ and $(\vect V_n)_i=v_{i,n}$.
    The convergence of the solution for the two-dimensional difference equation~\eqref{eq:1Dmoments} is then given by \cref{lemma:convergence_recursion_equations}.
    Note that $v_{i,0}\ge \beta/k_0$ for all $i\in\{1,\dotsc, d\}$,
    because by definition $k_0 = \beta\norm*{\widetilde C_0^{-1}} = \beta\norm*{\widehat C_0^{-1}}$.
    In the sampling case, we have
    \begin{align*}
        |\vect m_{n} - \vect a|_A
    &= |\widehat m_n|
    \le  \max (1, k_0)^{\frac{1}{1+\alpha}} \bigl((1-\alpha)\lambda+\alpha\bigr)^n|\widehat m_0|\\
    &=\max (1, k_0)^{\frac{1}{1+\alpha}}\bigl((1-\alpha)\lambda+\alpha\bigr)^n|\vect m_{0} - \vect a|_A\,.
    \end{align*}
    On the other hand, it holds for any $1\le i\le d$ that
    \begin{align*}
        \abs*{{(\widehat C_n)}_{ii} - \beta}
          &\leq \max (1, k_0)\left((1-\alpha^2)\lambda+\alpha^2\right)^n \abs*{(\widehat C_0)_{ii}- \beta}.
    \end{align*}
    From this, we deduce
    \begin{align*}
        \norm*{\widehat C_n - \beta I_d}
    &\le\max (1, k_0)\left((1-\alpha^2)\lambda+\alpha^2\right)^n \norm*{\widehat C_0- \beta I_d}.
    \end{align*}
    Since $\norm{QMQ^\t} = \norm{M}$ for any symmetric matrix $M$ and orthogonal matrix $Q$,
    we deduce
    \begin{align*}
        \norm*{\widetilde C_n - \beta I_d}
    &\le\max (1, k_0)\left((1-\alpha^2)\lambda+\alpha^2\right)^n \norm*{\widetilde C_0- \beta I_d}.
    \end{align*}
    The statement then follows because
    \(
    \norm{C_n - C_{\infty}}[A] = \beta^{-1} \norm*{\widetilde C_n - \beta I_d}
    \)
    by definition of $\norm{\dummy}[A]$.
    An analogous argument,
    using the estimates~\eqref{eq:convergence_v_opti_1} and~\eqref{eq:convergence_v_opti} in \cref{lemma:convergence_recursion_equations} and noting that
    the function $s \mapsto (s+1)/\bigl(s+1+(1-\alpha^2)n\bigr)$ is strictly decreasing for $s\ge 0$,
    yields the bounds for the optimization case $\lambda=1$.
\end{proof}

\begin{proof}
    [Proof of \cref{prop:cv-cont}]
    Letting $\widetilde {\vect m}(t) = A^{-1/2} (\vect m(t) - \vect a)$ and $\widetilde{\mat C}(t) = \beta A^{-1/2} \mat C(t) A^{-1/2}$,
    we can verify that $\widetilde {\vect m}$ and $\widetilde{\mat C}$ solve
    \begin{subequations}
        \begin{align*}
            \dot {\widetilde {\vect m}} &= - \widetilde{\mat C} \left(I_d + \widetilde{\mat C} \right)^{-1} \widetilde {\vect m}, \\
            \dot {\widetilde {\mat C}}  &= - 2 \widetilde {\mat C} \, \left(I_d + \widetilde{\mat C}\right)^{-1} \, \left( \widetilde {\mat C} -  \left(\frac{1-\lambda}{\lambda}\right)I_d \right).
        \end{align*}
    \end{subequations}
    It is then straightforward to show the result by employing the same reasoning as in the discrete-time case and using~\cref{lemma:cv-alpha1-1D},
    which characterizes the convergence to equilibrium for the following ODE system with $u,v$ scalar functions:
    \begin{align}
        \label{eq:1Dmoments-alph1}
        \dot u = - \left( \frac{v}{1 + v}\right) u,\qquad
        \dot v = -2 \left( \frac{v}{1 + v} \right)\left(v-v_{\infty}\right),
        \qquad v_{\infty} = \frac{1 - \lambda}{\lambda}.
    \end{align}
    We leave the details to the reader.
\end{proof}

\subsection{Proof of the Preliminary Bounds}%
\begin{proof}
    [Proof of \cref{lemma:bound_second_moment}]
    Recall notation \eqref{eq:rhobeta}, and let $\widetilde \theta$ denote the unique global minimizer of $V_{\beta}(\theta)$.
    The function $g$ defined by
    \[
        g(\theta) = f(\theta) - \left( f(\widetilde \theta) + \grad f(\widetilde \theta)^\t (\theta - \widetilde \theta) + \frac{1}{2} |\theta - \widetilde \theta|_{\Lhess^{-1}}^2  \right)
    \]
    is such that $g(\widetilde \theta) = \grad g(\widetilde \theta) = 0$ and $\hessian g(\theta) \succcurlyeq 0$ for all $\theta \in \real^d$,
    by the convexity assumption on the function $f$.
    We denote
    \[
        \widetilde V_{\beta}(\theta):=\frac{1}{2}|\theta|^2+\beta \widetilde g (\theta)\,,\qquad
         \widetilde g(\theta) := g\left(\widetilde \theta + \left( C^{-1} + \beta \Lhess \right)^{-1/2} \theta\right),
     \]
     and define $\widetilde \rho_\beta(\theta)= \frac{1}{\widetilde Z_{\beta}} \e^{-\widetilde V_{\beta}(\theta)}$ where $\widetilde Z_{\beta}$ is the normalization constant.
     By a change of variables, it holds
    \begin{align*}
        \mat C_\beta(\vect m, \mat C) =\mathcal C (\rho_\beta)
        = \left( C^{-1} + \beta \Lhess \right)^{-1/2}  \mathcal{C}(\widetilde \rho_{\beta}) \left( C^{-1} + \beta \Lhess \right)^{-1/2},
    \end{align*}
    It remains to show $\mathcal{C}(\widetilde \rho_{\beta}) \preccurlyeq I_d$ or,
    equivalently, that for every unit vector $\vect a\in \real^d$ it holds
    \begin{align}
        \label{eq:intermediate_equation_bound_above}
        \vect a^\t \mathcal{C}(\widetilde \rho_{\beta}) \vect a=\int_{\real^d} \abs{\vect a^\t \theta- \int_{\real^d} (\vect a^\t \theta) \, \widetilde \rho_\beta(\theta)  \d \theta}^2\widetilde \rho_\beta(\theta)\,\d \theta \le 1,
    \end{align}
    Clearly $\widetilde g(0) = \grad \widetilde g(0) = 0$ and $\hessian \widetilde g \succcurlyeq 0$, so $\hessian \widetilde V_{\beta} \succcurlyeq I_d$.
    Therefore, by the Bakry-Emery criterion~\cite[Theorem 2.10]{MR3509213}, the probability distribution $\d \mu(\theta):=\widetilde \rho_\beta(\theta)\d\theta$
    satisfies a logarithmic Sobolev inequality,
    and thus also a Poincar\'e inequality by~\cite[Proposition 2.12]{MR3509213},
    with the factor on the right equal to 1.
    That is, it holds
    \[
        \forall u \in H^1(\mu), \qquad
        \int_{\real^d} \left| u - \int_{\real^d} u \, \d \mu \right|^2 \d \mu
        \leq \int_{\real^d} \vecnorm{\grad u}^2 \d \mu.
    \]
    Applying this inequality with $u(\theta) = \vect a^\t \theta$ gives~\eqref{eq:intermediate_equation_bound_above}.
\end{proof}

\begin{proof}
    [Proof of \cref{lemma:lower_bound_second_moment}]
    Let $\widetilde \theta$ denote again the unique global minimizer of $V_{\beta}(\theta)$,
    where $V_{\beta}$ is given in \eqref{eq:rhobeta}.
    The function $g$ defined by
    \[
        g(\theta) = f(\theta) - \left( f(\widetilde \theta) + \grad f(\widetilde \theta)^\t (\theta - \widetilde \theta) + \frac{1}{2} |\theta - \widetilde \theta|_{\Uhess^{-1}}^2  \right)
    \]
    is such that $g(\widetilde \theta) = \grad g(\widetilde \theta) = 0$ and $\hessian g(\theta) \preccurlyeq 0$ for all $\theta \in \real^d$,
    by~\cref{assumption:convexity_potential_above}.
    By a change of variables, it holds
    \begin{align*}
        \mat C_\beta(\vect m, \mat C) =\mathcal C (\rho_\beta)
        = \left( C^{-1} + \beta \Uhess \right)^{-1/2}  \mathcal{C}(\widetilde \rho_{\beta}) \left( C^{-1} + \beta \Uhess \right)^{-1/2},
    \end{align*}
    where $\widetilde \rho_\beta(\theta) = \frac{1}{\widetilde Z_{\beta}} \e^{-\widetilde V_{\beta}(\theta)}$,
    with $\widetilde Z_{\beta}$ the normalization constant and
    \[
        \widetilde V_{\beta}(\theta):=\frac{1}{2}|\theta|^2+\beta \widetilde g (\theta)\,,\qquad
        \widetilde g(\theta) := g\left(\widetilde \theta + \left( C^{-1} + \beta \Uhess \right)^{-1/2} \theta\right).
    \]
    It remains to show that $\mathcal C(\widetilde \rho_{\beta}) \succcurlyeq I_d$.
    To this end, let $\bar \theta = \mathcal M(\widetilde \rho_{\beta})$ for brevity and,
    for a given unit vector $\vect a \in \real^d$, let $\partial_{\vect a} = \vect a^\t \grad$ and so $\partial_{\vect a}\widetilde \rho_\beta = -(\partial_{\vect a} \widetilde V_\beta )\widetilde \rho_\beta$.
    By the Cauchy--Schwarz inequality,
    \[
        \int_{\real^d} \partial_{\vect a} \widetilde V_{\beta}(\theta) \, \vect a^\t(\theta - \bar \theta) \, \widetilde \rho_{\beta}(\theta) \, \d \theta
        \leq \sqrt{\int_{\real^d} \abs{\partial_{\vect a}\widetilde V_{\beta}(\theta)}^2 \widetilde \rho_{\beta}(\theta) \, \d \theta} \,\,
        \sqrt{\int_{\real^d} \abs{\vect a^\t(\theta - \bar \theta)}^2 \widetilde \rho_{\beta}(\theta) \, \d \theta}.
    \]
    After rearranging and using integration by parts, this gives
    \begin{align*}
        \vect a^\t \mathcal C (\widetilde \rho_{\beta})  \vect a
        &\geq \frac{\left( \int_{\real^d} \partial_{\vect a}\widetilde V_{\beta}(\theta) \, \vect a^\t(\theta - \bar \theta) \, \widetilde \rho_{\beta}(\theta) \, \d \theta \right)^2}
        {\int_{\real^d} \abs{\partial_{\vect a}\widetilde V_{\beta}(\theta)}^2 \widetilde \rho_{\beta}(\theta) \, \d \theta} \\
        &= \frac{\left( \int_{\real^d}  \vect a^\t(\theta - \bar \theta) \, \partial_{\vect a}\widetilde \rho_{\beta}(\theta) \, \d \theta \right)^2}
        {- \int_{\real^d} \partial_{\vect a}\widetilde V_{\beta}(\theta) \partial_{\vect a}\widetilde \rho_{\beta}(\theta) \, \d \theta}
        = \frac{1}
        {\int_{\real^d} \partial_{\vect a}^2\widetilde V_{\beta}(\theta) \widetilde \rho_{\beta}(\theta) \, \d \theta},
    \end{align*}
    where we denote $\partial_{\vect a}^2h(\theta)= \vect a^T D^2h(\theta) \vect a$.
    Since $\hessian \widetilde V_{\beta} \preccurlyeq I_d$ because $\hessian \widetilde g \preccurlyeq 0$,
    it follows immediately that $\mathcal C (\widetilde \rho_{\beta}) \succcurlyeq I_d$.
\end{proof}

\begin{proof}
    [Proof of \cref{lemma:first_moment_several_dimensions}]
    Let $\widetilde \theta$ denote again the unique global minimizer of $V_{\beta}(\theta)$ given by \eqref{eq:rhobeta}.
    We first show a bound on $\widetilde \theta - \theta_*$.
    By the assumptions on $f$, it holds
    \[
        V_{\beta}(\theta) \geq \frac{1}{2} \vecnorm{\theta - m}_C^2 + \frac{\lhess \beta}{2} |\theta - \theta_*|^2 + \beta f(\theta_*)
        \geq \frac{\lhess \beta}{2} |\theta - \theta_*|^2 + \beta f(\theta_*).
    \]
    Likewise, it holds $V_{\beta}(\theta_*) \leq \frac{1}{2}\norm{C^{-1}} \vecnorm{\theta_* - m}^2 + \beta f(\theta_*)$,
    so we obtain
    \begin{align*}
        V_{\beta}(\theta) - V_{\beta}(\theta_*)
        &\geq \frac{\lhess \beta}{2} |\theta - \theta_*|^2 - \frac{1}{2} \norm{C^{-1}} \vecnorm{\theta_* - m}^2.
    \end{align*}
    In particular, for any $\theta$ such that
    \[
        \vecnorm{\theta - \theta_*} > \left(\frac{\norm{C^{-1}}}{\lhess \beta}\right)^{1/2}\vecnorm{\theta_* - m} =: R,
    \]
    it holds $V_{\beta}(\theta) - V_{\beta}(\theta_*) > 0$,
    implying that $\vecnorm*{\widetilde \theta - \theta_*} \leq R$.
    Now,
    \begin{align}
        \notag%
        \vecnorm*{\vect m_\beta(\vect m, \mat C) - \widetilde \theta}
        &=        \vecnorm*{\mathcal{M}(\rho_\beta) - \widetilde \theta}
        = \abs{\int_{\real} (\theta - \widetilde \theta) \rho_{\beta} (\theta) \, \d \theta} \\
        \label{eq:intermediate_bound_first_moment}
        &\leq \sqrt {\int_{\real^d}  \abs*{\theta - \widetilde \theta}^2 \, \rho_{\beta}(\theta) \, \d \theta}
        = \sqrt{\frac{\int_{\real^d}\abs*{\theta - \widetilde \theta}^2 \, \e^{-V_{\beta}(\theta)} \, \d \theta}
        {\int_{\real^d} \, \e^{-V_{\beta}(\theta)} \, \d \theta}}\,.
    \end{align}
    Since $V_{\beta}(\theta)$ is minimized at $\theta = \widetilde \theta$,
    it holds
    \[
        V_{\beta}(\widetilde \theta) + \frac{1}{2}  \vecnorm*{\theta - \widetilde \theta}[(C^{-1} + \beta \Lhess)^{-1}]^2
        \leq V_{\beta}(\theta) \leq
        V_{\beta}(\widetilde \theta) + \frac{1}{2}  \vecnorm*{\theta - \widetilde \theta}[(C^{-1} + \beta \Uhess)^{-1}]^2
    \]
    Using these inequalities,
    we can obtain an upper bound for the numerator in~\eqref{eq:intermediate_bound_first_moment}
    and a lower bound for the denominator in~\eqref{eq:intermediate_bound_first_moment}, respectively:
    \begin{align*}
        \int_{\real^d}\vecnorm*{\theta - \widetilde \theta}^2 \, \e^{-V_{\beta}(\theta)} \, \d \theta
        &\leq \e^{-V_{\beta}(\widetilde \theta)} \trace \left( \left( C^{-1} + \beta \Lhess \right)^{-1} \right)  \det \left(C^{-1} + \beta \Lhess\right)^{-1/2} (2\pi)^{d/2} \\
        \int_{\real^d} \, \e^{-V_{\beta}(\theta)} \, \d \theta
        &\geq \e^{-V_{\beta}(\widetilde \theta)} \det \left(C^{-1} + \beta\Uhess \right)^{-1/2} (2\pi)^{d/2}.
    \end{align*}
    Combining these inequalities,
    writing the determinant as a product of eigenvalues,
    and using the inequality $\frac{1+x}{1+y}\le \frac{x}{y}$ for all $0<y\le x$, we deduce
    \begin{align*}
        \notag%
        \vecnorm*{\vect m_\beta(\vect m, \mat C) - \widetilde \theta}
        &\leq \sqrt{ \trace \left( \left( C^{-1} + \beta \Lhess \right)^{-1} \right) } \frac{\det \left(C^{-1} + \beta \Uhess\right)^{1/4}}{\det \left(C^{-1} + \beta \Lhess\right)^{1/4}} \\
        &\leq \sqrt{d \norm{ (C^{-1} + \beta \Lhess)^{-1} }} \frac{\det \left(C^{-1} + \beta \uhess I_d\right)^{1/4}}{\det \left(C^{-1} + \beta \lhess I_d\right)^{1/4}}
        \leq  \sqrt{d \norm{ (C^{-1} + \beta \Lhess)^{-1} }}\left(\frac{\uhess}{\lhess}\right)^{d/4}.
    \end{align*}
    The statement then follows from the triangle inequality,
    \[
        \vecnorm{\vect m_\beta(\vect m, \mat C) - \theta_*} \leq \vecnorm*{\theta_* - \widetilde \theta} + \vecnorm*{\vect m_\beta(\vect m, \mat C) - \widetilde \theta},
    \]
    and from the fact that $\norm{ (C^{-1} + \beta \Lhess)^{-1} } \leq \norm{ (C^{-1} + \beta \lhess I_d)^{-1} } \leq \left( \norm{C}^{-1} + \beta \lhess \right)^{-1}$.
\end{proof}

\subsection{\texorpdfstring%
    {Proof of \cref{lemma:collapse_optimization} and \cref{thm:no_bad_convergence}}
    {Proof that the optimization scheme can converge only to the minimizer}}%

\begin{proof}
    [Proof of \cref{lemma:collapse_optimization}]
    Let $x = \alpha^2$ for simplicity.
    It holds by~\eqref{eq:second_moment_discrete} and \cref{lemma:bound_second_moment} that
    \[
        C_{n+1} \preccurlyeq x C_n + (1 - x) (C_n^{-1} + \beta \Lhess)^{-1}.
    \]
    Therefore, introducing $\bar C_n = \beta \Lhess^{1/2} C_{n} \Lhess^{1/2}$, it holds
    \[
        \bar C_{n+1} \preccurlyeq x \bar C_{n} + (1 - x) (\bar C_n^{-1} + I_d)^{-1}.
    \]
    Let $\bar D_{n}$ denote the solution to the discrete-time equation
    \[
        \bar D_{n+1} = x \bar D_{n} + (1 - x) (\bar D_n^{-1} + I_d)^{-1},
        \qquad \bar D_0 = \bar C_0.
    \]
    It is clear that $\bar C_{n} \preccurlyeq \bar D_{n}$ for all $n \geq 0$.
    Indeed, this is true for $n = 0$, and if $\bar C_{n} \preccurlyeq \bar D_{n}$ then
    \begin{align*}
        \bar D_{n+1} - \bar C_{n+1}
        &\succcurlyeq x(\bar D_n - \bar C_n) + (1 - x) \left( (\bar D_n^{-1} + I_d)^{-1} - (\bar C_n^{-1} + I_d)^{-1} \right)\\
        &\succcurlyeq (1 - x) \left( (\bar D_n^{-1} + I_d)^{-1} - (\bar C_n^{-1} + I_d)^{-1} \right)
    \end{align*}
    By~\cite[Proposition V.1.6]{MR1477662}, the function $\real \ni s \mapsto -1/s$ is operator monotone on $(0, \infty)$,
    meaning that if two symmetric positive definite matrices $M_1$ and $M_2$ are such that $M_1 \succcurlyeq M_2$, then it holds that $M_1^{-1} \preccurlyeq M_2^{-1}$.
    Therefore
    \[
        \bar C_n \preccurlyeq \bar D_n
        ~~ \Rightarrow ~~
        \bar C_n^{-1} \succcurlyeq \bar D_n^{-1}
        ~~ \Rightarrow ~~
        \bar C_n^{-1} + I_d \succcurlyeq \bar D_n^{-1} + I_d
        ~~ \Rightarrow ~~
        (\bar C_n^{-1} + I_d)^{-1} \preccurlyeq (\bar D_n^{-1} + I_d)^{-1},
    \]
    which shows that $\bar D_{n+1} - \bar C_{n+1} \succcurlyeq 0$.
    Now note that $\bar D_n$ satisfies the same equation as $\widetilde C_n$ in~\eqref{eq:second_moment_quadratic_discrete_tilde},
    so we deduce by a reasoning similar to the proof of~\cref{prop:cv-discrete} that $\bar D_n$ satisfies
    \[
        \bar D_n \preccurlyeq \left(\frac{\norm{\bar C_0^{-1}}+1}{\norm{\bar C_0^{-1}}+1+(1-x)n}\right) \bar C_0,
    \]
    which implies the statement for the discrete-time case $\alpha\in (0,1)$.
    If $\alpha=0$, then it follows from \cref{prop:cv-alpha0-main} that
        \[
        \bar D_n \preccurlyeq \left(\frac{\norm{\bar C_0^{-1}}}{\norm{\bar C_0^{-1}}+n}\right) \bar C_0.
    \]

    Similarly in the continuous-time case,
    let $\bar C(t) = \beta \Lhess^{1/2} C(t) \Lhess^{1/2}$ and let $\bar D(t)$ denote the solution to
    the equation
    \[
        \derivative*{1}{t}\bar D(t) = -2 \bar D(t) + 2 \bigl(\bar D(t)^{-1} + I_d\bigr)^{-1},
        \qquad \bar D(0) = \bar C(0).
    \]
    We have by~\eqref{eq:second_moment} and \cref{lemma:bound_second_moment} that
    \[
        \derivative*{1}{t} \bar C(t) \preccurlyeq -2 \bar C(t) + 2 \bigl(\bar C(t)^{-1} + I_d\bigr)^{-1}.
    \]
    Using the same reasoning as in the discrete-time case, we derive that
    \[
        \derivative*{1}{t} \bigl(\bar D(t) - \bar C(t)\bigr) \succcurlyeq -2 \bigl( \bar D(t) - \bar C(t)\bigr)
        \qquad\Leftrightarrow\qquad \derivative*{1}{t} \Bigl( \e^{2t}  \bigl(\bar D(t) - \bar C(t)\bigr) \Bigr) \succcurlyeq 0,
    \]
    and so $\bar C(t) \preccurlyeq \bar D(t)$ for all $t \geq 0$.
    Employing a reasoning similar to that in \cref{prop:cv-cont},
    we obtain the statement.
\end{proof}
We show a similar result establishing a lower bound on $C_n$.
\begin{lemma}[Lower bound on the covariance in optimization mode]
    \label{eq:optim_lower_bound_cov}
    Let $\lambda=1$, $\beta>0$ and $\alpha\in [0,1)$, and assume that \cref{assumption:convexity_potential_above} holds.
    Then, for any solution $\{(m_n, C_n)\}_{n \in \nat}$ to~\cref{eq:first_moment_discrete_gaussian,eq:second_moment_discrete_gaussian} with $C_0 \in \pos^d$,
    it holds that
    \begin{equation}
        \label{eq:lower_bound_collapse}
        C_{n} \succcurlyeq
        \left(C_{0}^{-1} + n (1 - \alpha^2)\beta \Uhess\right)^{-1} \,.
    \end{equation}
    Likewise, for any solution $\bigl\{ \bigl(m(t), C(t)\bigr)\bigr\}_{t \in \real_{\geq 0}}$ to~\cref{eq:momentprop-a,eq:momentprop-b} with $C(0) \in \pos^d$,
    the following inequality holds:
    \begin{equation}
        \label{eq:lower_bound_collapse_continuous}
        C(t) \succcurlyeq
        \left(C(0)^{-1} + 2 t \beta \Uhess\right)^{-1} \,.
    \end{equation}
\end{lemma}
\begin{proof}
    Let us now use the notation $\widehat C_n = \beta \Uhess^{1/2} C_{n} \Uhess^{1/2}$.
    It holds by \cref{lemma:lower_bound_second_moment}
    \[
        \widehat C_{n+1} \succcurlyeq x \widehat C_{n} + (1 - x) (\widehat C_n^{-1} + I_d)^{-1}.
    \]
    Defining $\widehat P_n = \widehat C_n^{-1}$ for $n \in \{0, 1, \dotsc\}$
    we have
    \[
    \widehat P_{n+1} \preccurlyeq (xI_d + \widehat P_n)^{-1} (I_d + \widehat P_n) \widehat P_n
    = \widehat P_n + (1-x) (I_d + x\widehat P_n^{-1})^{-1} \preccurlyeq \widehat P_n + (1 - x)I_d,
    \]
    so we deduce~\eqref{eq:lower_bound_collapse}.
    For the continuous-time case,
    we employ the notation $\widehat C(t) = \beta \Uhess^{1/2} C(t) \Uhess^{1/2}$ and $\widehat P(t) = \widehat C(t)^{-1}$.
    By~\eqref{eq:momentprop-b} and \cref{lemma:lower_bound_second_moment},
    we have that
    \begin{align*}
        \derivative*{1}{t} \widehat C(t) &\succcurlyeq  -2 \widehat C(t) + 2 (\widehat C(t)^{-1} + I_d)^{-1}\\
        &= -2  (\widehat C(t)^{-1} + I_d)^{-1} \left[( \widehat C(t)^{-1} + I_d)\widehat C(t) - I_d\right]
        = -2 \widehat C(t)\bigl(\widehat C(t)+I_d\bigr)^{-1} \widehat C(t)\,.
    \end{align*}
    Hence,
    \[
 \derivative*{1}{t} \widehat P(t)
 = - \widehat C(t)^{-1}\derivative*{1}{t} \widehat C(t) \widehat C(t)^{-1}
 \preccurlyeq 2 \bigl(I_d + \widehat C(t) \bigr)^{-1} \preccurlyeq 2 I_d,
    \]
    leading to the statement.
\end{proof}

\begin{remark}
    \label{remark:conditioning}
    A simple corollary of~\cref{lemma:collapse_optimization,eq:optim_lower_bound_cov} is that
    the condition number
    $$
    \cond(\mat C_n)= \norm*{\mat C_n} \norm*{\mat C_n^{-1}}
    $$
    of $C_n$ remains bounded as $n \to \infty$,
    and similarly in continuous time.
\end{remark}

In order to prove~\cref{thm:no_bad_convergence}, we first show the following auxiliary result.

\begin{lemma}
    \label{lemma:no_bad_cv}
   Let $\beta>0$ and suppose $f$ satisfies \cref{assumption:convexity_potential,assumption:convexity_potential_above}.
   Then there exists a constant $K=K(\beta, d, \lhess, \uhess)>0$ such that
   the following inequality holds
    \begin{align}
        \notag%
        \vecnorm{\vect m_\beta(\vect m, \mat C) - \vect m + \beta \mat C \grad f(\vect m)}
        &\leq \e^{\beta f(\vect m)}
        \frac{K \beta \vecnorm{C \grad f(\vect m)} \norm{\mat C} + K \norm{\mat C}^{3/2}}
        {1 - K \e^{\beta f(\vect m)} \norm{C}},
    \end{align}
    for all $(\vect m \times C) \in \real^d \times \pos^d$ such that the denominator is positive.
\end{lemma}

\begin{proof}
    By Taylor's theorem, there exists for all $(\theta, \vect m) \in \real^d \times \real^d$ a point $\xi = \xi(\theta, \vect m) \in \real^d$
    on the straight segment between $\theta$ and $\vect m$ such that
    \begin{align*}
        \e^{-\beta f(\theta)} =& \e^{- \beta f(\vect m)} - \e^{-\beta f(\vect m)} \beta \grad f(\vect m) \cdot (\theta - \vect m)  \\
                               &+ \frac{1}{2} \e^{- \beta f(\xi)} \left( \beta^2 \big(\grad f(\xi) \otimes \grad f(\xi) \big) - \beta \hessian f(\xi) \right) : \big( (\theta - \vect m) \otimes (\theta -\vect m) \big) \\
                              =: &\e^{- \beta f(\vect m)} - \e^{-\beta f(\vect m)} \beta \grad f(\vect m) \cdot (\theta - \vect m) + R(\theta; \vect m).
    \end{align*}
    By~\cref{assumption:convexity_potential} and \cref{assumption:convexity_potential_above},
    it is clear that
    \[
        \frac{1}{2} \sup_{\xi \in \real^d} \biggl( \e^{- \beta f(\xi)} \Bigl( \beta^2 \vecnorm{\grad f(\xi)}^2 + \beta \norm{\hessian f(\xi)}[\rm F] \Bigr) \biggr) < \infty,
    \]
    where $\norm{\dummy}[\rm F]$ denote the Frobenius norm.
    Consequently, there exists a constant $M$ such that
    \begin{align}
        \label{eq:remainder}
        \forall (\theta, \vect m) \in \real^d \times \real^d, \qquad
        \abs{R(\theta; \vect m)} \leq  M \abs{\theta - \vect m}^2.
    \end{align}
    We therefore deduce
    \begin{subequations}
    \begin{align}
        \label{eq:laplace_quadratic_0}%
        \int_{\real^d} g(\theta; \vect m, C) \, \e^{-\beta f(\theta)} \, \d \theta &= \e^{-\beta f(\vect m)} + \, R_0(\vect m, C), \\
        \label{eq:laplace_quadratic_1}%
        \int_{\real^d} (\theta - \vect m) \, g(\theta; \vect m, C) \, \e^{-\beta f(\theta)} \, \d \theta &= - \e^{-\beta f(\vect m)} \beta C \grad f(\vect m) + R_1(\vect m, C),
    \end{align}
    \end{subequations}
    with remainder terms satisfying the bounds
    \begin{equation}
        \label{eq:bounds_remainders}
        \forall (\vect m, \mat C) \in \real^d \times \pos^d, \qquad
        \left\{
        \begin{aligned}
            &\abs{R_0(\vect m, C)} \leq K \norm{C}, \\
            &\abs{R_1(\vect m, C)} \leq K \norm{C}^{3/2}.
        \end{aligned}
        \right.
    \end{equation}
    The second bound holds because,
    by~\eqref{eq:remainder} and a change of variable, we have
    \begin{align*}
        \vecnorm{ R_1(\vect m, C) }
        &\leq M\int_{\real^d} \vecnorm{ \theta - \vect m }^3 \, g(\theta; \vect m, \mat C) \, \d \theta \\
        &= M\int_{\real^d} \vecnorm*{C^{1/2}u}^3 \, g(u; \vect 0, I_d) \, \d u
        \leq M\norm*{C}^{3/2} \int_{\real^d} \vecnorm*{u}^3 g(u; \vect 0, I_d) \, \d u.
    \end{align*}
    Using \cref{eq:laplace_quadratic_0,eq:laplace_quadratic_1},
    we obtain
    \[
        \vect m_\beta(\vect m, \mat C) - \vect m = \frac{- \e^{-\beta f(\vect m)} \beta C \grad f(\vect m) + R_1(\vect m, C)}{\e^{-\beta f(\vect m)} + R_0(\vect m, C)}.
    \]
    In view of~\eqref{eq:bounds_remainders},
    it therefore holds
    \begin{align*}
        \notag%
        \vecnorm{\vect m_\beta(\vect m, \mat C) - \vect m + \beta \mat C \grad f(\vect m)}
        &= \abs{\frac{ \beta \mat C \grad f(\vect m) R_0(\vect m, C) + R_1(\vect m, \mat C)}{\e^{-\beta f(\vect m)} + R_0(\vect m, \mat C)}} \\
        &\leq \e^{\beta f(\vect m)}
        \frac{K \beta \abs{C \grad f(\vect m)} \norm{\mat C} + K \norm{\mat C}^{3/2}}
        {\abs{1 + \e^{\beta f(\vect m)} R_0(\vect m, \mat C)}}.
    \end{align*}
    Using the bound on $R_0(\vect m, \mat C)$ given in~\eqref{eq:bounds_remainders},
    we obtain the statement.
\end{proof}

\begin{proof}
    [Proof of~\cref{thm:no_bad_convergence}]
    For a contradiction, assume $\vect m_n \to \hat \theta$ and $\hat \theta\neq \theta_*$, where $\theta_*$ denotes the global minimizer of $f$.
    Then, by the convexity assumption on $f$,
    it holds that $\vecnorm*{\grad f(\hat \theta)} > 0$.
    By~\cref{lemma:collapse_optimization},
    it holds $C_n \to 0$, and by \cref{remark:conditioning},
    the condition number of $C_n$ satisfies $\cond(C_n) \leq \kappa$ for some $\kappa > 0$ and all $n \in \{0, 1, \dotsc\}$.
    By continuity of $\nabla f$ at $\hat \theta$, we have that for any $\varepsilon > 0$, there is $\delta=\delta(\varepsilon) > 0$ such that
    \begin{equation}
        \label{eq:definition_delta}
        \forall \vect m \in B_{\delta}(\hat\theta), \qquad
        \vecnorm{ \grad f(\hat\theta) - \grad f(\vect m) }
        \leq \frac{\varepsilon}{\kappa} \vecnorm{\grad f(\hat\theta)}
    \end{equation}
    Fix $0 < \varepsilon \ll 1$ and let $\delta = \delta(\varepsilon)$.
    From Lemma~\ref{lemma:no_bad_cv},
    there exists $ K> 0$ such that the inequality
     \begin{align}
         \label{eq:inequality_from_lemma}
         \vecnorm{\vect m_\beta(\vect m, \mat C) - \vect m + \beta \mat C \grad f(\vect m)}
         &\leq \frac{K \beta \e^{\beta f(\vect m)} \vecnorm{C \grad f(\vect m)} \norm{\mat C} + K \e^{\beta f(\vect m)} \norm{\mat C}^{3/2}}
         {1 - K \e^{\beta f(\vect m)} \norm{C}}
    \end{align}
    is satisfied for all $(\vect m, \mat C) \in  (\real^d \times \pos^d)$ such that the denominator is positive.
    We claim that there exists $\widetilde c > 0$ such that
    the following inequalities are satisfied
    for all $\vect m \in B_{\delta}(\hat\theta)$
    and all matrices $0 < C \leq \widetilde c I_d$ such that $\cond(C) \leq \kappa$:
    \begin{equation}
        \label{eq:three_ineualities}%
        \left\{
        \begin{aligned}
            & \abs{1 - K \e^{\beta f(\vect m)} \norm{C}} \geq \frac{1}{2}, \\
            & K \beta \e^{\beta f(\vect m)} \norm{C} \leq \frac{\varepsilon}{4}, \\
            & K \e^{\beta f(\vect m)} \norm{C}^{3/2} \leq \frac{\varepsilon}{4} \vecnorm{C \grad f(\vect m)}.
        \end{aligned}
        \right.
    \end{equation}
    Indeed, it suffices to choose
    \[
        \widetilde c = \min \left(
            \frac{I}{2K} ,
            \frac{\varepsilon I}{4 K \beta},
            {\left(\frac{\varepsilon I}{ 4K \kappa } \inf_{\vect m \in B_{\delta}(\hat\theta)} \vecnorm{\grad f(\vect m)}\right)}^2
        \right),
        \qquad \text{ where }
        I = \inf_{\vect m \in B_{\delta}(\hat\theta)} \e^{-\beta f(\vect m)}.
    \]
    Here the arguments of the minimum guarantee that each of the three inequalities in~\eqref{eq:three_ineualities} are satisfied,
    respectively.
    We note that $\inf_{\vect m \in B_{\delta}(\hat\theta)} \vecnorm{\grad f(\vect m)} > 0$ by~\eqref{eq:definition_delta} and the fact that $\varepsilon/\kappa < 1$.
    To justify that the third inequality in~\eqref{eq:three_ineualities} is indeed satisfied for this choice of $\widetilde c$,
    notice that
    \[
        \vecnorm{C \grad f(\vect m)}
        \geq \lambda_{\min}(C) \vecnorm{\grad f(\vect m)}
        \geq \cond(C)^{-1} \vecnorm{\grad f(\vect m)} \norm{C}.
    \]
    Substituting the three inequalities in~\eqref{eq:three_ineualities} into the estimate~\eqref{eq:inequality_from_lemma} from~\cref{lemma:no_bad_cv},
    we obtain that,
    for all $\vect m \in B_{\delta}(\hat\theta)$ and all $0 < C \leq \widetilde c I_d$ such that $\cond(C) \leq \kappa$,
    it holds
    \begin{align}
        \label{eq:bound_remainder}
        \vecnorm{\vect m_\beta(\vect m, \mat C) - \vect m + \beta C \grad f(\vect m)}
        \leq \varepsilon \vecnorm{ C \grad f(\vect m)}.
    \end{align}
    Now since $(\vect m_n, \mat C_n) \to (\hat\theta, 0)$ as $n \to \infty$ by assumption,
    there exists $N$ sufficiently large such that $\vect m_n \in B_{\delta}(\hat\theta)$ and $0 < C_n \leq \widetilde c I_d$ and $\cond(C_n) \leq \kappa$ for all $n \geq N$.
    By~\eqref{eq:equations_moments_discrete},
    we have that for any $n \geq N$ it holds
    \[
        \vect m_{n+1} - \vect m_n = (1 - \alpha) \bigl(\vect m_\beta(\vect m_n, C_n) - \vect m_n\bigr)
        = - (1 - \alpha) \bigl( \beta C_n \grad f(\vect m_n) + \vect r(\vect m_n, C_n) \bigr),
    \]
    where $\vect r(\vect m_n, \mat C_n)$ is the remainder term, bounded by~\eqref{eq:bound_remainder}.
    Taking the inner product of both sides with $\grad f(\hat\theta)$ and using~\eqref{eq:bound_remainder},
    we deduce
    \[
        - (\vect m_{n+1} - \vect m_n)^\t \grad f(\hat\theta)
        \geq \beta (1 - \alpha) \left( \grad f(\hat\theta)^\t C_n \grad f(\vect m_n) \right) - \varepsilon (1-\alpha)\norm{C_n} \vecnorm{\grad f(\vect m_n)} \vecnorm*{\grad f(\hat\theta)} .
    \]
    For any $(\vect x, \vect y) \in \real^d \times \real^d$ with $\vecnorm{\vect x - \vect y} \leq \zeta \vecnorm{\vect x}$,
    it holds
    \begin{align*}
        \vect x^\t C_n \vect y
        &= \vect x^\t C_n \vect x - \vect x^\t C_n \left(\vect x - \vect y\right) \\
        &\geq \vect x^\t C_n \vect x - \sqrt{\vect x^\t C_n \vect x^\t}\sqrt{\left(\vect x - \vect y\right)^\t C_n \left(\vect x - \vect y\right)}
        \geq \lambda_{\min}(C_n) \vecnorm{\vect x}^2 \left(1 - \cond(C_n) \, \zeta \right).
    \end{align*}
    Together with~\eqref{eq:definition_delta}, this implies
    \begin{align*}
        \forall n \geq N, \qquad
        - (\vect m_{n+1} - \vect m_n)^\t \grad f(\hat\theta)
        \geq& \, \beta (1 - \alpha)  (1 - \varepsilon) \lambda_{\min}(C_n) \vecnorm*{\grad f(\hat\theta)}^2 \\
            &\qquad - \varepsilon \left(1 + \frac{\varepsilon}{\kappa}\right)(1-\alpha) \lambda_{\max}(C_n) \vecnorm*{\grad f(\hat\theta)}^2.
    \end{align*}
    By repeating this reasoning with a smaller $\varepsilon$ if necessary,
    we can ensure
    \begin{equation}
        \label{eq:result_large_n}%
        \forall n \geq N, \qquad
        - (\vect m_{n+1} - \vect m_n)^\t \grad f(\hat\theta) \geq K \lambda_{\min}(C_n)\abs*{\grad f(\hat\theta)}^2,
    \end{equation}
    with a constant $K$ independent of $n$.
    Since $\lambda_{\min}(C_n) \geq \frac{\lambda}{n}$ by~\eqref{eq:lower_bound_collapse},
    for some other constant $\lambda$ independent of $n$,
    we conclude that for any $n \geq N$, it holds
    \[
        - (\vect m_{n+1} - \vect m_N)^\t \grad f(\hat\theta) \geq \left(\sum_{s=N}^n\frac{1}{s}\right){K \lambda}\abs{\grad f(\hat\theta)}^2 \xrightarrow[n \to \infty]{} \infty,
   \]
   which is a contradiction because we assumed $\vect m_n \to \hat \theta$.
   A similar reasoning applies in the continuous-time setting.
\end{proof}

\subsection{\texorpdfstring%
    {Proof of \cref{proposition:convergence_optimization,prop:convergence_rate_1d_opti}}
    {Proof of convergence of the optimization scheme in one dimension}}%

For simplicity, we introduce the ``dimensionless'' notation $\widetilde m = \sqrt{\lhess \beta} (m - \theta_*)$ and $\widetilde C = \lhess \beta C$.
We also introduce
\begin{align*}
    &\widetilde m_{\beta}(\widetilde m, \widetilde C)
    = \sqrt{\lhess \beta}\left(m_{\beta}\left(\theta_* + \frac{\widetilde m}{\sqrt{\lhess \beta}}, \frac{\widetilde C}{\lhess \beta}\right)  - \theta_*\right)
    = \sqrt{\lhess \beta}\bigl(m_{\beta}\left(m, C\right)  - \theta_*\bigr),\\
    &\widetilde C_\beta(\widetilde m, \widetilde C)
    = \lhess \beta \,C_\beta\left(\theta_* + \frac{\widetilde m}{\sqrt{\lhess \beta}}, \frac{\widetilde C}{\lhess \beta}\right)
    =\lhess \beta\, C_\beta(m,C)\,.
\end{align*}
We begin by obtaining auxiliary results.
\begin{lemma}
    [Bound on the weighted mean]
    \label{lemma:fine_bound_first_moment}
   Let $d=1$ and $\beta>0$.
   If \cref{assumption:convexity_potential} is satisfied,
    then it holds
    \begin{align}
        \label{eq:bound_first_moment}
        \forall (\widetilde m, \widetilde C) \in \real \times \real_{> 0}, \qquad
        \abs*{\widetilde m_\beta(\widetilde m, \widetilde  C)}
        \leq  \frac{\abs{\widetilde m}}{1 + \widetilde C}  \left( 1 + 2 \frac{\phi\left(\frac{\abs{\widetilde m}}{\sqrt{\widetilde C(1 + \widetilde C)}} \right)}{\frac{\abs{\widetilde m}}{\sqrt{\widetilde C(1 + \widetilde C)}}}  \right),
    \end{align}
    with $\phi$ the probability density function of the standard normal distribution, i.e.~$\phi = g(\dummy; 0, 1)$.
\end{lemma}
\begin{proof}
    Let $\rho_+(\theta) := \frac{1}{Z_+}1_{[\theta_*, \infty)}(\theta) \rho_{\beta}(\theta)$ and $\rho_-(\theta) := \frac{1}{Z_-}1_{(-\infty, \theta_*)}(\theta) \rho_{\beta}(\theta)$, where $\rho_\beta$ is defined as in \eqref{eq:rhobeta} and $Z_+$, $Z_-$ are the normalization constants.
    It is clear that
    \[
        \mathcal M(\rho_-) \leq \mathcal M(\rho_{\beta}) \leq \mathcal M(\rho_+) \quad \text{ and } \quad \mathcal M(\rho_-) \leq \theta_* \leq \mathcal M(\rho_+).
    \]
    For example, we have
    \begin{align*}
        \mathcal M(\rho_+) - \theta_*
        &= \frac{\int_{\theta_*}^\infty (\theta-\theta_*) \rho_{\beta}(\theta)}{\int_{\theta_*}^\infty \rho_{\beta}(\theta)}
        \geq \frac{\int_{\theta_*}^\infty (\theta-\theta_*) \rho_{\beta}(\theta)}{\int_{-\infty}^\infty \rho_{\beta}(\theta)}
        \geq \frac{\int_{-\infty}^\infty (\theta-\theta_*) \rho_{\beta}(\theta)}{\int_{-\infty}^\infty \rho_{\beta}(\theta)} = \mathcal M(\rho_{\beta}) - \theta_*.
    \end{align*}
    Now notice that,
    since $f(\theta) = f(\theta_*) + \frac{\lhess}{2} \abs*{\theta - \theta_*}^2 + g(\theta)$ for a function $g$ that is nondecreasing on $[\theta_*, \infty)$ and such that $g(\theta_*) = g'(\theta_*) = 0$,
    it holds by~\cref{lemma:auxiliary_lemma_tilted_distribution} that
    \begin{align*}
        \mathcal M(\rho_+) - \theta_*
        &= \frac{\int_{\theta_*}^{\infty} (\theta - \theta_*) \exp \left( - \frac{(\theta-m)^2}{2 C} - \beta f(\theta)  \right) \d \theta} {\int_{\theta_*}^{\infty} \exp \left( - \frac{(\theta-m)^2}{2 C} - \beta f(\theta) \right) \d \theta} \\
        &\leq \frac{\int_{\theta_*}^{\infty} (\theta - \theta_*) \exp \left( - \frac{(\theta-m)^2}{2 C} - \frac{\beta \lhess}{2} \abs*{\theta -\theta_*}^2  \right) \d \theta} {\int_{\theta_*}^{\infty} \exp \left( - \frac{(\theta-m)^2}{2 C} - \frac{\beta\lhess}{2} \abs*{\theta - \theta_*}^2  \right) \d \theta}.
    \end{align*}
    Completing the square in the last expression,
    we obtain
    \begin{align*}
        \mathcal M(\rho_+) - \theta_*
            \leq \frac{\int_{\theta_*}^{\infty} (\theta - \theta_*) \exp \left( - \frac{1}{2} \left(\frac{1}{C} + \beta \lhess \right) \left(  \theta - \frac{ \frac{m}{C}  + \lhess \beta \theta_* }{\frac{1}{C} + \lhess \beta} \right)^2  \right) \d \theta}
            {\int_{\theta_*}^{\infty} \exp \left( - \frac{1}{2} \left(\frac{1}{C} + \beta \lhess \right) \left(  \theta - \frac{ \frac{m}{C}  + \lhess \beta \theta_* }{\frac{1}{C} + \lhess \beta} \right)^2  \right) \d \theta}
            =: D(m, C).
    \end{align*}

    We claim that $D(\dummy, C)$,
    which is the mean of a truncated Gaussian up to the additive constant $\theta_*$,
    is a nondecreasing function for fixed $C$.
    Indeed, let us introduce the function $\mu: (m, C) \mapsto \frac{ m/C  + \lhess \beta \theta_* }{1/C + \lhess \beta}$.
    Since $\mu(m, C)$ is an increasing function of $m$ for fixed $C$,
    it is sufficient to show that the function
    \begin{equation}
        \label{eq:intermediate_mu}
        \mu \mapsto
        \frac{\int_{\theta_*}^{\infty} (\theta - \theta_*) \exp \left( - \frac{1}{2} \left(\frac{1}{C} + \beta \lhess \right) \left(  \theta - \mu \right)^2 \right) \d \theta}
        {\int_{\theta_*}^{\infty} \exp \left( - \frac{1}{2} \left(\frac{1}{C} + \beta \lhess \right) \left(  \theta - \mu \right)^2  \right) \d \theta}
    \end{equation}
    is nondecreasing for fixed $C$.
    To this end, assume that $\mu_1 \leq \mu_2$ and note that
    \begin{align*}
     &\exp \left(-\frac{1}{2} \left(\frac{1}{C} + \beta \lhess \right) |  \theta - \mu_1 |^2  \right)  \\
     &\qquad \qquad \propto \exp \left(-\frac{1}{2}\left(\frac{1}{C} + \beta \lhess \right) |  \theta - \mu_2 |^2  \right)
     \exp \left(- \left(\frac{1}{C} + \beta \lhess \right) \left( \mu_2 - \mu_1 \right) \theta  \right).
    \end{align*}
    Since the second factor is decreasing for $\theta \in [\theta_*, \infty)$,
    we deduce by~\cref{lemma:auxiliary_lemma_tilted_distribution} that the function defined in~\eqref{eq:intermediate_mu} is nondecreasing,
    and therefore $D(\dummy, C)$ is also nondecreasing.

    Using the standard formula for the mean of a truncated normal distribution,
    we deduce
    \[
        D(m, C) =
        \mu(m,C) - \theta_* +
        \frac{\phi\left(\sqrt{\frac{1}{C} + \lhess \beta}\bigl(\theta_* - \mu(m, C)\bigr)\right)}
        {1 - \Phi\left(\sqrt{\frac{1}{C} + \lhess \beta}\bigl(\theta_* - \mu(m, C)\bigr)\right)} \frac{1}{\sqrt{\frac{1}{C} + \lhess \beta}},
    \]
    where $\Phi$ denotes the CDF of the standard normal distribution.
    Using the notation introduced at the beginning of this section and the fact that $\Phi(x) + \Phi(-x) = 1$, this rewrites
    \[
        \sqrt{\lhess \beta}D(m, C)
        = \frac{1}{1 + \widetilde C}
        \left( \widetilde m + \frac{\phi\left(\frac{\widetilde m}{\sqrt{\widetilde C(1 + \widetilde C)}} \right)}{\Phi\left(\frac{\widetilde m}{\sqrt{\widetilde C(1 + \widetilde C)}} \right)} \sqrt{\widetilde C(1 + \widetilde C)} \right)
        =: \widetilde D(\widetilde m, \widetilde C).
    \]
    Since $D(\dummy, C)$ is nondecreasing,
    we deduce that
    \[
        \sqrt{\lhess \beta}\bigl(\mathcal M(\rho_+) - \theta_*\bigr)
        \leq \widetilde D(|\widetilde m|, \widetilde C).
    \]
    Employing the same reasoning for $\mathcal M(\rho_-)$,
    we obtain similarly
    \[
        \sqrt{\lhess \beta}(\mathcal M(\rho_-)-\theta_*)
        \geq -\widetilde D(|\widetilde m|, \widetilde C).
    \]
    Using the fact that $\Phi(x) \geq \Phi(0) = 1/2$ for all $x \geq 0$, and
    \[
     \sqrt{\lhess \beta}(\mathcal M(\rho_-)-\theta_*)\le
     \widetilde m_{\beta}(\widetilde m, \widetilde C)
     \le \sqrt{\lhess \beta}\bigl(\mathcal M(\rho_+) - \theta_*\bigr),
    \]
    we obtain the statement.
\end{proof}

In order to establish \cref{prop:convergence_rate_1d_opti},
we prove the following technical result.
\begin{lemma}
    [Bound on the ratio of weighted moments]
    \label{lemma:auxiliary_result}
    Let $d=1$ and $\beta>0$.
    If~\cref{assumption:convexity_potential,assumption:convexity_potential_above} are satisfied,
    then there exists for all $\varepsilon \in (0, 1)$ a constant
    $\gamma = \gamma(\lhess, \uhess, \varepsilon) > 0$ such that
    \[
        \forall  (\widetilde m, \widetilde C) \in \real \times \real_{>0}, \qquad
        \frac{\abs{\widetilde m_{\beta}(\widetilde m, \widetilde C)}}{\widetilde C_{\beta}(\widetilde m, \widetilde C)^{\frac{1}{r}}} \leq \max \left( \gamma,  \frac{\abs{\widetilde m}}{\widetilde C^{\frac{1}{r}}}\right),
    \]
    where $r = \max\left(\frac{\uhess}{\lhess}, (2 + \varepsilon)\right)$.
\end{lemma}
\begin{proof}
    Using \cref{lemma:lower_bound_second_moment} and \cref{lemma:fine_bound_first_moment}, we deduce
    \begin{align}
        \label{eq:bound_cone}
        \abs{\frac{\widetilde m_{\beta}(\widetilde m, \widetilde C)}{\widetilde C_{\beta}(\widetilde m, \widetilde C)^{\frac{1}{r}}}}
        \leq \abs{\frac{\widetilde m_{\beta}(\widetilde m, \widetilde C)}{\left(\frac{\widetilde C}{1 + \frac{\uhess}{\lhess} \widetilde C}\right)^{\frac{1}{r}}}}
        \leq
        \frac{\abs{\widetilde m}}{\widetilde C^{\frac{1}{r}}} \frac{\left( 1 + r\widetilde C \right)^{\frac{1}{r}}}{1 + \widetilde C}  \left( 1 +2 \frac{\phi\left(\frac{|\widetilde m|}{\sqrt{\widetilde C(1 + \widetilde C)}} \right)}{\frac{|\widetilde m|}{\sqrt{\widetilde C(1 + \widetilde C)}}} \right)
        =: B(\widetilde m, \widetilde C).
    \end{align}
    If $\abs{\widetilde m} \geq \gamma\widetilde C^{1/r}$ for some $\gamma > 0$,
    then it holds that
    \begin{align}
        \label{eq:intermediate_bound}
        \frac{\left( 1 + r\widetilde C \right)^{\frac{1}{r}}}{1 + \widetilde C}
        \left( 1 +2 \frac{\phi\left(\frac{|\widetilde m|}{\sqrt{\widetilde C(1 + \widetilde C)}} \right)}{\frac{|\widetilde m|}{\sqrt{\widetilde C(1 + \widetilde C)}}} \right)
        \leq
        \frac{\left(1 + r \widetilde C \right)^{\frac{1}{r}}} {1 + \widetilde C}
        \left( 1 + 2 \frac{ \phi\left(\frac{\gamma\widetilde C^{\frac{1}{r}}}{\sqrt{\widetilde C(1 + \widetilde C)}} \right) }{\frac{\gamma\widetilde C^{\frac{1}{r}}}{\sqrt{\widetilde C(1 + \widetilde C)}} } \right)
    \end{align}
    since $\phi(z)/z$ is non-increasing.
    We claim that, for $\gamma$ sufficiently large, the right-hand side of this inequality is bounded from above by 1 for all $\widetilde C > 0$.
    Checking this claim is technical but not difficult,
    so we postpone the proof to~\cref{lemma:technical_inequality_opti} in the appendix.
    For such a value of $\gamma$,
    it holds by~\eqref{eq:bound_cone} that if $\abs{\widetilde m} \geq \gamma  \widetilde C^{1/r}$,
    then
    \[
        \frac{\abs{\widetilde m_{\beta}(\widetilde m, \widetilde C)}}{\widetilde C_{\beta}(\widetilde m, \widetilde C)^{\frac{1}{r}}} \leq \frac{\abs{\widetilde m}}{\widetilde C^{\frac{1}{r}}}.
    \]

    On the other hand, since $B(\dummy, \widetilde C)$ is increasing for fixed $\widetilde C$
    (because the function $x \mapsto x + 2 \phi(x)$ is increasing),
    it holds that, if $\abs{\widetilde m} \leq \gamma  \widetilde C^{1/r}$,
    then $B(\abs{\widetilde m}, \widetilde C) \leq B(\gamma  \widetilde C^{1/r}, \widetilde C) \leq \gamma$ by \cref{lemma:technical_inequality_opti} again,
    which proves the result.
\end{proof}

\begin{proof}
    [Proof of \cref{proposition:convergence_optimization}]
    Let us first assume that $\alpha=0$.
    Then, by \eqref{eq:equations_moments_discrete_gaussian},
    since the moments of successive iterates are related by
    \[
        \widetilde m_{n+1} = \widetilde m_{\beta}(\widetilde m_n, \widetilde C_n) \qquad \text{ and } \qquad
        \widetilde C_{n+1} = \widetilde C_{\beta}(\widetilde m_n, \widetilde C_n)
    \]
    for this value of $\alpha$,
    it holds by \cref{lemma:auxiliary_result} that
    \begin{equation}
       \frac{\abs{\widetilde m_{n+1}}}{\widetilde C_{n+1}^{1/r}} \leq \max \left( \gamma,  \frac{\abs{\widetilde m_n}}{\widetilde C_n^{1/r}}\right) \leq \dotsc \leq \max \left( \gamma,  \frac{\abs{\widetilde m_0}}{\widetilde C_0^{1/r}}\right),
       \label{eq:aux1}
    \end{equation}
    which gives directly the convergence of $\widetilde m_{n}$ to 0, in view of the fact that $\widetilde C_n \to 0$ by \cref{lemma:collapse_optimization}.
    In the case where $\alpha \in (0,  1)$,
    the moments of successive iterates are related by the equations
    \begin{align*}
        \widetilde m_{n+1} &= (1 - \alpha) \widetilde m_{\beta}(\widetilde m_n, \widetilde C_n)  + \alpha \widetilde m_n, \\
        \widetilde C_{n+1} &= (1 - \alpha^2) \widetilde C_{\beta}(\widetilde m_n, \widetilde C_n) + \alpha^2 \widetilde C_n,
    \end{align*}
    so clearly
    \begin{align*}
        \abs{\widetilde m_{n+1}}
        &\leq (1 - \alpha) \abs{\widetilde m_{\beta}(\widetilde m_n, \widetilde C_n)}  + \alpha \abs{\widetilde m_n} \\
        &\leq (1 - \alpha) \abs{\widetilde m_{\beta}(\widetilde m_n, \widetilde C_n)}  + \alpha \max \left( \abs{\widetilde m_n}, \gamma \widetilde C_n^{1/r} \right)
        =: \widehat m_{n+1}.
    \end{align*}
    We will now use the technical~\cref{lemma:auxiliary_convergence_opti} in the appendix with parameters
    \[
        (\widehat C_{\beta}, \widehat C_n, \widehat m_{\beta}, \widehat u)
        = \left(\widetilde C_{\beta}(\widetilde m_n, \widetilde C_n), \widetilde C_n, \abs{\widetilde m_{\beta}(\widetilde m_n, \widetilde C_n)}, \max \left( \abs{\widetilde m_n}, \gamma \widetilde C_n^{1/r} \right)\right).
    \]
    Using \cref{lemma:auxiliary_result},
    we check that the assumptions of \cref{lemma:auxiliary_convergence_opti} are satisfied:
    \[
        \frac{\widehat m_{\beta}}{\widehat C_{\beta}^{1/r}}
        = \frac{\abs{\widetilde m_{\beta}(\widetilde m_n, \widetilde C_n)}}{\widetilde C_{\beta}^{1/r}}
        \leq \max \left( \gamma, \frac{\abs{\widetilde m_n}}{\widetilde C_n^{1/r}}\right)
        = \frac{\widehat u}{\widehat C_n^{1/r}},
    \]
    so we deduce that, for $q = 2r$,
    \[
        \frac{\abs{\widetilde m_{n+1}}}{\widetilde C_{n+1}^{1/q}}
        \leq \frac{\widehat m_{n+1}}{\widetilde C_{n+1}^{1/q}}
        \leq \frac{\widehat u}{\widehat C_{n}^{1/q}}
        = \frac{\max \left( \abs{\widetilde m_n}, \gamma \widetilde C_n^{1/r} \right)}{\widetilde C_{n}^{1/q}}
        = \max \left( \frac{\abs{\widetilde m_n}}{\widetilde C_{n}^{1/q}}, \gamma \widetilde C_n^{1/r - 1/q} \right).
    \]
    Since $\widetilde C_n \leq \widetilde C_0$ by \cref{lemma:collapse_optimization},
    this implies
    \[
        \frac{\abs{\widetilde m_{n+1}}}{\widetilde C_{n+1}^{1/q}}
        \leq \max \left( \frac{\abs{\widetilde m_n}}{\widetilde C_{n}^{1/q}}, \gamma \widetilde C_0^{1/r - 1/q} \right)
        \leq  \dotsc \leq \max \left( \frac{\abs{\widetilde m_0}}{\widetilde C_{0}^{1/q}}, \gamma \widetilde C_0^{1/r - 1/q} \right),
    \]
    implying the convergence of $\widetilde m_n\to 0$ with rate $n^{-1/q}$.

    A similar reasoning can be employed to show the convergence in continuous time%
    ; the details are omitted for conciseness.
\end{proof}

\begin{proof}
    [Proof of \cref{prop:convergence_rate_1d_opti}]
    Let us now obtain a convergence rate in the case where $\alpha = 0$.
    To this end,
    the main idea is to express that,
    close to equilibrium, i.e.\ when $C_n \ll 1$ and $\abs{m_n - \theta_*} \ll 1$,
    the algorithm behaves similarly to how it would in a quadratic potential.
    Employing the same reasoning as in the derivation of~\eqref{eq:laplace_quadratic_0} and~\eqref{eq:laplace_quadratic_1},
    now using Taylor's theorem up to higher orders,
    we deduce
    \begin{subequations}
        \begin{align}
            \int_{\real} g(\theta; m, C) \, \e^{-\beta f(\theta)} \, \d \theta = \e^{-\beta f(m)} &\left( 1 + \left( \beta^2 \abs{f'(m)}^2 - \beta f''(m) \right) \frac{C}{2} \right) + R_0(m, C),
            \label{eq:aux0}
            \\
            \int_{\real} (\theta - m) \, g(\theta; m, C) \, \e^{-\beta f(\theta)} \, \d \theta &= - \e^{-\beta f(m)} \beta C f'(m) + R_1(m, C), \label{eq:aux}\\
            \int_{\real} (\theta - m)^2 \, g(\theta; m, C) \, \e^{-\beta f(\theta)} \, \d \theta &= \e^{-\beta f(m)} C \left( 1 + \left( \beta^2 \abs{f'(m)}^2 - \beta f''(m) \right) \frac{3C}{2} \right) + R_2(m, C),
            \nonumber
        \end{align}
    \end{subequations}
    with remainder terms (different from the ones in the proof of \cref{thm:no_bad_convergence}) satisfying
    \[
        \forall m \in (\theta_* - 1, \theta_* + 1), \quad \forall  0 < C \leq 1, \qquad
        \left\{
        \begin{aligned}
            &\abs{R_0(m, C)} \leq K \abs{C}^{2}, \\
            &\abs{R_1(m, C)} \leq K \abs{C}^{2}, \\
            &\abs{R_2(m, C)} \leq K \abs{C}^{3},
        \end{aligned}
        \right.
    \]
    for an appropriate constant $K$.
    We claim that
    \begin{subequations}
    \begin{align}
        \label{eq:limit_m_equation}%
        m_{\beta} (m, C) - \theta_* &=  \bigl(C^{-1} + \beta f''(m)\bigr)^{-1} C^{-1}(m -\theta_*) + R_m(m, C) \\
        \label{eq:limit_C_equation}%
        C_{\beta} (m, C) &= \bigl(C^{-1} + \beta f''(m)\bigr)^{-1} + R_C(m, C),
    \end{align}
    \end{subequations}
    with $R_m$ and $R_C$ satisfying
    \[
        \forall m \in (\theta_* - \overline m, \theta_* + \overline m), \quad \forall 0 < C < \overline C \qquad
        \left\{
        \begin{aligned}
            &\abs{R_m(m, C)} \leq K \left( C^2 + C^2\abs{m - \theta_*} + C \abs{m - \theta_*}^2 \right), \\
            &\abs{R_C(m, C)} \leq K \abs{C}^{3},
        \end{aligned}
        \right.
    \]
    for a possibly different constant $K$ independent of $m$ and $C$ and appropriate positive constants $\overline m$ and $\overline C$.
    For completeness, let us present the details of the proof of~\eqref{eq:limit_m_equation}.
    To simplify the notation, we will write $u(m, C) = \mathcal O\bigl( v(m, C) \bigr)$ to mean that
    there exist constants~$K$, $\widetilde m$ and $\widetilde C$ such that $\abs{u(m, C)} \leq K \, v(m, C)$
    for all $m \in (\theta_* - \widetilde m, \theta_* + \widetilde m)$
    and for all~$0 < C < \widetilde C$.
    It holds, by a Taylor expansion of the function $x \mapsto (1 + x)^{-1}$ around $x = 0$,
    \begin{align*}
        \bigl(C^{-1} + \beta f''(m)\bigr)^{-1} C^{-1}(m -\theta_*)
        &= m - \theta_* - C \beta f''(m)(m - \theta_*) + \mathcal O(C^2 \abs{m - \theta_*}) \\
        &= m - \theta_* - C \beta f'(m) + \mathcal O(C^2 \abs{m - \theta_*} + C \abs{m - \theta_*}^2) \\
        &= m_{\beta}(m, C) - \theta_* + \mathcal O(C^2 + C^2\abs{m - \theta_*} + C \abs{m - \theta_*}^2).
    \end{align*}
    In the second line, we used that $f''(m) (\theta_* - m) = f'(\theta_*) - f'(m) - \frac{1}{2} f'''(\xi) |\theta_* - m|^2$
    by Taylor's theorem, for some appropriate $\xi$.
    Moreover, the third line is a consequence of the estimate
    $$
      |m_{\beta}(m, C)-m+C \beta f'(m)|=\mathcal O(C^2)
    $$
    due to \eqref{eq:aux0}-\eqref{eq:aux}.
    Equation~\eqref{eq:limit_C_equation} can be shown using a similar approach,
    so we will omit its derivation.
    Combining~\eqref{eq:limit_m_equation} and~\eqref{eq:limit_C_equation},
    we deduce
    \begin{align*}
        C_{\beta}(m, C)^{-1} \bigl(m_{\beta}(m, C) - \theta_*\bigr)
        &= \frac{\bigl(C^{-1} + \beta f''(m)\bigr)^{-1} C^{-1}(m -\theta_*) + R_m(m, C)}{\frac{1}{C^{-1} + \beta f''(m)} + R_C(m, C)} \\
        &= \frac{C^{-1} (m - \theta_*) + \bigl(C^{-1} + \beta f''(m)\bigr) R_m(m, C)}
        {1 + \bigl(C^{-1} + \beta f''(m)\bigr)R_C(m, C)} \\
        & = C^{-1} (m - \theta_*) + \mathcal O(C + C \abs{m - \theta_*} + \abs{m -\theta_*}^2).
    \end{align*}
    Now let $(m_n, C_n)$ denote the iterates of the optimization scheme.
    In view of the definition of the $\mathcal O$ notation,
    and since we already showed that $C_n \leq K n^{-1}$ and $|m_n - \theta_*| \leq K n^{- \frac{1}{r}}$ for some positive constant $K$ and some $r > 2$ due to \eqref{eq:aux1},
    the previous equation implies that there exists another constant $K$ and an index $k$ sufficiently large such that,
    for all $n \geq k$,
    \begin{align}
        \label{eq:estimate_near_gaussian}
        \abs{C_{n}^{-1} (m_n - \theta_*)
        - C_k^{-1} (m_k - \theta_*)}
        \le K \sum_{i=k}^{n-1} \Bigl(C_i + \underbrace{C_i \abs{m_i - \theta_*}}_{\text{summable}} + \abs{m_i -\theta_*}^2\Bigr).
    \end{align}
    All the summands are bounded from above by the worst decay given by the last summand $i^{-\frac{2}{r}}$, up to a constant factor.
    Since
    \begin{equation}\label{eq:aux2}
        \sum_{i=k}^{n-1} i^{-\frac{2}{r}} \leq \int_{k-1}^{n-1} x^{-\frac{2}{r}} \, \d x \leq  \widetilde K  n^{1 - \frac{2}{r}},
    \end{equation}
    with $\widetilde K$ a constant independent of $n$ changing from occurrence to occurrence,
    we deduce that the right-hand side of \eqref{eq:estimate_near_gaussian} is controlled by $ \widetilde K  \left(1+n^{1 - \frac{2}{r}}\right)$. Therefore,  using the fact that $C_n \to 0$ with rate $1/n$, we obtain
    \[
        \forall n \geq k, \qquad \abs{m_n - \theta_*} \leq \left( \frac{\widetilde K}{n} \right) C_k^{-1} (m_k - \theta_*) + \widetilde K \left(n^{-1} + n^{- \frac{2}{r}}\right).
    \]
    We have thus upgraded the convergence rate to $n^{-\frac{2}{r}}$.
    This procedure can be repeated until only the first term in the sum on the right-hand side of~\eqref{eq:estimate_near_gaussian} is non-summable,
    leading finally to the estimate
    \[
        \abs{m_n - \theta_*} \leq \widetilde K \left( \frac{\log n}{n}\right),
    \]
    by a similar argument as in \eqref{eq:aux2} applied to the decay $1/i$.
\end{proof}

\subsection{\texorpdfstring{Proof of \cref{thm:convergence}}{Proof of the existence of a steady state}}%

In this section, we analyze the mean-field dynamics \cref{eq:mean_field,eq:iterative_scheme_measures}.
We show, in the convex one-dimensional case,
the existence and uniqueness of a steady state close to the Laplace approximation of the Bayesian posterior at the MAP estimator.
We begin by showing a version of Laplace's method, which is based on reducing all information about the objective function $f$ into the unique smooth and increasing function $\tau:\real\to\real$ satisfying
    \begin{equation}\label{eq:implicit_equation_tau}
        \forall \theta \in \real, \qquad
        f\bigl(\theta_* + \tau(\theta)\bigr) = f(\theta_*) + \theta^2.
    \end{equation}
    with $\tau(0)=0$. For details, see~\cref{lemma:change_of_variable}.
\begin{proposition}[Laplace's method]
    \label{proposition:laplace_principle}
    Let $d=1$.
    Suppose \cref{assumption:convexity_potential,assumption:assumption_f} hold,
    and assume additionally that $\varphi$ is a smooth function such that
    \begin{alignat}{2}
        \label{eq:assumption_laplace}
        &\forall i \in \{0, \dotsc, 2N+2\}, \qquad && \norm*{\derivative*{i}[\varphi]{\theta}}[\infty] \leq M_{\varphi} < \infty,
    \end{alignat}
    for some $N \in \nat$ and some $M_{\varphi} \geq 0$.
    Then, introducing the function $\psi(\theta) = \varphi\bigl(\theta_* + \tau(\theta)\bigr) \, \derivative*{1}[\tau]{\theta}(\theta)$,
    where $\tau$ is the map provided by~\cref{lemma:change_of_variable},
    it holds
    \[
        I_{\beta} := \int_{\real} \e^{-\beta f(\theta)} \, \varphi(\theta) \, \d \theta
        = \e^{-\beta f(\theta_*)}  \left( \sum_{n=0}^{N} \psi_{2n} \, \frac{\Gamma(n + 1/2)}{\beta^{n+1/2}} + R_{\beta} \right),
        \qquad \psi_{2n} := \frac{\derivative*{2n}[\psi]{t}(0)}{(2n)!},
    \]
    and the remainder $R_{\beta}$ satisfies the bound
    \begin{align*}
        \abs{R_{\beta}} \leq \frac{K \, M_{\varphi}}{(\beta - \beta_0)^{N + 3/2}},
    \end{align*}
    for some constants $K=K(f,N)>0$ and $\beta_0=\beta_0(f,N)\ge0$.
\end{proposition}
\begin{proof}
   Applying~\cref{lemma:change_of_variable}, we can use the change of variable $\theta\mapsto \theta_*+\tau(\theta)$ to obtain
    \[
        \forall \varphi \in C^{\infty}(\real), \qquad \int_{\real} \e^{-\beta f(\theta)} \, \varphi(\theta) \, \d \theta
        = \e^{- \beta f(\theta_*)} \int_{\real} \e^{- \beta \theta^2} \, \varphi(\theta_* + \tau(\theta)) \, \derivative*{1}[\tau]{\theta}(\theta) \, \d \theta
        =: \e^{- \beta f(\theta_*)} \widetilde I_{\beta}.
    \]
    By Fa\`a di Bruno's formula (generalized chain rule),
    we have
    \begin{equation*}
        % \label{eq:fa_di_bruno}
        \forall n \in \nat, \qquad
        \derivative*{n}{\theta^n} \Bigl( \varphi\bigl(\theta_* + \tau(\theta)\bigr) \, \derivative*{1}[\tau]{\theta}(\theta) \Bigr)
        = \sum_{i=0}^{n} \derivative*{i}[\varphi]{\theta^i}\bigl(\theta_* + \tau(\theta)\bigr) \, B_{n+1,i+1}\left(\derivative*{1}[\tau]{\theta}(\theta), \dotsc, \derivative*{n-i+1}[\tau]{\theta^n}(\theta)\right),
    \end{equation*}
    where, for $n \in \nat$, the functions $\{B_{n,i}\}_{i \in \{0, \dotsc, n\}}$ are polynomials (more precisely, Bell polynomials) of degree $0$ to $n$.
    By \cref{lemma:change_of_variable},
    there exist a constant $\lambda = \lambda(f, N) \geq 0$ such that
    \[
         \forall i \in \{0, \dotsc, 2N + 3\}, \qquad
         \norm*{\e^{- \lambda \theta^2} \derivative*{i}[\tau]{t}(\theta)}[\infty] < \infty.
    \]
    It is clear, therefore, that
    \begin{align*}
        &\forall n \in \{0, \dotsc, 2N + 2\}, \quad
        \forall i \in \{0, \dotsc, n\}, \\
        &\qquad \norm{\e^{- (i+1) \lambda \theta^2} B_{n+1,i+1}\left(\derivative*{1}[\tau]{\theta}(\theta), \dotsc, \derivative*{n-i+1}[\tau]{\theta^n}(\theta)\right)}[\infty] < \infty.
    \end{align*}
    Combining this inequality with~\eqref{eq:assumption_laplace},
    we deduce that there exists $K = K(f, N)$ such that
    \[
        \forall n \in \{0, \dotsc, 2N + 2\}, \qquad
        \norm{\e^{-\bigl( (2N+3) \lambda\bigr) \theta^2} \derivative*{n}{\theta^n} \Bigl( \varphi\bigl(\theta_* + \tau(\theta)\bigr) \, \derivative*{1}[\tau]{\theta}(\theta) \Bigr)}[\infty]
        \leq K \, M_{\varphi} < \infty.
    \]
    It follows that, in particular, the assumptions of~\cref{lemma:watson} are satisfied for the function $\psi(\theta)$,
    with the parameters $M = K \, M_{\varphi}$ and $\beta_0 = (2N+3) \lambda$.
    By \cref{lemma:watson}, it holds that
    \[
        \widetilde I_{\beta} =
        \sum_{n=0}^{N} \psi_{2n} \, \frac{\Gamma(n + 1/2)}{\beta^{n+1/2}} + R_{\beta},
        \qquad \psi_{2n} := \frac{\derivative*{2n}[\psi]{t}(0)}{(2n)!},
    \]
    where the remainder $R_{\beta}$ satisfies the bound
    \begin{align*}
        \abs{R_{\beta}} \leq \frac{M}{(2N + 2)!} \, \frac{\Gamma(N + 3/2)}{(\beta - \beta_0)^{N + 3/2}},
    \end{align*}
    which concludes the proof.
\end{proof}

In order to prove \cref{thm:convergence},
let us now introduce the following map on $\real \times \real_{>0}$:
\begin{equation}
    \label{eq:phi_beta}
    \Phi_{\beta}:
    \begin{pmatrix}
        m \\
        C
    \end{pmatrix} \mapsto
    \begin{pmatrix}
        m_{\beta}(m, C) \\
        \lambda^{-1} \, C_{\beta}(m, C)
    \end{pmatrix}, \qquad \lambda = (1 + \beta)^{-1}.
\end{equation}
In view of~\cref{lemma:basic_results_steady_states},
existence of a fixed point of $\Phi_\beta$ implies the existence of a steady state solution
both for the iterative scheme~\eqref{eq:iterative_scheme_measures} with any $\alpha \in [0, 1)$ and for the nonlinear Fokker--Planck equation~\cref{eq:mean_field}.
In order to prove the existence of a fixed point of $\Phi_\beta$
we will apply Laplace's method \cref{proposition:laplace_principle}, and therefore need to calculate the coefficients $\psi_{2n}$,
which requires the calculation of the derivatives of the smooth function $\tau$ at 0.
This can be achieved by implicit differentiation of the equation~\eqref{eq:implicit_equation_tau}.
For example, differentiating twice,
we obtain
\[
    \derivative*{1}[\tau]{\theta}(0) = \pm \sqrt{\frac{2}{\derivative*{2}[f]{\theta^2}(\theta_*)}}.
\]
Since, $\tau$ refers here to the unique increasing function such that~\eqref{eq:implicit_equation_tau} holds,
only the positive solution is retained.
Differentiating~\eqref{eq:implicit_equation_tau} again
we obtain
\[
    \derivative*{2}[\tau]{\theta}(0) = - \frac{\derivative*{3}[f]{\theta^3}(\theta_*) }{3\derivative*{2}[f]{\theta^2}(\theta_*)}\, \abs{\derivative*{1}[\tau]{\theta}(0)}^2.
\]

The following result therefore implies the
existence of steady state close to the Laplace approximation of the target distribution
both for the iterative scheme~\eqref{eq:iterative_scheme_measures} with any $\alpha \in [0, 1)$ and for the nonlinear Fokker--Planck equation~\cref{eq:mean_field}.
\begin{proposition}[Existence of a fixed point of $\Phi_{\beta}$]
    \label{proposition:sampling_existence_steady_state_1d}
    Let $d=1$ and assume that \cref{assumption:convexity_potential,assumption:assumption_f} hold.
    Then there exist $\widetilde k = \widetilde k(f)$ and $\widetilde \beta = \widetilde \beta(f)$ such that,
    for all $\beta \geq \widetilde \beta$,
    there exists a fixed point $\bigl(m_{\infty}(\beta), C_{\infty}(\beta)\bigr)$ of $\Phi_{\beta}$ satisfying
    \begin{align*}
        \abs{ m_{\infty}(\beta) - \theta_*}^2 + \abs{\mat C_{\infty}(\beta) - C_*}^2 \leq \abs{\frac{\widetilde k}{\beta}}^2.
    \end{align*}
\end{proposition}
\begin{proof}
    It is clear from the definitions of $m_{\beta}$ and $C_{\beta}$
    that the map $\Phi_{\beta}$ is continuous.
    Our approach in order to show the existence of a fixed point is to use Brouwer's fixed point theorem.
    To this end, let us define
    \begin{align*}
        \varphi_j(\theta) = (\theta - \theta_*)^j \, g(\theta; m, C), \qquad j = 0, 1, \dotsc, J.
    \end{align*}
    Introducing the function $\hat \theta: \real \ni u \mapsto m + \sqrt{C}u$,
    and using the notation $g(u) := g(u;0, 1)$ for conciseness,
    we calculate
    \begin{align*}
        \hat \varphi_j(u)
        := \varphi_j\bigl(\hat \theta(u)\bigr)
        &= \frac{1}{\sqrt{C}} (m - \theta_* + \sqrt{C}u)^j \, g(u; 0, 1) \\
        &= C^{\frac{j-1}{2}} \left(\frac{m - \theta_*}{\sqrt{C}} + u\right)^j \, g(u; 0, 1)
        = C^{\frac{j-1}{2}} \sum_{k=0}^{j} \binom{j}{k} \left(\frac{m - \theta_*}{\sqrt{C}} \right)^k u^{j-k} g(u; 0, 1),
    \end{align*}
    so we deduce
    \[
        \norm{\derivative*{n}[\hat \varphi_j]{u^n}}[\infty]
        \leq K_{j,n} \, C^{\frac{j-1}{2}} \left( 1 + \abs{ \frac{m - \theta_*}{\sqrt{C}} }^j \right)
        = K_{j,n} \, \left( C^{\frac{j-1}{2}} + C^{-\frac{1}{2}}\abs{m - \theta_*}^j \right),
    \]
    for some constant $K_{j,n}$ independent of $m$ and $C$.
    Since $\derivative*{n}[\varphi_j]{\theta^n}(\theta) = C^{-n/2} \derivative*{n}[\hat \varphi_j]{u}\bigl(C^{-1/2} (\theta -m) \bigr)$,
    this directly implies
    \begin{equation}
        \label{eq:bound_varphis}
        \norm{\derivative*{n}[\varphi_j]{\theta^n}}[\infty]
        \leq K_{j,n} \, \left( C^{\frac{j-n-1}{2}} + C^{-\frac{n+1}{2}}\abs{m - \theta_*}^j \right).
    \end{equation}
    Let us take any $R \in (0, C_*)$
    and introduce the notation $u(\beta,m,C) = \mathcal O_R \bigl(v(\beta)\bigr)$ for any functions $u(\beta, m, C)$ and $v(\beta)$ to mean that
    there exist constants $c$ and $\widetilde \beta$ such that
    \[
        \forall (m, C) \in B_R(\theta_*, C_*),
        \quad \forall \beta > \widetilde \beta, \qquad
        \abs{u(\beta,m,C)} \leq c v(\beta),
    \]
    where $B_R(\theta_*, C_*)$ denotes the closed ball of radius~$R$ centered at $(\theta_*, C_*)$.
    Since $R<C_*$, it is clear that, for all $j \in \nat$ and $N \in \nat$,
    the right-hand side of~\eqref{eq:bound_varphis} is bounded from above by a constant over $B_R(\theta_*, C_*)$,
    uniformly in $m$, $C$ and $n\in\{0,\dotsc,2N+2\}$.
    Thus, we can apply Laplace's method, \cref{proposition:laplace_principle}.
    Letting $\psi_j(\theta) = \varphi_j\bigl(\theta_* + \tau(\theta)\bigr) \tau'(\theta)$,
    we calculate
    \begin{subequations}
    \begin{align*}
        \psi_j(0)
        &=  \varphi_j(\theta_*) \tau'(0) = \varphi_j(\theta_*) \sqrt{\frac{2}{f''(\theta_*)}},\\
        \psi_j''(0)
        &=
        \varphi_j''(\theta_*) \tau'(0)^3
        + 3 \varphi_j'(\theta_*) \, \tau''(0) \, \tau'(0)
        + \varphi_j(\theta_*) \, \tau'''(0).
    \end{align*}
    \end{subequations}
    Note that only the first term in the expression of $\psi_2''(0)$ is nonzero. Therefore,
    Laplace's method applied with $N=0$ or $N=1$ gives
    \begin{subequations}
    \begin{align}
        \label{eq:intermediate_bound_sampling_0}
        &\e^{\beta f(\theta_*)}\int_{\real} \e^{- \beta f} g(\theta; m, C) \, \d \theta
        =  g(\theta_*; m, C) \, \Gamma(1/2)  \, \frac{\tau'(0)}{\beta^{1/2}}
            +  \mathcal O_R\left(\frac{1}{\beta^{3/2}}\right)\,, \\
        \label{eq:intermediate_bound_sampling_1}
        &\e^{\beta f(\theta_*)}\int_{\real} (\theta - \theta_*) \e^{- \beta f} g(\theta; m, C) \, \d \theta
        = \mathcal O_R\left(\frac{1}{\beta^{3/2}}\right)\,, \\
        \label{eq:intermediate_bound_sampling_2}
        &\e^{\beta f(\theta_*)}\int_{\real} (\theta - \theta_*)^2 \e^{- \beta f} g(\theta; m, C) \, \d \theta =
             g(\theta_*; m, C) \, \Gamma(3/2) \, \frac{\tau'(0)^3}{\beta^{3/2}}
            +  \mathcal O_R\left(\frac{1}{\beta^{5/2}}\right)\,.
    \end{align}
    \end{subequations}
    Further, $g(\theta_*; m, C)$ is bounded above and below on $B_R(\theta_*,C_*)$ by positive constants.
    Hence, equation~\eqref{eq:intermediate_bound_sampling_1} leads to
    \begin{align}
        \label{eq:mean_estimate}
        m_{\beta}(m, C) &= \frac{\int_{\real} \theta \e^{- \beta f(\theta)} g(\theta; m, C) \, \d \theta}{\int_{\real} \e^{- \beta f(\theta)} g(\theta; m, C) \, \d \theta}
        =\theta_* + \mathcal O_R \left(\beta^{-1}\right).
    \end{align}
    For the covariance term, note that
    \[
        C_{\beta}(m, C) = \frac{\int_{\real} (\theta - \theta_*)^2 \e^{- \beta f(\theta)} g(\theta; m, C) \, \d \theta}{\int_{\real} \e^{- \beta f(\theta)} g(\theta; m, C) \, \d \theta}
        - \bigl(m_{\beta}(m, C) - \theta_*\bigr)^2,
    \]
    which by~\eqref{eq:intermediate_bound_sampling_2} and the equality $\Gamma(1/2) = 2 \Gamma(3/2)$,
    leads to
    \begin{align*}
        \lambda^{-1} C_{\beta}(m, C) &= \frac{1 + \beta}{\beta} \left( \frac{\Gamma(3/2)}{\Gamma(1/2)} \abs{\tau'(0)}^2  +  \mathcal O_R(\beta^{-1}) \right) + \mathcal O_R(\beta^{-2})
                        = \frac{1}{f''(\theta_*)} +  \mathcal O_R(\beta^{-1}).
    \end{align*}
    Consequently, we deduce by definition of $\mathcal O_R$ that
    there exist constants $\beta^{\dagger}$ and $\widetilde k$ such that
    \[
        \quad \forall \beta > \beta^\dagger, \qquad
        \sup_{(m, C) \in B_R(\theta_*, C_*)} \abs{\Phi_{\beta}(m, C) - (\theta_*, C_*)} \leq \frac{\widetilde k}{\beta^\dagger},
    \]
    where $\abs{\dummy}$ denotes the Euclidean norm.
    That is, it holds that $\Phi_{\beta} \bigl(B_{R}(\theta_*, C_*)\bigr) \subset B_{\widetilde k/\beta}(\theta_*, C_*)$ for any $\beta \geq \beta^\dagger$.
    If additionally $\beta \geq \widetilde k/R$, we have $B_{\widetilde k/\beta}(\theta_*, C_*) \subset B_{R}(\theta_*, C_*)$ and so
    \[
        \Phi_{\beta} \bigl(B_{\widetilde k/\beta}(\theta_*, C_*)\bigr) \subset \Phi_{\beta} \bigl( B_{R}(\theta_*, C_*) \bigr) \subset B_{\widetilde k/\beta}(\theta_*, C_*).
    \]
    Consequently, in this case Brouwer's theorem implies the existence of a fixed point of $\Phi_{\beta}$ in~$B_{\widetilde k/\beta}(\theta_*, C_*)$.
    This proves the statement with $\widetilde \beta = \max(\beta^\dagger, \widetilde k/R)$.
\end{proof}
Next, we show that the map $\Phi_{\beta}$ given in~\eqref{eq:phi_beta} is a contraction for sufficiently large $\beta$.
\begin{proposition}
    [$\Phi_{\beta}$ is a contraction]
    \label{prop:contraction}
    Under the same assumptions as in \cref{proposition:sampling_existence_steady_state_1d} and for any $R \in (0, C_*)$,
    there exists a constant $\widehat \beta = \widehat \beta(f, R)$ and $\widehat k = \widehat k(f, R)$ such that,
    for all $\beta \geq \widehat \beta$,
    the map $\Phi_{\beta}$
    is a contraction with constant $\widehat k/\beta$ for the Euclidean norm over the closed ball of radius $R$ centered at $(\theta_*, C_*)$:
    for all $(m_1, C_1)$ and $(m_2, C_2)$ in $B_R(\theta_*, C_*)$,
    it holds that
    \begin{align*}
        \vecnorm{\Phi_{\beta}(m_1, C_1) - \Phi_{\beta}(m_2, C_2)}
         \le \frac{\widehat k}{\beta} \vecnorm{\begin{pmatrix} m_2 \\ C_2 \end{pmatrix} - \begin{pmatrix} m_1 \\  C_1 \end{pmatrix}}.
    \end{align*}
\end{proposition}
\begin{proof}
    We assume without loss of generality that $\theta_* = 0$, which is justified because the method is affine invariant, discussed in \cref{sec:properties},
    and we recall that $\Phi_{\beta}$ relates the moments of successive iterates from~\eqref{eq:mean_field_sdes} with $\alpha = 0$
    when this scheme is initialized at a Gaussian density.
    Let us introduce the notation
    \[
        J_{\beta}(\varphi) = \int_{\real} \varphi(\theta) \, \exp \left( - \frac{\abs{\theta - m}^2}{2C} \right) \e^{- \beta f(\theta)} \, \d \theta.
    \]
    Using the fact that
    \[
        m_{\beta}(m, C) = \frac{J_{\beta}(\theta)}{J_{\beta}(1)}, \qquad C_{\beta} (m, C) = \frac{J_{\beta}(\theta^2)}{\abs{J_{\beta}(1)}} - \abs{m_{\beta}(m, C)}^2
        = \frac{J_{\beta}(\theta^2) J_{\beta}(1) - \abs{J_{\beta}(\theta)}^2}{\abs{J_{\beta}(1)}^2},
    \]
    and noting that
    \[
        \partial_m J_{\beta}(\varphi) = \frac{J_{\beta}\bigl(\varphi(\theta) (\theta - m)\bigr)}{C},
        \qquad
        \partial_C J_{\beta}(\varphi) = \frac{J_{\beta}\bigl(\varphi(\theta) |\theta - m|^2 \bigr)}{2 C^2},
    \]
    we calculate
    \begin{align*}
        \partial_m m_{\beta}
        &= \frac{1}{C \abs{J_{\beta}(1)}^2} \Bigl( J_{\beta}\bigl(\theta^2\bigr) J_{\beta}(1) - \abs{J_{\beta}(\theta)}^2 \Bigr) \,,\\
        \partial_C m_{\beta}
        &= \frac{1}{2 C^2 \abs{J_{\beta}(1)}^2} \Bigl( J_{\beta}\bigl(\theta|\theta - m|^2\bigr) J_{\beta}(1) - J_{\beta}\bigl(\abs{\theta - m}^2\bigr) J_{\beta}(\theta) \Bigr) \\
        &=\frac{1}{2 C^2 \abs{J_{\beta}(1)}^2} \Bigl[ J_{\beta}\bigl(\theta^3\bigr) J_{\beta}(1) - J_{\beta}\bigl(\theta^2\bigr) J_{\beta}(\theta) -2m\left(J_{\beta}\bigl(\theta^2\bigr) J_{\beta}(1) - J_{\beta}(\theta)^2\right) \Bigr]\,, \\
        \partial_m C_{\beta}
        &= \frac{1}{C \abs{J_{\beta}(1)}^2} \Bigl( J_{\beta}\bigl(\theta^3\bigr) J_{\beta}(1) - J_{\beta}(\theta) J_{\beta}(\theta^2) \Bigr) - 2 m_{\beta} \, \partial_m m_{\beta} \,,\\
        \partial_C C_{\beta}
        &= \frac{1}{2 C^2 \abs{J_{\beta}(1)}^2} \Bigl( J_{\beta}\bigl(\theta^2|\theta - m|^2\bigr) J_{\beta}(1) - J_{\beta}\bigl(\abs{\theta - m}^2\bigr) J_{\beta}(\theta^2) \Bigr) - 2 m_{\beta} \, \partial_C m_{\beta}\\
        &= \frac{1}{2 C^2 \abs{J_{\beta}(1)}^2} \Bigl[ J_{\beta}\bigl(\theta^4\bigr) J_{\beta}(1) - J_{\beta}(\theta^2)^2
        - 2m\left(J_{\beta}\bigl(\theta^3\bigr) J_{\beta}(1) - J_{\beta}(\theta^2)J_{\beta}(\theta)
        \right)\Bigr] - 2 m_{\beta} \, \partial_C m_{\beta}\,.
    \end{align*}
    Applying Laplace's method,
    and noting that $\frac{\d^n}{\d\theta^n}\bigl(\theta^jg(\theta; m,C)\bigr)$ vanishes at $\theta=\theta_*=0$ for all $n<j$,
    we obtain that
    \begin{align*}
        & \e^{\beta f(\theta_*)} \int_{\real} \theta^3 \e^{- \beta f} g(\theta; m, C) \, \d \theta = \mathcal O_R\left(\frac{1}{\beta^{5/2}}\right), \\
        &\e^{\beta f(\theta_*)} \int_{\real} \theta^4 \e^{- \beta f} g(\theta; m, C) \, \d \theta = \mathcal O_R\left(\frac{1}{\beta^{5/2}}\right).
    \end{align*}
    Combining these estimates with~\cref{eq:intermediate_bound_sampling_0,eq:intermediate_bound_sampling_1,eq:intermediate_bound_sampling_2} and~\eqref{eq:mean_estimate},
    and using the same notation as in the proof of~\cref{proposition:sampling_existence_steady_state_1d},
    we deduce
    \begin{align*}
        &\partial_m m_{\beta} (m, C)
        = \mathcal O_R(\beta^{-1}),
        &\partial_C m_{\beta} (m, C)
        = \mathcal O_R(\beta^{-1}), \\
        &\partial_m C_{\beta} (m, C)
        = \mathcal O_R(\beta^{-2}),
        &\partial_C C_{\beta} (m, C)
        = \mathcal O_R(\beta^{-2}).
    \end{align*}
    It easily follows that
    \begin{equation}
        \label{eq:gradient}
        D \Phi_{\beta} :=
        \begin{pmatrix}
        \partial_m \Phi_{\beta}^m
        &\partial_C \Phi_{\beta}^m  \\
        \partial_m \Phi_{\beta}^C
        &\partial_C \Phi_{\beta}^C
        \end{pmatrix}
        = \begin{pmatrix}
        \mathcal O_R(\beta^{-1})
        & \mathcal O_R(\beta^{-1}) \\
        \mathcal O_R(\beta^{-1})
        & O_R(\beta^{-1})
        \end{pmatrix}.
    \end{equation}
    Therefore,
    for all $(m_1, C_1) \in B_{R}(\theta_*, C_*)$ and $(m_2, C_2) \in B_{R}(\theta_*, C_*)$,
    it holds
    \begin{align*}
        \vecnorm{\Phi_{\beta}(m_1, C_1) - \Phi_{\beta}(m_2, C_2)}
        &= \abs{ \int_{0}^{1} D \Phi_{\beta}\bigl(m_t, C_t\bigr)  \,
        \begin{pmatrix} m_2 - m_1 \\ C_2 - C_1 \end{pmatrix} \, \d t } \\
        &\leq \int_{0}^{1} \norm{D \Phi_{\beta}\bigl(m_t, C_t\bigr)}  \,
         \, \d t \vecnorm{\begin{pmatrix} m_2 \\ C_2 \end{pmatrix} - \begin{pmatrix} m_1 \\  C_1 \end{pmatrix}},
    \end{align*}
    where $(m_t, C_t)^\t = \bigl(m_1 + t (m_2-m_1), C_1 + t(C_2 - C_1)\bigr)^\t$.
    Since $\norm{D \Phi_{\beta}} = \mathcal O_R(\beta^{-1})$, by~\eqref{eq:gradient},
    this concludes the proof of the statement.
\end{proof}

The proof of \cref{thm:convergence} is now a simple consequence of~\cref{proposition:sampling_existence_steady_state_1d,prop:contraction}.
\begin{proof}
    [Proof of \cref{thm:convergence}]
    Let $\widetilde \beta$ and $\widehat \beta$,
    as well as $\widetilde k$ and $\widehat k$,
    be as given in the statements of
    \cref{proposition:sampling_existence_steady_state_1d,prop:contraction}, respectively.
    Let $\underline \beta(f, R)$ and $k(f, R)$ be defined by
    \[
        \underline \beta = \max\left(\widetilde \beta(f), \widehat \beta(f, R), \frac{\widetilde k}{R}\right),
        \qquad k(f,R) = \max\bigl(\widetilde k(f), \widehat k(f, R)\bigr).
    \]
    By \cref{proposition:sampling_existence_steady_state_1d},
    there exists for all $\beta \geq \underline \beta$ a fixed point of $\Phi_{\beta}$ in $B_{\widetilde k/\beta}(\theta_*, C_*) \subset B_R(\theta_*, C_*)$.
    Since $\Phi_{\beta}$ is a contraction over $B_R(\theta_*, C_*)$ for such value of $\beta$ by \cref{prop:contraction},
    this fixed point is unique in $B_R(\theta_*, C_*)$.
    Let us now show the convergence to the fixed point in the discrete and continuous-time cases.

    \begin{itemize}
        \item[(i)] Case $\alpha \in [0, 1)$.
            We consider the iteration \eqref{eq:equations_moments_discrete},
            \begin{align*}
                m_{n+1} &= \alpha m_n + (1 - \alpha) m_{\beta}(m_n, C_n), \\
                C_{n+1} &= \alpha^2 C_n + (1 - \alpha^2) \lambda^{-1} C_{\beta}(m_n, C_n),
            \end{align*}
            Denoting the fixed point by $(m_{\infty}, C_{\infty})^\t$,
            we rewrite this system as
            \begin{align*}
                m_{n+1} - m_{\infty} &= \alpha (m_n - m_{\infty}) + (1 - \alpha) \bigl(m_{\beta}(m_n, C_n) - m_{\beta}(m_{\infty}, C_{\infty})\bigr), \\
                C_{n+1} - C_{\infty} &= \alpha^2 (C_n - C_{\infty}) + (1 - \alpha^2) \lambda^{-1} \bigl( C_{\beta}(m_n, C_n) - C_{\beta}(m_{\infty}, C_{\infty}) \bigr),
            \end{align*}
            and so, by the triangle inequality,
            \begin{align*}
                \abs{%
                    \begin{pmatrix}
                        m_{n+1} \\
                        C_{n+1}
                    \end{pmatrix}
                    -
                    \begin{pmatrix}
                        m_{\infty}  \\
                        C_{\infty}
                    \end{pmatrix}
                }
        &\leq
        \alpha
        \abs{%
            \begin{pmatrix}
                m_{n} \\
                C_{n}
            \end{pmatrix}
            -
            \begin{pmatrix}
                m_{\infty}  \\
                C_{\infty}
            \end{pmatrix}
        }
        +  (1-\alpha^2)\abs{%
            \Phi_{\beta} (m_n, C_n)
            - \Phi_{\beta} (m_{\infty}, C_{\infty})
        } \\
        &\leq
        \left( \alpha + (1-\alpha^2)\frac{k}{\beta} \right)
        \abs{%
            \begin{pmatrix}
                m_{n} \\
                C_{n}
            \end{pmatrix}
            -
            \begin{pmatrix}
                m_{\infty}  \\
                C_{\infty}
            \end{pmatrix}
        },
    \end{align*}
    from where the statement follows easily.
    \item[(ii)]
        Case $\alpha = 1$.
        Similarly, in the continuous-time setting,
        we can rewrite the equations \eqref{eq:equations_moments} for the moments as
        \begin{align*}
            \dot m(t) &= - (m(t) - m_{\infty}) +  \Bigl(m_{\beta}\bigl(m(t), C(t)\bigr) - m_{\beta}(m_{\infty}, C_{\infty})\Bigr), \\
            \dot C(t) &= - 2(C(t) - C_{\infty}) +  2\lambda^{-1} \Bigl( C_{\beta}\bigl(m(t), C(t)\bigr) - C_{\beta}(m_{\infty}, C_{\infty}) \Bigr).
        \end{align*}
        Therefore
        \begin{align*}
            \frac{1}{2} \derivative*{1}{t}
            \abs{%
                \begin{pmatrix}
                    m(t) \\
                    C(t)
                \end{pmatrix}
                -
                \begin{pmatrix}
                    m_{\infty}  \\
                    C_{\infty}
                \end{pmatrix}
            }^2
        &\leq
        -
        \abs{%
            \begin{pmatrix}
                m(t) \\
                C(t)
            \end{pmatrix}
            -
            \begin{pmatrix}
                m_{\infty}  \\
                C_{\infty}
            \end{pmatrix}
        }^2 \\
        &\quad + 2 \abs{\Phi_{\beta} (m_n, C_n) - \Phi_{\beta} (m_{\infty}, C_{\infty})}
        \abs{%
            \begin{pmatrix}
                m(t) \\
                C(t)
            \end{pmatrix}
            -
            \begin{pmatrix}
                m_{\infty}  \\
                C_{\infty}
            \end{pmatrix}
        }
        \\
        &\leq
        - \left( 1 - \frac{2 k}{\beta} \right)
        \abs{%
            \begin{pmatrix}
                m_{n} \\
                C_{n}
            \end{pmatrix}
            -
            \begin{pmatrix}
                m_{\infty}  \\
                C_{\infty}
            \end{pmatrix}
        }^2,
    \end{align*}
    which leads to the statement by Gr\"onwall's inequality.
    \end{itemize}
\end{proof}

\appendix
\section{Auxiliary Technical Results}%

\begin{lemma}
    \label{lemma:convergence_recursion_equations}%
    Let $(u_n, v_n)$ denote the solution to the recurrence relation~\eqref{eq:1Dmoments}
\begin{subequations}
\begin{align}
    u_{n+1} &= \left[\alpha +(1-\alpha) (1 + v_{n})^{-1}\right] u_{n}, \\
    v_{n+1}  &=
    \left[\alpha^2 + (1-\alpha^2) \lambda^{-1} (1 + v_n )^{-1}\right]v_n,
\end{align}
\end{subequations}
with initial condition $(u_0, v_0)$ and $v_0 > 0$. Denote $v_\infty=(1-\lambda)/\lambda$.
We separate the sampling and optimization cases.
\begin{itemize}
    \item[(i)] Case $\lambda \in (0, 1)$.
        It holds, for all $n \in \nat$, that
        \begin{subequations}
        \begin{align}
            &\min \left( 1, \frac{v_{\infty}}{v_0} \right)^{\frac{1}{1+\alpha}} \left((1 - \alpha) \lambda  + \alpha\right)^n
            \le \abs{\frac{u_n}{u_0}} \leq \max \left( 1, \frac{v_{\infty}}{v_0} \right)^{\frac{1}{1+\alpha}} \bigl((1 - \alpha) \lambda  + \alpha\bigr)^n ,\label{eq:convergence_u} \\
             &\min \left(1, \frac{v_{\infty}}{v_0} \right) \left((1 - \alpha^2) \lambda  + \alpha^2\right)^n
             \leq\abs{\frac{v_n - v_{\infty}}{v_0 - v_{\infty}}}
             \leq \max \left( 1, \frac{v_{\infty}}{v_0} \right) \left((1 - \alpha^2) \lambda  + \alpha^2\right)^n  \,;\label{eq:convergence_v}
        \end{align}
        \end{subequations}
    \item[(ii)] Case $\lambda=1$.
        For all $n \in \nat$, it holds that
        \begin{subequations}
        \begin{align}
            \label{eq:convergence_v_opti_1}
            \left(\frac{1}{1+v_0(1-\alpha^2)n}\right)^{\frac{1}{1+\alpha}} \le
            \abs{\frac{u_n}{u_0}} &\leq \left(  \frac{1+v_0}{1 + v_0 + v_0(1 - \alpha^2)n} \right)^{\frac{1}{1 + \alpha}}
        \end{align}
        and
        \begin{align}
            \label{eq:convergence_v_opti}
            \left(\frac{1}{1+v_0(1-\alpha^2)n}\right)\le
            \frac{v_n}{v_0} &\leq \left(  \frac{1+v_0}{1 + v_0 + v_0(1 - \alpha^2)n} \right).
        \end{align}
        \end{subequations}
    \end{itemize}
\end{lemma}
\begin{proof}
    \textbf{Case $\lambda \in (0, 1)$.~}
    Rearranging the equation for $\{v_n\}_{n=0, \dotsc}$, we obtain
    \begin{equation}
        \label{eq:rearranged_v}
        v_{n+1} - v_{\infty} = \gamma(v_n) (v_{n} - v_{\infty}),
        \qquad \gamma(s) := \frac{1+\alpha^2 s}{1+s}.
    \end{equation}
    If $v_0 \geq v_{\infty}$, then clearly $v_0\ge v_n \geq v_{\infty}$ for all $n \in \nat$.
    Therefore, since $0<\gamma(\dummy)<1$ is a strictly decreasing function on $[0, \infty\bigr)$,
    it holds that $0 \leq v_{n+1} - v_{\infty} \leq \gamma(v_\infty)  (v_{n} - v_{\infty})$
    which leads directly to the convergence estimate
    \[
        \abs{v_{n} - v_{\infty}} \leq \gamma(v_\infty)^n \abs*{v_0 - v_{\infty}}.
    \]
    Similarly, for $v_0<v_\infty$, we obtain $v_0 < v_n < v_{\infty}$ for all $n \in \nat$,
    which leads to the lower bound $\abs{v_{n} - v_{\infty}} \geq \gamma(v_\infty)^n \abs*{v_0 - v_{\infty}}$.
    For the opposite bounds, we calculate using~\eqref{eq:rearranged_v} that
    \begin{align*}
        \frac{v_{n+1}^{-1}(v_{n+1} - v_{\infty})}{v_{n}^{-1}(v_{n} - v_{\infty})}
        &= \left(\frac{1 + \alpha^2 v_n}{1 + v_n}\right) \frac{v_n}{v_{\infty} + \frac{1 + \alpha^2 v_n}{1 + v_n} (v_n - v_{\infty})} \\
      &= \frac{1 + \alpha^2 v_n}{(1 - \alpha^2) v_{\infty} + 1 + \alpha^2 v_n}
      = \frac{1 + \alpha^2 v_{\infty} + \alpha^2(v_n-v_{\infty})}{1 + v_{\infty} + \alpha^2 (v_n-v_{\infty})}\,.
    \end{align*}
    Hence, if $v_0 > v_\infty$, then
    \begin{align*}
     \frac{v_{n+1}^{-1}(v_{n+1} - v_{\infty})}{v_{n}^{-1}(v_{n} - v_{\infty})}
      &\ge \gamma(v_\infty) \quad \Longrightarrow \quad
      \frac{\abs{v_{n+1}-v_\infty}}{v_{n+1}} \ge  \gamma(v_\infty) \frac{\abs{v_n-v_\infty}}{v_n}\,,
    \end{align*}
    with the inequalities reversed in the case $v_0<v_\infty$.
    Iterating the last inequality, combining the above estimates and
    noting that $\gamma(v_\infty)= (1-\alpha^2)\lambda+\alpha^2$ gives~\eqref{eq:convergence_v}.

    Next, notice that the equation for $u_n$ can be rewritten as
    \begin{align}
        \label{eq:equation_un_rearranged}%
        u_{n+1} =  \tilde\gamma(v_n) u_n\,,\qquad \widetilde\gamma(s) := \frac{1+\alpha s}{1+s},
    \end{align}
    where $\widetilde\gamma$ is strictly decreasing on $[0, \infty)$.
   Clearly, if $v_0 \geq v_{\infty}$,
    then we have
    \begin{align}
        \abs{u_{n}}\leq  \widetilde\gamma(v_\infty)^n \abs{u_0}\,,
    \end{align}
    with the reversed inequality holding for $v_0<v_\infty$.
    Noting that $u_n$ and $v_n-v_\infty$ do not change sign with $n$,
    we calculate by analogy with the continuous-time case \cref{lemma:cv-alpha1-1D} that
    \begin{align}
        \frac{\abs{u_{n+1}}}{\abs{u_n}} \left( \frac{\abs{v_{n}-v_\infty}}{\abs{v_{n+1}-v_\infty}} \right)^{\frac{1}{1+\alpha}}
        &=\frac{u_{n+1}}{u_n} \left( \frac{v_{n}-v_\infty}{v_{n+1}-v_\infty} \right)^{\frac{1}{1+\alpha}}
        \notag\\
        &= \left(\frac{1+ \alpha v_n} {1 + v_{n}}\right) \left(\frac{1 + v_{n}} {1+ \alpha^2 v_n}\right)^{\frac{1}{1+\alpha}}
        =: h_{\alpha}(v_n).
        \label{eq:h-alpha}
    \end{align}
    Since  $h_{\alpha}'(s) \geq 0$ for all $\alpha \in (0, 1)$ and all $s> 0$, we deduce for $v_0<v_\infty$,
    \[
    \frac{\abs{u_{n+1}}}{\abs{u_n}} \left( \frac{\abs{v_{n}-v_\infty}}{\abs{v_{n+1}-v_\infty}} \right)^{\frac{1}{1+\alpha}}\le h_\alpha(v_\infty),
    \]
    and iterating this inequality, then using~\eqref{eq:convergence_v}, we have
    \begin{align*}
        |u_n|&\le |u_0| h_\alpha(v_\infty)^n\left(\frac{\abs{v_{n}-v_\infty}}{\abs{v_0-v_\infty}} \right)^{\frac{1}{1+\alpha}}
        \le  |u_0| \left(\frac{v_\infty}{v_0}\right)^{\frac{1}{1+\alpha}}\left(h_\alpha(v_\infty) \gamma(v_\infty)^{\frac{1}{1+\alpha}}\right)^n\\
        &=  |u_0| \left(\frac{v_\infty}{v_0}\right)^{\frac{1}{1+\alpha}}\widetilde \gamma(v_\infty)^n\,.
    \end{align*}
    with reversed inequality of $v_0 > v_\infty$.
    Since $\widetilde \gamma(v_{\infty}) = (1-\alpha)\lambda+\alpha$,
    this concludes the proof of~\eqref{eq:convergence_u}.

    \textbf{Case $\lambda = 1$.~}%
    Rearranging the equation for $v_n$, we have
    \begin{equation}
        \label{eq:bound_v_opti}
        v_{n+1}^{-1}  = \left( \frac{1 + v_n}{1 + \alpha^2 v_n} \right) v_n^{-1} = \gamma(v_n)^{-1}v_n^{-1}\,.
    \end{equation}
    Since clearly
    \begin{equation}
        \label{eq:aux_inequality}
        \forall (x, y) \in \real_+^2, \qquad
        \frac{1 + x}{1 + y}  \leq 1+ \abs{x - y},
    \end{equation}
    we have
    \[
        v_{n+1}^{-1} \leq \left( 1 + (1 - \alpha^2) v_n \right) v_{n}^{-1} \leq v_{n}^{-1} + (1 - \alpha^2),
    \]
    so we obtain a lower bound on $v_n$:
    \begin{equation}
        \label{eq:upper_bound_inverse_cov}
        \forall n \in \nat, \qquad
        v_{n}^{-1} \leq v_{0}^{-1} + (1 - \alpha^2) n =: \underline v_{n}^{-1}.
    \end{equation}
    In order to obtain an upper bound for $v_n$,
    we note that
    \[
        v_{n+1} = \left( \frac{1 + \alpha^2 v_n}{1 + v_n} \right) v_n \leq \left( \frac{1 + \alpha^2 \underline v_n}{1 + \underline v_n} \right) v_n
        = \left( \frac{v_0^{-1} + n (1 - \alpha^2) + \alpha^2}{v_0^{-1} + n (1 - \alpha^2) + 1} \right) v_n.
    \]
    Therefore we deduce
    \[
        v_n \leq \prod_{k=0}^{n-1} \left( 1 - \frac{1 - \alpha^2}{v_0^{-1} + k(1 - \alpha^2) + 1} \right) v_0 =: \Pi_{n-1} v_0
    \]
    Using $\log(1-\epsilon)\le -\epsilon$ for all $\epsilon\in(0,1)$, we have
    \begin{align*}
        \log \Pi_{n-1}
        &\leq - \sum_{k=0}^{n-1}\frac{1 - \alpha^2}{v_0^{-1} + k(1 - \alpha^2) + 1}
        \leq - \int_{0}^{n}\frac{1 - \alpha^2}{v_0^{-1} + x(1 - \alpha^2) + 1} \, \d x \\
        &= - \log \left( \frac{v_0^{-1} + n(1 - \alpha^2) + 1}{v_0^{-1} + 1}\right),
    \end{align*}
    so we conclude that the upper bound in~\eqref{eq:convergence_v_opti} holds.
    A similar reasoning with the inequality
    \[
        \abs{u_{n+1}} \leq \left(\frac{1 + \alpha \underline v_n}{1 + \underline v_n} \right) \abs{u_n}
    \]
    can be employed in order to show the upper bound on $u_n$ in \eqref{eq:convergence_v_opti_1}.
    To obtain the lower bound on $u_n$,
    we use the fact that $h_\alpha$ is increasing to estimate from \eqref{eq:h-alpha} that
    \begin{align*}
        \abs{\frac{u_{n+1}}{u_n}} \abs{\frac{v_n}{v_{n+1}}}^{\frac{1}{1+\alpha}}
        =h_{\alpha}(v_n)\ge h_\alpha(0)=1
        \quad \Leftrightarrow \quad
        \frac{\abs{u_{n+1}}}{v_{n+1}^{\frac{1}{1+\alpha}}} \geq \frac{\abs{u_{n}}}{v_{n}^{\frac{1}{1+\alpha}}}.
    \end{align*}
    By iterating this inequality and using \eqref{eq:upper_bound_inverse_cov}, we conclude
    \[
        \abs{\frac{u_n}{u_0}} \ge \abs{\frac{v_n}{v_0}}^{\frac{1}{1+\alpha}}
        \ge \left(\frac{1}{1+v_0(1-\alpha^2)n}\right)^{\frac{1}{1+\alpha}},
    \]
    which is the result.
\end{proof}

\begin{lemma}
    \label{lemma:cv-alpha1-1D}
    Let $\lambda \in (0, 1]$,
    and let $\bigl(u(t),v(t)\bigr)$ denote the unique global solution to the ODE system~\eqref{eq:1Dmoments-alph1} with initial condition $(u_0,v_0)$ and $v_0 > 0$.
    \begin{itemize}
        \item[(i)] Case $\lambda \in (0, 1)$.
            It holds that
            \begin{subequations}
            \begin{align}
                \label{eq:convergence_u_continuous}
                \min \left(1,  \left(\frac{v_{\infty}}{v_0}\right)^{\lambda/2} \right) \e^{-(1 - \lambda) t} \le
                \abs{\frac{u(t)}{u_0}} &\leq \max \left(1,  \left(\frac{v_{\infty}}{v_0}\right)^{\lambda/2} \right) \e^{-(1 - \lambda) t}, \\
                \label{eq:convergence_v_continuous}
                \min \left( 1,  \left(\frac{v_{\infty}}{v_0}\right)^{\lambda} \right) \e^{-2 (1 - \lambda) t} \le
                    \abs{\frac{v(t) - v_{\infty}}{v_0 - v_{\infty}}} &\leq
                \max \left( 1,  \left(\frac{v_{\infty}}{v_0}\right)^{\lambda} \right) \e^{-2 (1 - \lambda) t}.
            \end{align}
            \end{subequations}
        \item[(ii)] Case $\lambda=1$.
            For all $t \geq 0$, it holds that
            \begin{subequations}
            \begin{align}
                \label{eq:convergence_u_opti_cont}
                \left( \frac{1}{1 + 2v_0 t} \right)^{\frac{1}{2}} \leq \abs{\frac{u(t)}{u_0}} &\leq \left(  \frac{1 + v_0}{1 + v_0 + 2 v_0 t} \right)^{\frac{1}{2}}
            \end{align}
            and
            \begin{align}
             \label{eq:convergence_v_opti_cont}
                \frac{1}{1 + 2v_0 t}\le
                \frac{v(t)}{v_0} &\leq \frac{1 + v_0}{1 + v_0 + 2v_0 t}.
            \end{align}
            \end{subequations}
    \end{itemize}
\end{lemma}

\begin{proof}
Note that solutions to \eqref{eq:1Dmoments-alph1} are unique, and exist globally in time.

    \textbf{Case $\lambda \in (0, 1)$.~}
    We begin with the sampling case (i) when $\lambda\neq 1$,
    The second equation in \eqref{eq:1Dmoments-alph1} can be rewritten as
    \begin{align*}
        % \label{eq:rearranged_continuous_eq}
        \derivative*{1}{t} (v - v_{\infty}) = -2 \left( \frac{v}{v + 1} \right) \left(v- v_{\infty} \right).
    \end{align*}
    For $x = v/v_{\infty}$, we obtain
    \begin{align*}
        \dot x = -2 \left( \frac{(x - 1)x}{v_{\infty}^{-1} + x} \right)
        = -2 \left( \frac{1 + v_{\infty}^{-1}}{x - 1} - \frac{v_{\infty}^{-1}}{x}\right)^{-1} = -2(1 - \lambda) \left( \frac{-1}{1 - x} - \frac{\lambda}{x}\right)^{-1}.
    \end{align*}
    We can rewrite this equation as
    \begin{equation}
        \label{eq:continuous_time_c}
        \derivative*{1}{t} \Bigl( \log \bigl(1 - x(t)\bigr) - \lambda \log \bigl(x(t)\bigr) \Bigr) = - 2(1 - \lambda),
    \end{equation}
    leading to
    \[
        |1 - x(t)|
        = \left( \frac{x(t)}{x(0)} \right)^{\lambda}  \e^{-2 (1 - \lambda) t} |1 - x(0)|.
    \]
    Since $v(t)$ is decreasing if $v_0 > v_{\infty}$ and increasing if $v_0 < v_{\infty}$,
    estimate ~\eqref{eq:convergence_v_continuous} directly follows.

   Next, we consider the first equation in \eqref{eq:1Dmoments-alph1},
   and note that it can be rewritten as
   \[
       \frac{\dot u}{u} =  \frac{1}{2} \left( \frac{1}{v - v_{\infty}} \right) \derivative*{1}{t} (v - v_{\infty}).
   \]
   This implies
   \begin{equation}
       \label{eq:relation_u_v}
       \log \left( \frac{u(t)}{u_0} \right) = \frac{1}{2} \log \left( \frac{v(t)-v_{\infty}}{v_0-v_{\infty}}\right),
   \end{equation}
   where it is not difficult to verify that the arguments of the logarithms are positive for all times.
   Applying~\eqref{eq:convergence_v_continuous}, we conclude that~\eqref{eq:convergence_u_continuous} holds.

   \textbf{Case $\lambda = 1$.~}
    The argument follows analogously;
    the second equation in \eqref{eq:1Dmoments-alph1} reads
    \begin{align}
        \label{eq:rearranged_continuous_eq_opti}
        \dot v = -2 \left( \frac{v}{v + 1} \right) v,
    \end{align}
    Since the right-hand is bounded from below by $-2 v^2$,
    we directly deduce that
    \begin{equation}
        \label{eq:lower_bound_v_continuous}
        \forall t > 0, \qquad v(t) \geq \frac{1}{v_0^{-1} + 2 t} := \underline v(t).
    \end{equation}
    Now, since the function $s \mapsto \frac{s}{1 + s}$ is increasing,
    it is clear that $v(t)$ satisfies
    \[
        \dot v(t) \leq -2 \left( \frac{\underline{v}(t)}{\underline{v}(t) + 1} \right) v(t).
    \]
    Using Gr\"onwall's inequality, we obtain the upper bound in~\eqref{eq:convergence_v_opti_cont}.

    The bounds on $u(t)$ are then obtained from~\eqref{eq:relation_u_v} and the bounds on $v(t)$.
\end{proof}

\begin{remark}
    Notice that, by letting $\alpha = \e^{-t/n}$ in the bounds obtained in \cref{lemma:convergence_recursion_equations}
    and taking the limit $n \to \infty$,
    we recover the bounds in \cref{lemma:cv-alpha1-1D}.
\end{remark}
\begin{remark}
    It is possible to slightly improve the upper bounds in~\eqref{eq:convergence_v_opti} and~\eqref{eq:convergence_v_opti_cont}.

\begin{itemize}
    \item
    In the discrete-time case,
    rearranging the equation for $v_{n+1}$ and using that $\log(1 + \varepsilon) \geq \frac{\varepsilon}{1 + \varepsilon}$ for all $\varepsilon > 0$,
    we have
    \begin{align}
        \notag
        &v_{n+1}^{-1} - \log(v_{n+1}) - v_{n}^{-1} + \log(v_n) \\
        &\qquad = \frac{1 - \alpha^2}{1 + \alpha^2 v_n} + \alpha^2 \log\left( \frac{1 + v_n}{1 + \alpha^2 v_n} \right)
        \notag
        = \frac{1 - \alpha^2}{1 + \alpha^2 v_n} + \log\left( 1 + \frac{(1 - \alpha^2)v_n}{1 + \alpha^2 v_n} \right) \\
        &\qquad \geq \frac{1 - \alpha^2}{1 + \alpha^2 v_n} + \frac{(1 - \alpha^2) v_n}{1 + v_n}
        \label{eq:crude_estimate_vn}
        \geq (1 - \alpha^2) \left( \frac{1}{1 + \alpha^2 v_n} + \frac{ v_n}{1 + v_n} \right) \geq 1 - \alpha^2.
    \end{align}
    Since $v_n$ is decreasing with $n$,
    this directly implies, using the lower bound~\eqref{eq:upper_bound_inverse_cov},
    \[
        v_{n}^{-1} \geq v_0^{-1} + n (1 - \alpha^2) + \log \left( \frac{v_n}{v_0} \right)
        \geq v_0^{-1} + (1 - \alpha^2) n - \log \left( 1 + v_0 (1-\alpha^2) n \right).
    \]
    so we deduce the following inequality:
    \[
        \frac{v_n}{v_0} \leq \frac{1}{1 + v_0 (1 - \alpha^2) n - v_0\log \left( 1 + v_0 (1-\alpha^2) n \right)}
    \]
    which holds for $n\in\nat$ large enough to ensure that the right-hand side is strictly positive.

    \item
    In the continuous-time case,
    one may rewrite~\eqref{eq:rearranged_continuous_eq_opti} as
    \[
        \derivative*{1}{t} \left( \log v(t) - \frac{1}{v(t)} \right)
        = -2
    \]
    Integrating, rearranging and taking reciprocals,
    we obtain
    \[
        v(t) = \frac{1}{v_0^{-1} + 2t + \log \left( \frac{v}{v_0} \right) }.
    \]
    Using the lower bound~\eqref{eq:lower_bound_v_continuous} to bound the argument of the logarithm,
    we obtain
    \[
        v(t) \leq \frac{v_0}{1 + 2v_0 t - v_0\log(1 + 2 v_0 t)}.
    \]
    \end{itemize}
Though slightly better in the long time limit,
these bounds are more cumbersome to manipulate than the ones presented in~\cref{lemma:convergence_recursion_equations,lemma:cv-alpha1-1D}.
\end{remark}

\begin{lemma}
    \label{lemma:auxiliary_lemma_tilted_distribution}
    Assume that $\mu$ is a probability measure on $\bigl(\real, \mathcal B(\real)\bigr)$,
    with $\mathcal B(\real)$ the Borel $\sigma$-algebra on $\real$,
    and that $f:\real\to\real$ is a positive and nondecreasing (resp.\ nonincreasing) function.
    Let $\widetilde \mu$ be the probability measure defined by
    \[
        \widetilde \mu: \mathcal B(\real) \ni A \mapsto \frac{\int_{A} f(x) \d \mu(x)}{\int_{\real} f(x) \, \d \mu(x)}.
    \]
    Then it holds $\expect_{X \sim \widetilde \mu}(X) \geq \expect_{X \sim \mu}(X)$
    (resp.\ $\expect_{X \sim \widetilde \mu}(X) \leq \expect_{X \sim \mu}(X)$).
\end{lemma}
\begin{proof}
    Let us assume that $f$ is nondecreasing,
    and let us denote the cumulative distribution functions (CDFs) by
    $F(x) := \proba_{X \sim \mu} (X \leq x)$ and $\widetilde F(x) := \proba_{X \sim \widetilde \mu} (X \leq x)$.
    For any probability measure $\nu$ with CDF $F_{\nu}$, it holds
    \[
        \expect_{X \sim \nu}(X) = \int_{0}^{\infty} 1 - F_{\nu}(x) - F_{\nu}(-x) \, \d x,
    \]
    so it is sufficient to show $\widetilde F(x) \leq F(x)$ for all $x \in \real$.
    If $\widetilde F(x) = 0$, this inequality is clearly satisfied,
    so let us verify the inequality for any $x$ such that $\widetilde F(x) > 0$.
    For such a value of $x$,
    employing the fact that $f$ is nondecreasing, we obtain
    \[
        \frac{1 - \widetilde F(x)}{\widetilde F(x)}
        = \frac{\int_{(x, \infty)} f(y) \, \d \mu(y)}{\int_{(-\infty, x]} f(y) \, \d \mu(y)}
        \geq \frac{\int_{(x, \infty)} f(x) \, \d \mu(y)}{\int_{(-\infty, x]} f(x) \, \d \mu(y)}
        = \frac{\mu\bigl((x, \infty)\bigr)}{\mu\bigl((-\infty, x]\bigr)}
        = \frac{1 - F(x)}{F(x)}.
    \]
    Applying the function $y \mapsto \frac{1}{1+y}$ to both sides of this inequality,
    and flipping the direction of the inequality accordingly (because this function is decreasing over $[0, \infty)$),
    we obtain the desired inequality $\widetilde F(x) \leq F(x)$.
\end{proof}

\begin{lemma}
    \label{lemma:technical_inequality_opti}
    Let $r > 2$ be given.
    There exists $\gamma > 0$ sufficiently large such that
    \[
        \forall \widetilde C > 0, \qquad
        h(\widetilde C; \gamma) :=
        \frac{\left(1 + r \widetilde C \right)^{\frac{1}{r}}} {1 + \widetilde C}
        \left( 1 + 2 \frac{ \phi\left(\frac{\gamma\widetilde C^{\frac{1}{r}}}{\sqrt{\widetilde C(1 + \widetilde C)}} \right) }{\frac{\gamma\widetilde C^{\frac{1}{r}}}{\sqrt{\widetilde C(1 + \widetilde C)}} } \right)
        \leq 1,
    \]
    where $\phi$ denotes the density of the standard normal distribution,
    i.e.\ $\phi = g(\dummy; 0, 1)$.
\end{lemma}
\begin{proof}
    If $\widetilde C \geq 1$, then
    \[
        h(\widetilde C, \gamma) \leq
            \frac{\left(1 + r \widetilde C \right)^{\frac{1}{r}}} {1 + \widetilde C}
            \left( 1 + 2 \frac{ \phi\left(0\right) }{\frac{\gamma\widetilde C^{\frac{1}{r}}}{\sqrt{\widetilde C(1 + \widetilde C)}} } \right)
        \leq \frac{\left(1 + r \widetilde C \right)^{\frac{1}{r}}} {1 + \widetilde C}
        + \frac{2 \phi(0)}{\gamma}  \frac{\left(1 + r \widetilde C \right)^{\frac{1}{r}}}{\widetilde C^{\frac{1}{r}}}
        \sqrt{\frac{\widetilde C}{1 + \widetilde C}}.
    \]
    By concavity of $\widetilde C \mapsto (1 + r \widetilde C)^{\frac{1}{r}}$,
    and the fact that the first term is strictly decreasing, we have
    \[
        h(\widetilde C, \gamma)
        \leq \frac{\left(1 + r \right)^{\frac{1}{r}}} {2}
        + \frac{2 \phi(0)}{\gamma}  \left(\frac{1 +  (r\widetilde C)^{\frac{1}{r}}}{\widetilde C^{\frac{1}{r}}}\right)
        \leq \frac{\left(1 + r \right)^{\frac{1}{r}}} {2}
        + \frac{2 \phi(0)}{\gamma}  \left(1 + r^{\frac{1}{r}}\right).
    \]
    Since the first term is strictly less than 1,
    there exists $\gamma$ sufficiently large such that the right-hand side is bounded from above by 1.

    If $0<\widetilde C <1$,
    on the other hand,
    we have
    \[
        h(\widetilde C, \gamma) \leq
        \frac{\left(1 + r \widetilde C \right)^{\frac{1}{r}}} {1 + \widetilde C}
        \left( 1 + \frac{4}{\gamma}  \phi\left(\frac{\gamma\widetilde C^{\frac{1}{r}}}{\sqrt{2\widetilde C}} \right) \right).
    \]
    Therefore,
    \[
        \log\bigl(h(\widetilde C, \gamma)\bigr)
        \leq
        \frac{1}{r} \log(1 + r \widetilde C) - \log(1 + \widetilde C)
        + \log \left( 1 + \frac{4}{\gamma}  \phi\left(\frac{\gamma\widetilde C^{\frac{1}{r}}}{\sqrt{2\widetilde C}} \right) \right).
    \]
    The sum of the first two terms is bounded as follows (where we employ that $\widetilde C \leq 1$):
    \begin{align*}
        \frac{1}{r} \log(1 + r \widetilde C) - \log(1 + \widetilde C)
        &= \int_{0}^{\widetilde C} \left( \frac{1}{1 + r x} - \frac{1}{1 + x} \right) \, \d x \\
        &\leq - (r - 1) \int_{0}^{\widetilde C} \frac{x}{2(1+r)} \, \d x
        = - \frac{1}{4} \left( \frac{r-1}{r+1} \right) \widetilde C^2.
    \end{align*}
    Employing this estimate together with the elementary bound $\log(1 + \varepsilon) \leq \varepsilon$,
    we have
    \begin{align*}
        \log\bigl(h(\widetilde C, \gamma)\bigr)
        \leq
        - \frac{1}{4} \left( \frac{r-1}{r+1} \right) \widetilde C^2
    + \frac{4}{\gamma}  \phi\left(\frac{\gamma}{\sqrt{2}} \widetilde C^{- \frac{r - 2}{2r}} \right).
    \end{align*}
    Clearly, there exists $K$ such that $\phi(x) \leq K (1 + x)^{- \frac{4r}{r-2}}$ uniformly,
    so we deduce
    \begin{align*}
        \log\bigl(h(\widetilde C, \gamma)\bigr)
        \leq
        - \frac{1}{4} \left( \frac{r-1}{r+1} \right) \widetilde C^2
        + \frac{4K}{\gamma} \left( \frac{\sqrt{2}}{\gamma}\right)^{\frac{4r}{r-2}} \widetilde C^{2}.
    \end{align*}
    It is possible to choose $\gamma$ sufficiently large such that the right-hand side of this equation is bounded from above by 0 for $\widetilde C \in (0, 1]$,
    and the statement then follows easily.
\end{proof}

\begin{lemma}
    \label{lemma:auxiliary_convergence_opti}
    Assume that $\alpha \in [0, 1]$ and that $\widehat C_{\beta}$, $\widehat C_n$, $\widehat m_{\beta}$ and $\widehat u$ are nonnegative real numbers satisfying
    $0 < \widehat C_{\beta} \leq \widehat C_n$ and
    \[
        \frac{\widehat m_{\beta}}{\widehat C_{\beta}^{1/r}} \leq \frac{\widehat u}{\widehat C_{n}^{1/r}}
    \]
    for some $r \geq 2$.
    Then $(\widehat m_{n+1}, \widehat C_{n+1})$ defined by
    \begin{align*}
        \widehat m_{n+1} &= (1 - \alpha) \widehat m_{\beta} + \alpha \widehat u, \\
        \widehat C_{n+1} &= (1 - \alpha^2) \widehat C_{\beta} + \alpha^2 \widehat C_n
    \end{align*}
    satisfy
    \[
        \frac{\widehat m_{n+1}}{\widehat C_{n+1}^{1/2r}} \leq \frac{\widehat u}{\widehat C_{n}^{1/2r}}.
    \]
\end{lemma}
\begin{proof}
    Letting $m_{n+1} = \widehat m_{n+1} / \widehat u$, $C_{n+1} = \widehat C_{n+1} / \widehat C_n$,
    $m_{\beta} = \widehat m_{\beta} / \widehat u$, and $C_{\beta} = \widehat C_{\beta} / \widehat C_n$,
    we can rewrite the equations for $\widehat m_{n+1}$ and $\widehat C_{n+1}$ as
    \begin{align*}
         m_{n+1} &= (1 - \alpha)  m_{\beta} + \alpha, \\
         C_{n+1} &= (1 - \alpha^2)  C_{\beta} + \alpha^2.
    \end{align*}
    By the assumptions, it holds that $C_{\beta} \le  1$ and
    \(
        m_\beta \leq C_{\beta}^{1/r},
    \)
    and so
    \[
        \frac{m_{n+1}^{2r}}{C_{n+1}}
        = \frac{\bigl((1 - \alpha)  m_{\beta} + \alpha\bigr)^{2r}}{(1 - \alpha^2)  C_{\beta} + \alpha^2}
        \leq \frac{\bigl((1 - \alpha)  x + \alpha\bigr)^{2r}}{(1 - \alpha^2)  x^r + \alpha^2} =: h(x, \alpha),
        \qquad x := C_{\beta}^{1/r} \in (0, 1].
    \]
    We claim that
    \begin{equation}
        \label{eq:claim}
        \forall (y, \alpha) \in (0, 1] \times [0, 1), \qquad
        \partial_{x} h(y, \alpha) \geq 0.
    \end{equation}
    This will imply that
    \(
        h(x, \alpha) = h(1, \alpha) - \int_{x}^{1} \partial_{x} h(y, \alpha) \, \d y  \leq h(1, \alpha) = 1
    \)
    and thus $m_{n+1}^{2r} \leq C_{n+1}$, giving the statement.
    Let us now prove~\eqref{eq:claim}.
    A simple calculation gives
    \begin{align*}
        \sign\bigl(\partial_{x} h(y, \alpha)\bigr)
        &= \sign \Bigl( 2r (1 - \alpha) \bigl((1 - \alpha^2)  y^r + \alpha^2\bigr)- r (1-\alpha^2) y^{r-1}\bigl((1 - \alpha)  y + \alpha\bigr) \Bigr) \\
        &= \sign \Bigl( 2 \bigl((1 - \alpha^2)  y^r + \alpha^2\bigr)- (1+\alpha) y^{r-1}\bigl((1 - \alpha)  y + \alpha\bigr) \Bigr) \\
        &= \sign \Bigl( \alpha^2 \left( 2 - y^r - y^{r-1}\right) - \alpha  y^{r-1} + y^r\Bigr) =: \sign\bigl(g(y, \alpha)\bigr).
    \end{align*}
    The argument of the sign function in the last line, i.e.\ $g(y, \alpha)$,
    is a quadratic function of $\alpha$ with a minimizer at $\alpha_*(y) = \frac{1}{2}y^{r-1} (2 - y^r - y^{r-1})^{-1}$.
    If $\alpha_*(y) \geq 1$, then $g(y, \alpha) \geq g(y, 1) \geq 0$.
    On the other hand, for any $y$ such that $\alpha_*(y) \leq 1$,
    it holds
    \[
        \forall \alpha \in [0, 1], \qquad
        g(y, \alpha) \geq g(y, \alpha_*)
        =  y^r \left( 1 - \frac{1}{2y} \left(\frac{\frac{1}{2}y^{r-1}}{2 - y^r - y^{r-1}}\right) \right).
    \]
    If $y \in (0, \frac{1}{2}]$,
    a direct bound of the right-hand side of the previous equation shows that $g(y, \alpha_*)~\geq~0$,
    and if $y \geq 1/2$ we have by the constraint $\alpha_*(y) \leq 1$ that
    \[
        g(y, \alpha) \geq g(y, \alpha_*)
        \geq  y^r \left( 1 - \frac{1}{2y}\right) \geq 0,
    \]
    which concludes the proof of~\eqref{eq:claim}.
\end{proof}

\begin{lemma}
    [Generalization of Watson's lemma with bound on remainder]
    \label{lemma:watson}
    Assume that $\phi$ is a smooth function satisfying
    \begin{equation}
        \label{eq:assumption_phi}%
            M := \norm{\e^{- \beta_0 \theta^2} \derivative*{2N + 2}[\phi]{\theta}(\theta)}[\infty] < \infty.
    \end{equation}
    for some constant $\beta_0 \in \real$ and $N \in \nat$.
    Then for $\beta > \beta_0$ it holds
    \[
        I_{\beta} :=  \int_{-\infty}^{\infty} \e^{- \beta \theta^2} \, \phi(\theta) \, \d \theta =
        \sum_{n=0}^{N} \phi_{2n} \, \frac{\Gamma(n + 1/2)}{\beta^{n+1/2}} + R_{\beta},
        \qquad \phi_{2n} := \frac{\derivative*{2n}[\phi]{\theta}(0)}{(2n)!},
    \]
    where the remainder $R_{\beta}$ satisfies the bound
    \begin{align*}
        \abs{R_{\beta}} \leq \frac{M}{(2N + 2)!} \, \frac{\Gamma(N + 3/2)}{(\beta - \beta_0)^{N + 3/2}}.
    \end{align*}
\end{lemma}
\begin{proof}
    We follow here the approach of~\cite[Chapter 2]{MR2238098}.
    We first notice that
    \begin{align*}
        I_{\beta}  &= 2 \, \int_{0}^{\infty} \e^{- \beta \theta^2} \, \left( \frac{\phi(\theta) + \phi(-\theta)}{2} \right) \, \d \theta
        =: 2 \int_{0}^{\infty} \e^{- \beta \theta^2} \, \psi(\theta) \, \d \theta\,.
    \end{align*}
    The function $\psi$ is even and smooth,
    all its odd derivatives vanish at $\theta = 0$.
    Therefore, by Taylor's theorem,
    for any $\theta \geq 0$ there exists $\xi(\theta) \in [0, \theta]$ such that
    \[
        \psi(\theta) = \sum_{n=0}^{N} \phi_{2n} \, \theta^{2n} + \frac{\derivative*{2N + 2}[\psi]{\theta}\bigl(\xi(\theta)\bigr)}{(2N + 2)!} \, \theta^{2N + 2}.
    \]
    With a change of variables $\sigma=\theta^2$,
    this leads to
    \begin{align*}
        I_{\beta} &= \sum_{n=0}^{N} \, \phi_{2n} \, \int_{0}^{\infty} \e^{- \beta \sigma} \, \sigma^{n - 1/2} \, \d \sigma + R_{\beta} = \sum_{n=0}^{N} \phi_{2n} \, \frac{\Gamma(n + 1/2)}{\beta^{n+1/2}} + R_{\beta},
    \end{align*}
    where, by~\eqref{eq:assumption_phi} and for $\beta>\lambda_0$, the remainder term is bounded from above as follows:
    \begin{align*}
        \abs{R_{\beta}}
                          &\leq \frac{M}{(2N + 2)!} \, \int_{0}^{\infty} \e^{-(\beta - \beta_0) \sigma} \, \sigma^{N+1/2} \, \d \sigma
                          = \frac{M}{(2N + 2)!} \, \frac{\Gamma(N + 3/2)}{(\beta - \beta_0)^{N + 3/2}},
    \end{align*}
    which concludes the proof.
\end{proof}

\begin{lemma}
    \label{lemma:change_of_variable}
    Suppose that \cref{assumption:convexity_potential,assumption:assumption_f} are satisfied.
    Then there exists a unique smooth and increasing function $\tau(\theta)$ such that
    \[
        \forall \theta \in \real, \qquad
        f\bigl(\theta_* + \tau(\theta)\bigr) = f(\theta_*) + \theta^2.
    \]
    In addition, the function $\tau$ and all its derivatives are bounded from above by the reciprocal of a Gaussian,
    in the sense that for all $i \in \{0, 1, 2, \dotsc\}$ there exists $\mu_i \in \real$ such that
    \begin{equation*}
        \norm{\e^{- \mu_i \theta^2} \derivative*{i}[\tau]{\theta}(\theta)}[\infty] < \infty.
    \end{equation*}
\end{lemma}
\begin{proof}
    Introducing $g(\theta) := f(\theta + \theta_*) - f(\theta_*)$,
    we must prove the existence of a function $\tau$ satisfying
    \begin{equation}
        \label{eq:equation_tau}
        \forall \theta \in \real,
        \qquad g\bigl(\tau(\theta)\bigr) = \theta^2.
    \end{equation}
    By assumption $\derivative*{2}[g]{\theta}(\theta) \geq \lhess$,
    so $g(\theta) \geq \lhess \, \theta^2/2$ and $\abs{\derivative*{1}[g]{\theta}(\theta)} \geq \lhess |\theta|$ for all $\theta \in \real$.
    This implies that the preimage set $g^{-1}(\theta^2)$ contains exactly two elements for any value of $\theta \neq 0$,
    a positive one $g^{-1}_{+}(\theta^2)$ and a negative one $g^{-1}_{-}(\theta^2)$. Further, the preimage $g^{-1}(0)$ is simply $\{0\}$.
    If $\tau$ satisfies~\eqref{eq:equation_tau} and is increasing,
    then it holds necessarily that
    \[
        \tau(\theta) =
        \begin{cases}
            g^{-1}_-(\theta^2) &\qquad \text{ if $\theta < 0$}, \\
            0 &\qquad \text{ if $\theta = 0$}, \\
            g^{-1}_+(\theta^2) &\qquad \text{ if $\theta > 0$}.
        \end{cases}
    \]
    By the inverse function theorem,
    we observe that $g^{-1}_+$ and $g^{-1}_-$ are smooth on $(0, + \infty)$,
    because~$g$ is smooth and strictly monotonic over $(-\infty, 0)$ and $(0, \infty)$,
    and consequently $\tau$ is smooth on $(- \infty, 0)$ and $(0, \infty)$.
    Therefore, in order to show that $\tau$ is a smooth function over $\real$,
    it is sufficient to verify that $\tau$ is also infinitely differentiable in a neighborhood of $\theta = 0$.
    To this end, we define, analogously to~\cite[Chapter 3]{MR2238098},
    \[
        G(u, \theta) =
        \begin{cases}
            \frac{g(u \theta)}{\theta^2} - 1 \quad & \text{if $\theta \neq 0$}, \\
            \frac{u^2}{2} g''(0) - 1 & \text{if $\theta = 0$}.
        \end{cases}
    \]
    The function $G$ is smooth over $\real^2$
    and it is simple to verify that $G(u^*,0) = 0$ for $u^* = \sqrt{2/g''(0)}$ and $\partial_u G(u^*, 0) = u^* g''(0) > 0$.
    Therefore, the implicit function theorem implies the existence of a unique smooth function $\hat u(\theta)$,
    defined on an interval~$(-\varepsilon, \varepsilon)$,
    such that $\hat u(0) = u^*$ and $G\bigl(\hat u(\theta), \theta\bigr) = 0$ for any $\theta \in (-\varepsilon, \varepsilon)$.
    Since the function $\hat \tau : (-\varepsilon, \varepsilon) \ni \theta \mapsto \hat u(\theta) \theta$ satisfies $g(\hat \tau(\theta)) = \theta^2$ by construction,
    and since it is increasing for $\varepsilon$ sufficiently small because $\hat u(0) > 0$,
    this function must necessarily coincide with $\tau$ on the interval $(- \varepsilon, \varepsilon)$,
    implying that $\tau$ is indeed smooth over~$\real$.

    Now note that, since the function $f$ and its derivatives are bounded by the reciprocal of a Gaussian by assumption,
    then clearly so are the function $g$ and its derivatives;
    for any $i \in \{0, 1, 2, \dotsc\}$,
    there exists $r_i$ such that
    \[
        \norm{\e^{-r_i \theta^2} \derivative*{i}[g]{\theta}(\theta)}[\infty] < \infty.
    \]
    Differentiating~\eqref{eq:equation_tau} repeatedly,
    we obtain
    \begin{subequations}
        \label{eq:derivatives_of_tau}
        \begin{align}
        \label{eq:derivatives_of_tau_1}
            &\derivative*{1}[g]{\theta}\bigl(\tau(\theta)\bigr) \, \derivative*{1}[\tau]{\theta}(\theta) = 2 \theta \\
        \label{eq:derivatives_of_tau_2}
        &\derivative*{2}[g]{\theta^2}(\tau(\theta)) \, \abs{\derivative*{1}[\tau]{\theta}(\theta)}^2 + \derivative*{1}[g]{\theta}\bigl(\tau(\theta)\bigr) \, \derivative*{2}[\tau]{\theta^2}(\theta) = 2, \\
        & p_i\left(\derivative*{1}[g]{\theta}\bigl(\tau(\theta)\bigr), \dots, \derivative*{i}[g]{\theta^i}(\tau(\theta)),
        \derivative*{1}[\tau]{\theta}(\theta), \dots, \derivative*{i-1}[\tau]{\theta^{i-1}}(\theta) \right) + \derivative*{1}[g]{\theta}\bigl(\tau(\theta)\bigr) \, \derivative*{i}[\tau]{\theta^i}(\theta) = 0, \qquad i = 3, \dots
    \end{align}
    \end{subequations}
    where $p_i$ are polynomials.
    Recalling that $\abs{\derivative*{1}[g]{\theta}(\theta)} \geq \lhess |\theta|$ for all $\theta \in \real$, we can therefore divide the equations in~\cref{eq:derivatives_of_tau} by $\derivative*{1}[g]{\theta}(\tau(\theta))$ in order to obtain expressions for the derivatives $\derivative*{i}[\tau]{\theta}(\theta)$
    which are valid when $\theta \neq 0$.
    From these expressions, it is then easy to obtain the desired bounds.
    For example,
    if we have already shown that~$\norm{\e^{-\mu_1 \theta^2} \derivative*{1}[\tau]{\theta}}[\infty] < \infty$, which follows from~\eqref{eq:derivatives_of_tau_1},
    then from~\eqref{eq:derivatives_of_tau_2} we obtain,
    using the fact that $\theta^2 = g(\tau(\theta)) \geq \frac{\ell}{2} |\tau(\theta)|^2$,
    \begin{align*}
        \abs{\derivative*{2}[\tau]{\theta^2}(\theta)}
        &\leq \frac{2 + \abs*{\derivative*{2}[g]{\theta^2}(\tau(\theta))} \, \abs*{\derivative*{1}[\tau]{\theta}(\theta)}^2}{\abs*{\derivative*{1}[g]{\theta}(\tau(\theta))}}
        \leq \frac{2 + C \e^{r_2 \abs{\tau(\theta)}^2} \, \e^{2\mu_1 \theta^2}}{\lhess \abs{\tau(\theta)}} \\
        &\leq \frac{2 + C \e^{\frac{2r_2}{\ell} \theta^2} \, \e^{2\mu_1 \theta^2}}{\lhess \abs{\tau(\theta)}}
        \leq C \e^{\left( \frac{2r_2}{\ell} + 2 \mu_1 \right) \theta^2} \qquad \text{if $\abs{\theta} \geq 1$},
    \end{align*}
    where $C$ is a constant changing from occurrence to occurrence.
    The last inequality is justified because $\max_{\abs{\theta} \geq 1} \abs{\tau(\theta)} > 0$.
    Since $\derivative*{2}[\tau]{\theta^2}$ is continuous and the set $\{\theta: \abs{\theta} \leq 1\}$ is compact,
    this shows the existence of $\mu_2 \in \real$ that $\norm*{\derivative*{2}[\tau]{\theta^2}(\theta) \, \e^{-\mu_2 \theta^2}}[\infty] < \infty$.
\end{proof}

\section*{Acknowledgements}
The authors are grateful to Zehua Lai for pointing out that the Poincaré inequality could be employed for proving~\cref{lemma:bound_second_moment}.
JAC was supported by the Advanced Grant Nonlocal-CPD (Nonlocal PDEs for Complex Particle Dynamics: 	Phase Transitions, Patterns and Synchronization) of the European Research Council Executive Agency (ERC) under the European Union's Horizon 2020 research and innovation programme (grant agreement No. 883363) and by EPSRC grant number EP/T022132/1.
JAC and UV were also supported by EPSRC grant number EP/P031587/1.
FH was funded by the
Deutsche Forschungsgemeinschaft (DFG, German Research Foundation) under
Germany's Excellence Strategy - GZ 2047/1, Projekt-ID 390685813.
AMS is supported by NSF (award AGS-1835860), by NSF (award DMS-1818977) and by the Office of Naval Research (award N00014-17-1-2079).
UV was also supported by the Fondation Sciences Mathématiques de Paris (FSMP),
through a postdoctoral fellowship in the ``mathematical interactions'' program.

\bibliographystyle{abbrv}
% \bibliography{references}

\begin{thebibliography}{10}

\bibitem{bergemann2010localization}
K.~Bergemann and S.~Reich.
\newblock A localization technique for ensemble kalman filters.
\newblock {\em Q. J. R. Meteorol. Soc.}, 136(648):701--707, 2010.

\bibitem{bergemann2012ensemble}
K.~Bergemann and S.~Reich.
\newblock An ensemble {K}alman--{B}ucy filter for continuous data assimilation.
\newblock {\em Meteorol. Z.}, 21(3):213, 2012.

\bibitem{MR1477662}
R.~Bhatia.
\newblock {\em Matrix analysis}, volume 169 of {\em Graduate Texts in
  Mathematics}.
\newblock Springer-Verlag, New York, 1997.

\bibitem{BCC}
F.~Bolley, J.~A. Ca\~{n}izo, and J.~A. Carrillo.
\newblock Stochastic mean-field limit: non-{L}ipschitz forces and swarming.
\newblock {\em Math. Models Methods Appl. Sci.}, 21(11):2179--2210, 2011.

\bibitem{BCC2}
F.~Bolley and J.~A. Carrillo.
\newblock Nonlinear diffusion: geodesic convexity is equivalent to
  {W}asserstein contraction.
\newblock {\em Comm. Partial Differential Equations}, 39(10):1860--1869, 2014.

\bibitem{borovykh2020stochastic}
A.~Borovykh, N.~Kantas, P.~Parpas, and G.~Pavliotis.
\newblock Stochastic mirror descent for fast distributed optimization and
  federated learning.
\newblock In {\em OPT2020: 12th Annual Workshop on Optimization for Machine
  Learning}, 2020.

\bibitem{borovykh2020interact}
A.~Borovykh, N.~Kantas, P.~Parpas, and G.~A. Pavliotis.
\newblock To interact or not? the convergence properties of interacting
  stochastic mirror descent.
\newblock In {\em International Conference on Machine Learning (ICML) Workshop
  on ‘Beyond First order methods in ML Systems}, 2020.

\bibitem{borovykh2021stochastic}
A.~Borovykh, N.~Kantas, P.~Parpas, and G.~A. Pavliotis.
\newblock On stochastic mirror descent with interacting particles: convergence
  properties and variance reduction.
\newblock {\em Physica D: Nonlinear Phenomena}, 418:132844, 2021.

\bibitem{brooks2011handbook}
S.~Brooks, A.~Gelman, G.~L. Jones, and X.-L. Meng, editors.
\newblock {\em Handbook of {M}arkov chain {M}onte {C}arlo}.
\newblock Chapman \& Hall/CRC Handbooks of Modern Statistical Methods. CRC
  Press, Boca Raton, FL, 2011.

\bibitem{bunch2016approximations}
P.~Bunch and S.~Godsill.
\newblock Approximations of the optimal importance density using {G}aussian
  particle flow importance sampling.
\newblock {\em J. Amer. Statist. Assoc.}, 111(514):748--762, 2016.

\bibitem{carrillo2018analytical}
J.~A. Carrillo, Y.-P. Choi, C.~Totzeck, and O.~Tse.
\newblock An analytical framework for consensus-based global optimization
  method.
\newblock {\em Math. Models Methods Appl. Sci.}, 28(6):1037--1066, 2018.

\bibitem{CFRT10}
J.~A. Carrillo, M.~Fornasier, J.~Rosado, and G.~Toscani.
\newblock Asymptotic flocking dynamics for the kinetic {C}ucker-{S}male model.
\newblock {\em SIAM J. Math. Anal.}, 42(1):218--236, 2010.

\bibitem{carrillo2010particle}
J.~A. Carrillo, M.~Fornasier, G.~Toscani, and F.~Vecil.
\newblock Particle, kinetic, and hydrodynamic models of swarming.
\newblock In {\em Mathematical modeling of collective behavior in
  socio-economic and life sciences}, Model. Simul. Sci. Eng. Technol., pages
  297--336. Birkh\"{a}user Boston, Boston, MA, 2010.

\bibitem{figures}
J.~A. Carrillo, F.~Hoffmann, A.~M. Stuart, and U.~Vaes.
\newblock Consensus based sampling: figshare media, 2021.

\bibitem{carrillo2019consensus}
J.~A. Carrillo, S.~Jin, L.~Li, and Y.~Zhu.
\newblock A consensus-based global optimization method for high dimensional
  machine learning problems.
\newblock {\em ESAIM Control Optim. Calc. Var.}, 27(suppl.):Paper No. S5, 22,
  2021.

\bibitem{carrillo2019wasserstein}
J.~A. Carrillo and U.~Vaes.
\newblock Wasserstein stability estimates for covariance-preconditioned
  fokker-planck equations.
\newblock {\em Nonlinearity}, 34:2275--2295, 2021.

\bibitem{Chen2012}
Y.~Chen and D.~S. Oliver.
\newblock Ensemble randomized maximum likelihood method as an iterative
  ensemble smoother.
\newblock {\em Math. Geosci.}, 44(1):1--26, Jan 2012.

\bibitem{cleary2020calibrate}
E.~Cleary, A.~Garbuno-Inigo, S.~Lan, T.~Schneider, and A.~M. Stuart.
\newblock Calibrate, emulate, sample.
\newblock {\em J. Comput. Phys.}, 424:109716, 20, 2021.

\bibitem{CS07}
F.~Cucker and S.~Smale.
\newblock On the mathematics of emergence.
\newblock {\em Jpn. J. Math.}, 2(1):197--227, 2007.

\bibitem{dashti2013map}
M.~Dashti, K.~J.~H. Law, A.~M. Stuart, and J.~Voss.
\newblock M{AP} estimators and their consistency in {B}ayesian nonparametric
  inverse problems.
\newblock {\em Inverse Problems}, 29(9):095017, 27, 2013.

\bibitem{del2006sequential}
P.~Del~Moral, A.~Doucet, and A.~Jasra.
\newblock Sequential {M}onte {C}arlo samplers.
\newblock {\em J. R. Stat. Soc. Ser. B Stat. Methodol.}, 68(3):411--436, 2006.

\bibitem{dorigo2005ant}
M.~Dorigo and C.~Blum.
\newblock Ant colony optimization theory: a survey.
\newblock {\em Theoret. Comput. Sci.}, 344(2-3):243--278, 2005.

\bibitem{duncan2021ensemble}
A.~B. Duncan, A.~M. Stuart, and M.-T. Wolfram.
\newblock Ensemble inference methods for models with noisy and expensive
  likelihoods.
\newblock {\em arXiv preprint arXiv:2104.03384}, 2021.

\bibitem{emerick2013investigation}
A.~A. Emerick and A.~C. Reynolds.
\newblock Investigation of the sampling performance of ensemble-based methods
  with a simple reservoir model.
\newblock {\em Comput. Geosci.}, 17(2):325--350, 2013.

\bibitem{engl1996regularization}
H.~W. Engl, M.~Hanke, and A.~Neubauer.
\newblock {\em Regularization of inverse problems}, volume 375 of {\em
  Mathematics and its Applications}.
\newblock Kluwer Academic Publishers Group, Dordrecht, 1996.

\bibitem{MR3400030}
O.~G. Ernst, B.~Sprungk, and H.~Starkloff.
\newblock Analysis of the ensemble and polynomial chaos {K}alman filters in
  {B}ayesian inverse problems.
\newblock {\em SIAM/ASA J. Uncertain. Quantif.}, 3(1):823--851, 2015.

\bibitem{fornasier2020consensus}
M.~Fornasier, H.~Huang, L.~Pareschi, and P.~S\"{u}nnen.
\newblock Consensus-based optimization on hypersurfaces: well-posedness and
  mean-field limit.
\newblock {\em Math. Models Methods Appl. Sci.}, 30(14):2725--2751, 2020.

\bibitem{fornasier2020consensusbased}
M.~Fornasier, H.~Huang, L.~Pareschi, and P.~Sünnen.
\newblock Consensus-based optimization on hypersurfaces: Well-posedness and
  mean-field limit.
\newblock {\em arXiv e-prints}, 2001.11994, 2020.

\bibitem{fornasier2021consensusbased}
M.~Fornasier, T.~Klock, and K.~Riedl.
\newblock Consensus-based optimization methods converge globally in mean-field
  law.
\newblock {\em arXiv e-prints}, 2103.15130, 2021.

\bibitem{garbuno2020interacting}
A.~Garbuno-Inigo, F.~Hoffmann, W.~Li, and A.~M. Stuart.
\newblock Interacting {L}angevin diffusions: gradient structure and ensemble
  {K}alman sampler.
\newblock {\em SIAM J. Appl. Dyn. Syst.}, 19(1):412--441, 2020.

\bibitem{garbuno2020affine}
A.~Garbuno-Inigo, N.~N\"usken, and S.~Reich.
\newblock Affine invariant interacting {L}angevin dynamics for {B}ayesian
  inference.
\newblock {\em SIAM Journal on Applied Dynamical Systems}, 19(3):1633--1658,
  2020.

\bibitem{goodman2010ensemble}
J.~Goodman and J.~Weare.
\newblock Ensemble samplers with affine invariance.
\newblock {\em Commun. Appl. Math. Comput. Sci.}, 5(1):65--80, 2010.

\bibitem{ha2020convergence}
S.-Y. Ha, S.~Jin, and D.~Kim.
\newblock Convergence of a first-order consensus-based global optimization
  algorithm.
\newblock {\em Math. Models Methods Appl. Sci.}, 30(12):2417--2444, 2020.

\bibitem{HL}
S.-Y. Ha and J.-G. Liu.
\newblock A simple proof of the {C}ucker-{S}male flocking dynamics and
  mean-field limit.
\newblock {\em Commun. Math. Sci.}, 7(2):297--325, 2009.

\bibitem{hastings1970monte}
W.~K. Hastings.
\newblock Monte {C}arlo sampling methods using {M}arkov chains and their
  applications.
\newblock {\em Biometrika}, 57(1):97--109, 1970.

\bibitem{HV18}
M.~Herty and G.~Visconti.
\newblock Kinetic methods for inverse problems.
\newblock {\em Kinet. Relat. Models}, 12(5):1109--1130, 2019.

\bibitem{iglesias2016regularizing}
M.~A. Iglesias.
\newblock A regularizing iterative ensemble {K}alman method for
  {PDE}-constrained inverse problems.
\newblock {\em Inverse Problems}, 32(2):025002, 45, 2016.

\bibitem{iglesias2013ensemble}
M.~A. Iglesias, K.~J.~H. Law, and A.~M. Stuart.
\newblock Ensemble {K}alman methods for inverse problems.
\newblock {\em Inverse Problems}, 29(4):045001, 20, 2013.

\bibitem{MR3041539}
M.~A. Iglesias, K.~J.~H. Law, and A.~M. Stuart.
\newblock Ensemble {K}alman methods for inverse problems.
\newblock {\em Inverse Problems}, 29(4):045001, 20, 2013.

\bibitem{Jabin-Wang}
P.-E. Jabin and Z.~Wang.
\newblock Mean field limit for stochastic particle systems.
\newblock In {\em Active particles. {V}ol. 1. {A}dvances in theory, models, and
  applications}, Model. Simul. Sci. Eng. Technol., pages 379--402.
  Birkh\"{a}user/Springer, Cham, 2017.

\bibitem{jin2018random}
S.~Jin, L.~Li, and J.-G. Liu.
\newblock Random batch methods ({RBM}) for interacting particle systems.
\newblock {\em J. Comput. Phys.}, 400:108877, 30, 2020.

\bibitem{kaipio2006statistical}
J.~Kaipio and E.~Somersalo.
\newblock {\em Statistical and computational inverse problems}, volume 160 of
  {\em Applied Mathematical Sciences}.
\newblock Springer-Verlag, New York, 2005.

\bibitem{kantas2019sharp}
N.~Kantas, P.~Parpas, and G.~A. Pavliotis.
\newblock The sharp, the flat and the shallow: Can weakly interacting agents
  learn to escape bad minima?
\newblock {\em arXiv preprint arXiv:1905.04121}, 2019.

\bibitem{kennedy2010particle}
J.~Kennedy.
\newblock Particle swarm optimization.
\newblock In {\em Encyclopedia of Machine Learning}, pages 760--766. Springer,
  2010.

\bibitem{Kovachki_2019}
N.~B. Kovachki and A.~M. Stuart.
\newblock Ensemble {K}alman inversion: a derivative-free technique for machine
  learning tasks.
\newblock {\em Inverse Problems}, 35(9):095005, 35, 2019.

\bibitem{MR3362507}
B.~Leimkuhler, C.~Matthews, and J.~Weare.
\newblock Ensemble preconditioning for markov chain monte carlo simulation.
\newblock {\em Statistics and Computing}, 28(2):277--290, 2018.

\bibitem{MR3509213}
T.~Leli{\`e}vre and G.~Stoltz.
\newblock Partial differential equations and stochastic methods in molecular
  dynamics.
\newblock {\em Acta Numer.}, 25:681--880, 2016.

\bibitem{liggett2012interacting}
T.~M. Liggett.
\newblock {\em Interacting particle systems}.
\newblock Classics in Mathematics. Springer-Verlag, Berlin, 2005.
\newblock Reprint of the 1985 original.

\bibitem{birthdeath}
Y.~Lu, J.~Lu, and J.~Nolen.
\newblock Accelerating langevin sampling with birth-death.
\newblock {\em arXiv e-prints}, 1905.09863, 2019.

\bibitem{lu2017gaussian}
Y.~Lu, A.~Stuart, and H.~Weber.
\newblock Gaussian approximations for probability measures on {$\mathbb R^d$}.
\newblock {\em SIAM/ASA J. Uncertain. Quantif.}, 5(1):1136--1165, 2017.

\bibitem{metropolis1953equation}
N.~Metropolis, A.~W. Rosenbluth, M.~N. Rosenbluth, A.~H. Teller, and E.~Teller.
\newblock Equation of state calculations by fast computing machines.
\newblock {\em J. Chem. Phys.}, 21(6):1087--1092, 1953.

\bibitem{MR2238098}
P.~D. Miller.
\newblock {\em Applied asymptotic analysis}, volume~75 of {\em Graduate Studies
  in Mathematics}.
\newblock American Mathematical Society, Providence, RI, 2006.

\bibitem{MT14}
S.~Motsch and E.~Tadmor.
\newblock Heterophilious dynamics enhances consensus.
\newblock {\em SIAM Rev.}, 56(4):577--621, 2014.

\bibitem{nusken2019note}
N.~{N{\"u}sken} and S.~{Reich}.
\newblock {Note on Interacting Langevin Diffusions: Gradient Structure and
  Ensemble Kalman Sampler by Garbuno-Inigo, Hoffmann, Li and Stuart}.
\newblock {\em arXiv e-prints}, 1908.10890, 2019.

\bibitem{2019arXiv190810890N}
N.~{N{\"u}sken} and S.~{Reich}.
\newblock {Note on Interacting Langevin Diffusions: Gradient Structure and
  Ensemble Kalman Sampler by Garbuno-Inigo, Hoffmann, Li and Stuart}.
\newblock {\em arXiv e-prints}, 1908.10890, 2019.

\bibitem{pavliotis2011applied}
G.~A. Pavliotis.
\newblock {\em Stochastic processes and applications}, volume~60 of {\em Texts
  in Applied Mathematics}.
\newblock Springer, New York, 2014.
\newblock Diffusion processes, the Fokker-Planck and Langevin equations.

\bibitem{pavliotis21derivative}
G.~A. {Pavliotis}, A.~M. {Stuart}, and U.~{Vaes}.
\newblock {Derivative-free Bayesian Inversion Using Multiscale Dynamics}.
\newblock {\em arXiv e-prints}, 2102.00540, Feb. 2021.

\bibitem{cookbook}
K.~B. Petersen and M.~S. Pedersen.
\newblock The matrix cookbook, Oct. 2008.
\newblock Version 20081110.

\bibitem{pidstrigach2021affine}
J.~Pidstrigach and S.~Reich.
\newblock Affine-invariant ensemble transform methods for logistic regression.
\newblock {\em arXiv preprint arXiv:2104.08061}, 2021.

\bibitem{pinnau2017consensus}
R.~Pinnau, C.~Totzeck, O.~Tse, and S.~Martin.
\newblock A consensus-based model for global optimization and its mean-field
  limit.
\newblock {\em Math. Models Methods Appl. Sci.}, 27(1):183--204, 2017.

\bibitem{reich2011dynamical}
S.~Reich.
\newblock A dynamical systems framework for intermittent data assimilation.
\newblock {\em BIT}, 51(1):235--249, 2011.

\bibitem{reich2015probabilistic}
S.~Reich and C.~Cotter.
\newblock {\em Probabilistic forecasting and {B}ayesian data assimilation}.
\newblock Cambridge University Press, New York, 2015.

\bibitem{reich2021fokker}
S.~Reich and S.~Weissmann.
\newblock Fokker--{P}lanck {P}article {S}ystems for {B}ayesian {I}nference:
  {C}omputational {A}pproaches.
\newblock {\em SIAM/ASA J. Uncertain. Quantif.}, 9(2):446--482, 2021.

\bibitem{schillings2017analysis}
C.~Schillings and A.~M. Stuart.
\newblock Analysis of the ensemble {K}alman filter for inverse problems.
\newblock {\em SIAM J. Numer. Anal.}, 55(3):1264--1290, 2017.

\bibitem{schillings2018convergence}
C.~Schillings and A.~M. Stuart.
\newblock Convergence analysis of ensemble {K}alman inversion: the linear,
  noisy case.
\newblock {\em Appl. Anal.}, 97(1):107--123, 2018.

\bibitem{shun1995laplace}
Z.~Shun and P.~McCullagh.
\newblock Laplace approximation of high-dimensional integrals.
\newblock {\em J. Roy. Statist. Soc. Ser. B}, 57(4):749--760, 1995.

\bibitem{swart2017course}
J.~M. Swart.
\newblock A course in interacting particle systems.
\newblock {\em arXiv e-prints}, 1703.10007, Mar. 2017.

\bibitem{Snitzman}
A.-S. Sznitman.
\newblock Topics in propagation of chaos.
\newblock In {\em \'{E}cole d'\'{E}t\'{e} de {P}robabilit\'{e}s de
  {S}aint-{F}lour {XIX}---1989}, volume 1464 of {\em Lecture Notes in Math.},
  pages 165--251. Springer, Berlin, 1991.

\bibitem{To06}
G.~Toscani.
\newblock Kinetic models of opinion formation.
\newblock {\em Commun. Math. Sci.}, 4(3):481--496, 2006.

\bibitem{van2000asymptotic}
A.~W. van~der Vaart.
\newblock {\em Asymptotic statistics}, volume~3 of {\em Cambridge Series in
  Statistical and Probabilistic Mathematics}.
\newblock Cambridge University Press, Cambridge, 1998.

\bibitem{van2019particle}
P.~J. Van~Leeuwen, H.~R. K{\"u}nsch, L.~Nerger, R.~Potthast, and S.~Reich.
\newblock Particle filters for high-dimensional geoscience applications: A
  review.
\newblock {\em Q. J. R. Meteorol. Soc.}, 145(723):2335--2365, 2019.

\bibitem{yang2013feedback}
T.~Yang, P.~G. Mehta, and S.~P. Meyn.
\newblock Feedback particle filter.
\newblock {\em IEEE Trans. Automat. Control}, 58(10):2465--2480, 2013.

\end{thebibliography}

\end{document}